\documentclass[11pt,a4paper]{article}
\usepackage{preambule}
\usepackage{nccmath}

\newcommand{\Pbf}{\mathbf{P}}
\newcommand{\Ebf}{\mathbf{E}}

\newcommand{\red}{\color{red}}

\newcommand{\rme}{\mathrm{e}}

\title{A confined random walk locally looks like tilted random interlacements}

\author{Nicolas Bouchot}

\AtEndDocument{\bigskip{\footnotesize%
		\textsc{CEREMADE, Universit\'e Paris Dauphine-PSL, UMR 7534, Place du Maréchal de Lattre de Tassigny, 75775 Paris Cedex 16, France.} \par  
		\textit{E-mail address}: \texttt{nicolas.bouchot@dauphine.psl.eu} }}

\begin{document}
	
	\pagestyle{plain}
	
	\maketitle

\begin{abstract}
In this paper we consider the simple random walk on $\ZZ^d$, $d \geq 3$, conditioned to stay forever in a large domain $D_N$ of typical diameter $N$. Considering the range up to time $t_N \geq N^{2+\delta}$ for some $\delta > 0$, we establish a coupling with what Li \& Sznitman \cite{liLowerBoundDisconnection2014} defined as \textit{tilted random interlacements}. This tilted interlacement can be described as random interlacements but with trajectories given by random walks on conductances $c_N(x,y) = \phi_N(x) \phi_N(y)$, where $\phi_N$ is the first eigenvector of the discrete Laplace--Beltrami operator on $D_N$. The coupling follows the methodology of the soft local times, introduced by Popov \& Teixeira \cite{popov2015soft} and used by {\v C}ern{\'y} \& Teixeira \cite{teixeiracoupling} to prove the well-known coupling between the simple random walk on the torus and the random interlacements.
\\[0.2cm]
\end{abstract}

	\section{Introduction}
	
	\subsection{Simple random walk on the torus and random interlacements}
	
	Consider the discrete torus of size $N \geq 1$, denoted by $\mathbb{T}_N^d = \ZZ^d / N \ZZ^d$, as well as the simple, nearest-neighbour random walk (SRW) $S=(S_n)_{n\geq 0}$ on $\mathbb{T}_N^d$, started from the uniform measure on~$\mathbb{T}_N^d$. For $d \geq 3$, the behaviour of this random walk has been extensively studied in the literature, with a notable focus on the local geometry of its range as $N \to +\infty$. Fixing $u > 0$, the range $\mathcal{R}(u N^d) \defeq \{S_0, \dots , S_{\lfloor u N^d \rfloor} \}$ is ``locally close'' to what Sznitman introduced in \cite{sznitman2010vacant} as random interlacements (RI) at level $u$, denoted by $\mathscr{I}(u)$. Informally speaking, $\mathscr{I}(u)$ is a random subset of $\ZZ^d$ that is the trace of a Poissonnian collection of independent random walk trajectories; we provide more details in Section \ref{ssec:entrelac}.
	
	In the previous decade, several works provided a quantitative and precise sense of this ``closeness'' through couplings of the range $\mathcal{R}(u N^d)$, considered as a subset of $\ZZ^d$, and the interlacements $\mathscr{I}(u)$, see \cite{belius2013gumbel, teixeira2011fragmentation} for example. One of the most significant results is due to \cite{teixeiracoupling}, with a coupling on macroscopic boxes inside $\mathbb{T}_N^d$; let us now state their result.
Let $\delta \in(0,1)$ be fixed and consider the macroscopic box $B = [0, (1-\delta)N]^d \subset \mathbb{T}_N^d$, for some $\delta>0$. 
Then, for any $u>0$, there is a non-increasing sequence $(\eps_N)_{N\geq 1}$ (that decreases as $N^{-c}$ for some~$c$ that depends on $\delta$) and a coupling $\mathcal{Q}$ of $\mathcal{R}(u N^d)$ and $\mathscr{I}((1\pm\eps_N)u)$ such that for, some $\eta > 0$, with $\mathcal{Q}$-probability at least $1- e^{-N^\eta}$, we have
	\begin{equation}\label{eq:couplage-entrelac-RW-tore}
		\mathscr{I}((1-\eps_N)u) \cap B \subseteq \mathcal{R}(u N^d) \cap B \subseteq \mathscr{I}((1+\eps_N)u) \cap B \, .
	\end{equation}
	
	This type of results were highly influencial in tackling existing questions for SRW on the torus. We can notably cite the covering time (see \cite{belius2013gumbel}) and the disconnection time for the cylinder (see \cite{SZN-cylindre}); we will comment further on these questions in Section~\ref{ssec:applications}.
	
	The core heuristical argument for these couplings is that after reaching some subset $A \subset \mathbb{T}_N^d$, the random walk $S$ will realize a set of excursions, until leaving $A$ for a time larger than the mixing time of the walk. The eventual return of $S$ in $A$ will result in a second set of excursions which is thus mostly independent of the previous set, and so on until time $u N^d$. The number of such sets of excursions is then well-approximated by a Poisson random variable. Seeing $A$ as a set in $\ZZ^d$, it can be shown that the trace $\mathscr{I}(u) \cap A$ of the RI on $A$, has exactly the law of this collection of sets of excursions.
	
	\par This core argument is not exactly specific to the SRW on the torus, and one may investigate any situation which involves the random walk visiting a large set for a long time, to see if the range exhibits this ``interlacing'' property.
	
	In this paper, we consider the random walk conditioned to stay in a large domain $D_N$ of $\ZZ^d$. After heuristical considerations involving this  ``interlacing'' property, we construct a coupling of the random walk range with a random interlacements that very much resembles $\mathscr{I}(u)$. This coupling is constructed using the same ideas as in~\cite{teixeiracoupling}, however a major difference lies in the fact that the walk conditioned to stay in $D_N$ behaves as a random walk in a time-dependent potential. This breaks the ``Markovianity'' of the walk and implies that we cannot consider standard random interlacements since they do not take the drift into account.
	
	In the next two sections, we explain how we can circumvent both of these issues. First, by considering an ``infinite'' time horizon for the walk we remove the time dependence of the potential. Second, instead of considering a Poissonnian collection of independent SRW trajectories for interlacements, we use walks on suitable conductances that coincide with the confining potential.
	
\subsection{Simple random walk conditioned to stay in a large domain, confined walk}\label{ssec:pseudo-h-transform}

	\subsubsection{Motivations and first notations}
	
	Let $S$ be the simple nearest-neighbour random walk on $\ZZ^d, d \geq 3$ and denote by $\Pbf_x$ its law starting at $x$ (with $\Pbf = \Pbf_0$). 
	 Fix a compact set $D \subset \RR^d$ which contains $0$ and which we assume to have a smooth boundary and a nonempty, connected interior. Let $N \geq 1$ and define $D_N \defeq (ND) \cap \ZZ^d$ the discrete blowup of $D$ with factor~$N$.
	
	In the rest of the paper, we fix $x_0 \in \mathrm{Int}(D)$ and we let $\alpha > 0$ and $\eps > 0$ small enough (see below Proposition \ref{prop:rester-ds-boule-anneau}) such that for all $N$ large enough, the discrete Euclidean ball of radius $(\alpha + 2\eps)N$ centered at $x_0^N = x_0 N$ is contained in~$D_N$. 
	%This is typically the case for convex $D$. 
	We define
	\begin{equation}\label{eq:def-B}
		B_N = B(x_0^N, \alpha N) = \mathset{y \in D_N \, | \, d(x_0^N,y) \leq \alpha N} \, ,
	\end{equation}
	where $d$ is the Euclidean distance. For $\delta \in (0,2\eps)$, we will write $B_N^\delta = B(x_0^N, (\alpha + \delta)N) \subseteq D_N$.

	\begin{figure}[h]\centering
		\includegraphics[width=9cm]{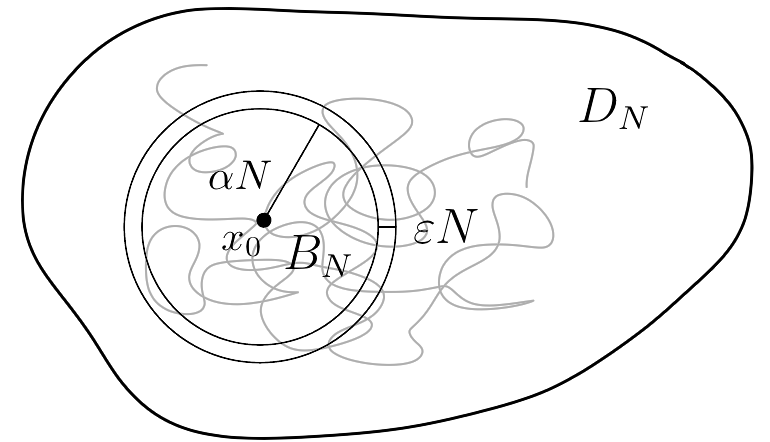}
		\caption{Representation of $D_N$ and $B_N$ with its $\eps N$-neighborhood, and a random walk path conditioned to stay in $D_N$.}
		\label{fig:B_N}
	\end{figure}
	
	Our motivation for this paper is to show that the range $\mathcal{R}(t_N)=\{S_0, \ldots, S_{t_N}\}$ of the SRW, under the conditioned probability $\Pbf( \cdot \, | \, \mathcal{R}(t_N) \subseteq D_N)$, can be coupled with a proper random interlacements at a well-chosen intensity $u_N = u_N(t_N)$ so that they coincide on $B_N$. 
	We will prove this under some condition on the sequence $(t_N)_{N\geq 0}$, but let us stress already that this will include cases where $u_N\to 0$, that is cases of RI with very low intensity.
	
	\begin{remark}
	\label{rem:shape}
		We restrict ourselves to get a coupling on a ball $B_N$ mostly for simplicity of the exposition --- some estimates for the random walk are simplified in this case. 
		
		However, our results remain valid if we replace $B_N$ with any set that is ``$s N$-regular'' for some $s > 0$, in the sense of \cite[Def.~8.1]{popov2015soft}: there is a $s > 0$ such that for any $N$ large enough, for any point $x \in \partial B_N$, there are balls $B^{\mathrm{in}} \subseteq B_N$, $B^{\mathrm{out}} \subseteq \ZZ^d \setminus B_N \cup \partial B_N$ of radius $s N$ that are both tangent to $B_N$ at point $x$. We will comment on this along the proofs.
		Note also that $B_N$ is macroscopic, and may be arbitrarily close to $D_N$, for example if $D$ is a ball, or if $D$ is well approximated by ``$s$-regular'' sets (meaning there is a subset $D^s$ of $D$, close to $D$ in Hausdorff distance, that is ``$s$-regular'' in the sense of tangent balls of radius $s$ explained previously).
	\end{remark}
	
	As previously explained, the finite time horizon of the conditioning poses substantial technical issues in its study by breaking the Markov property of the random walk. It should however be noted that if $t_N \gg N^2$, this induced ``defect'' in the Markov property is only relevant near time $t_N$. We therefore study a Markov chain that we argue properly approximates the conditioned walk up until time $(1-\eta)t_N$ for some $\eta \in (0,1)$. This process is the infinite time-horizon version of the conditioned walk, which we call the \emph{confined walk}.

	\subsubsection{The main object of study: the confined random walk}
	
	The main difference between the SRW on the torus and the SRW on $\ZZ^d$ conditioned to stay in $D_N$ is that the conditioning creates a drift that repels the walk away from the boundary $\partial D_N \defeq \{ y\notin D_N, \, \exists x\in D_N \text{ with } x\sim y \}$. In order to manage this space-time drift, we replace the random walk with another Markov chain that is morally the SRW conditioned to stay in $D_N$ \textit{forever}. Fix $N \geq 1$ and consider the substochastic matrix $P_N$ of the SRW killed when exiting $D_N$, given by
	\[ P_N(i,j) = \frac{1}{2d} \quad \text{if $i,j \in D_N$ and $i \sim j$} \, , \quad 0 \quad \text{else} \, . \]
	Write $\lambda_N$ and $\phi_N$ for the first eigenvalue and associated eigenvector of $P_N$, with the following normalization
	\begin{equation}
	\label{def:eigenfunction}
		P_N \phi_N = \lambda_N \phi_N \, , \quad \| \phi_N \|_2^2 \defeq \sum_{z \in D_N} \phi_N^2(z) = N^d \, .
	\end{equation}
We stress that the normalization is not fundamental to our result: the main requirement is that it implies that $\phi_N$ is of order one in the bulk of $D_N$, which is the case with our normalization (see \eqref{eq:phiN-borne} below). In the paper, we will write $\| \phi_N \|_2^2$ without replacing it by $N^d$ to keep this in mind.

Note that $\phi_N$ is defined on $D_N$, but it might be convenient to extend it to $\partial D_N$ (and $\mathbb{Z}^d$) by setting $\phi_N(x) =0$ for $x\notin D_N$. This first eigenvector will play a crucial role in the present paper, and we will be interested in its properties as a function $x \mapsto \phi_N(x)$. We provide a quick summary of some of these properties below and in Appendix \ref{sec:appendix:fonctions-propres}, summarizing results obtained in \cite{vecteurpropre}.

Let us introduce the following notation: for $h : \ZZ^d \longrightarrow \RR$  a real-valued function on $\ZZ^d$, we define
	\begin{equation}
	\label{def:barh-Deltad}
		\bar{h}(z) \defeq \frac{1}{2d} \sum_{|e| = 1} h(z+e) \, , \quad \Delta_d h(z) \defeq \bar{h}(z) - h(z)\, .
	\end{equation}
This way, one can rewrite~\eqref{def:eigenfunction} as the Dirichlet problem $\Delta_d \phi_N = (1-\lambda_N) \phi_N$, with  boundary condition $\phi_N\equiv 0$ on $\partial D_N$.

	We then define the transition kernel on $D_N \times D_N$, which is a Doob's transform of the initial SRW: 
	\begin{equation}\label{eq:def-p_N}
		p_N(x,y) \defeq \frac{\lambda_N^{-1}}{2d} \frac{\phi_N(y)}{\phi_N(x)} \indic{x \sim y} \, .
	\end{equation}
	We write $\Pbf^N_\mu$ for the law of the associated Markov chain with starting distribution $\mu$; we will also write $\Pbf^N_x$ when $\mu = \delta_x$ and $\Pbf^N_{\phi_N^2}$ when $\mu = c_N \phi_N^2(\cdot)$ with the correct normalizing constant~$c_N$.
	We will refer to this Markov chain as the \textit{confined random walk} (confined RW) and usually denote it by $X$.

	It is known, by standard Markov chain theory (see e.g. \cite[Appendix A.4.1]{lawlerRandomWalkModern2010}), that the transition kernel~\eqref{eq:def-p_N} is in some sense the limit, as $T \to +\infty$, of the transition kernels of the SRW conditioned to stay in $D_N$ until time $T$. Indeed, if $\tau_N$ is the first time the SRW exits $D_N$ and $x \sim y$, we have
	\begin{equation*}
		\frac{\Pbf_x(X_1 = y, \tau_N > T)}{\Pbf_x(\tau_N > T)} = \frac{1}{2d} \frac{\Pbf_y(\tau_N > T-1)}{\Pbf_x(\tau_N > T)}  = \frac{\lambda_N^{-1}}{2d} \frac{\lambda_N \Pbf_y(\tau_N > T-1)}{\Pbf_x(\tau_N > T)} \xrightarrow[T \to +\infty]{} \frac{\lambda_N^{-1}}{2d} \frac{\phi_N(y)}{\phi_N(x)} \, ,
	\end{equation*}
the last limit being a consequence of the fact that $\Pbf_z(\tau_N > T) = \phi_N(z) \frac{\| \phi_N \|_1}{N^{2d}} (\lambda_N)^T + \grdO((\beta_N)^T)$ as $T\to\infty$, for some $\beta_N < \lambda_N$ (see \cite[Prop.~6.9.1]{lawlerRandomWalkModern2010}).

\begin{remark}
Let us stress here that the transition kernel from~\eqref{eq:def-p_N} is that of a random walk among conductances $c_N(x,y) \defeq \phi_N(x)\phi_N(y)$, \textit{i.e.}\ we can rewrite $p_N(x,y)$ as 
\begin{equation}
\label{eq:kernel_conductances}
p_N(x,y) = \frac{c_N(x,y)}{\sum_{z\sim x} c_N(x,z)} = \frac{\phi_N(y)}{\sum_{z\sim x} \phi_N(z)} = \frac{1}{2d} \frac{\phi_N(y)}{\bar{\phi}_N(x)} , \quad \text{ for } x,y \in D_N, \, x\sim y\,.
\end{equation}
Indeed, since $\phi_N$ is an eigenfunction of $P_N$, we have that for any $x\in D_N$ 
\begin{equation}
\label{eq:phiharmonique}
\bar{\phi}_N(x)=\frac{1}{2d} \sum_{z\sim x} \phi_N(z) = P_N \phi_N(x) = \lambda_N \phi_N(x) \,.
\end{equation}
(Recall that we set $\phi_N(z)=0$ for $z\notin D_N$ so the sum is only on $z\in D_N$.)
Therefore, \eqref{eq:kernel_conductances} coincides with \eqref{eq:def-p_N}.
\end{remark}

	In the rest of the paper, we will focus on the confined RW 
%QB J'ai enlevé ça, il y a beaucoup  de répétitions et on parle des entrelacs juste après.
% and prove that its range is close to a \textit{tilted} random interlacements, that we describe below in Section~\ref{ssec:entrelac-tilte}.	
	since it will provide a good description of most of the range of SRW conditioned to stay in $D_N$. Indeed, under $\Pbf( \cdot \, | \, \mathcal{R}_{t_N} \subseteq D_N)$, the SRW before time $(1-\eta) t_N$ should ``look like'' the confined RW; put otherwise, it should not feel the fact that the time horizon of the conditioning is finite.

	%QB j'ai remplacé le paragraphe par une remarque: une fois qu'on a la remarque \paragraph{Occupation measure of the tilted random walk and inhomogeneous interlacements}
	
	\begin{remark}
	%The density of RI is directly tied with the occupation measure of the SRW. On the torus, the occupation measure is the uniform distribution, which is coherent with the translation-invariance of RI. On the other hand, t
	The confined random walk under $\Pbf^N$ has a stationary occupation measure given by~$\phi_N^2$, with proper normalisation. Therefore, the interlacements that we consider must exhibit a space-inhomogeneous density of trajectories, notably on $B$ where it should be given by $\phi_N^2$. Such interlacements are called \textit{tilted random interlacements} (tilted RI) and were introduced in \cite{teixeira2009interlacement,liLowerBoundDisconnection2014}. 
	They correspond to a Poisson point process of trajectories, with trajectories that are not SRW but rather a \textit{tilted} RW, which creates the spatial inhomogeneity (see Section \ref{ssec:entrelac-tilte} for details).
	\end{remark}
	
	\begin{comment}
		The core argument seems to hold in the case of our model, thus we can expect the local geometry to look like a random interlacement.
		
		There are however two major differences with our model. Firstly, the random walk on $\mathbb{T}_N^d$ naturally stays in the torus, while our model corresponds to conditionning the random walk to stay in $D_N$. Secondly, since the invariant measure of the walk on $\mathbb{T}_N^d$ is the uniform measure with mixing time $N^2$ (see \cite[Th. 5.5]{levin2017markov}, it is expected that the range up to time $N^d$ has the same geometry in all of the torus. However, in our case, we can expect the random walk to stay significantly more in the bulk of $D_N$ than close to the boundary: the occupation measure of the random walk is highly non uniform and should depend on the distance to the boundary of $D_N$.
		
		In the following, we explain how we can adapt the known statements to our model.
		
	\end{comment}

	\subsection{Some preliminaries on random interlacements}

	\subsubsection{Definition of random interlacements}\label{ssec:entrelac}

	Consider the space of doubly infinite transient paths on the $d$-dimensional lattice $\ZZ^d, d \geq 3$:
	\[ W = \Big\{ w : \ZZ \rightarrow \ZZ^d \, : \, \forall n \in \ZZ, |w(n) - w(n+1)|_1 = 1 \text{ and } \lim_{|n| \to \infty} |w(n)| = \infty \Big\} \, , \]
	endowed with the $\sigma$-algebra $\mathscr{W}$ generated by the maps $w \mapsto w(n), n \in \ZZ$.
	We can define the equivalence relation $w \sim w' \Longleftrightarrow \exists k \in \ZZ, w(\cdot + k) = w'$, and we write $W^* \defeq W/\sim$ the corresponding quotient space as well as $\pi^*$ the associated canonical projection. The set $W^*$ is endowed with the $\sigma$-algebra $\mathscr{W}^*$ generated by $\pi^*$.
	
	Let $K$ be a subset of $\ZZ^d$. We define for $w \in W$ the  hitting times
	\begin{equation}
	\label{def:HK}
	H_K(w) = \inf \mathset{k \in \ZZ \, : \, w(k) \in K}\,,
	\quad
	\bar{H}_K = \inf \mathset{k \geq 1 \, : \, w(k) \in K}
	\end{equation}
with $\inf \varnothing = +\infty$ by convention.
We also define $W_K = W \cap \mathset{H_K < +\infty}$ the set of trajectories that hit $K$. 
	
	Suppose that $K$ is finite. We define the equilibrium measure $e_K$ on $K$ and the harmonic measure $\bar{e}_K$ as
	\begin{equation}\label{eq:def-cap-mes-eq-RI}
		e_K(x) = \Pbf_x(\bar{H}_K = +\infty) \, , \quad \bar{e}_K(x) = \frac{e_K(x)}{\capN{K}} \,.
	\end{equation}
	Here,  $\capN{K} is$ the \textit{capacity} of the set $K$, given by
	\[ \capN{K} = e_K(K) = \sum_{x \in \partial K} e_K(x) \, ,\]
	where we have used that the measure $e_K$ is supported on the (inner) boundary of $K$, which we denote by $\partial K = \{ x \in K \, : \, \exists y \in \ZZ^d \setminus K, x \sim y \}$.
	
	Intuitively, $\bar{e}_K$ is the law of the first entrance point in $K$ of a random walk trajectory that is coming from far away. Indeed, we have $\Pbf_z( X_{H_K} = x \, | \, H_K < +\infty) \to \bar{e}_K(x)$ as $|z| \to +\infty$, and the capacity can be alternatively written as
	\[
	\capN{K} = \lim_{|z| \to +\infty} G(z)^{-1} \Pbf_z (H_K < +\infty) \, , \quad \text{with $G$ the SRW Green function} \, ,
	\] 
	and can be interpreted as the \og{}size\fg{} of $K$ seen from  a random walk starting at a faraway point on $\ZZ^d$. Recall that $G(z) \asymp |z|^{2-d}$, where $|\cdot|$ is the Euclidean norm on~$\RR^d$ (see \cite[Thm.~4.3.1]{lawlerRandomWalkModern2010}).
	
	Now, let $w^* \in W^*_K$ be a class of paths hitting $K$, and denote by $\tilde{w}$  the unique $w \in w^*$ such that $H_K(w) = 0$. For $K$ finite, we can define on $\pi^*(W_K)$ a finite measure $\nu_K$ given by
	\begin{equation}\label{eq:mesures-nu_K}
		\nu_K(w^*) = \Pbf_{\tilde{w}(0)}(w(\ZZ_-) \, | \, \bar{H}_K = +\infty) \times e_K(w(0)) \times \Pbf_{\tilde{w}(0)}(w(\ZZ_+)) \, ,
	\end{equation}
	which can be interpreted, after normalization, as the law of a SRW trajectory that hits~$K$. The measures $\nu_K$ for $K$ finite can be extended to a $\sigma$-finite measure $\nu$ on $W^*$.
	
	The random interlacements process is a Poisson Point Process $\chi$ on the space $W^* \times \RR_+$ with intensity $\nu \otimes du$. The random interlacement of level $u > 0$ is the subset $\mathscr{I}(u) \subset \ZZ^d$ defined as
	\begin{equation}
		\mathscr{I}(u) \defeq \big\{ z \in \ZZ^d \, : \, \exists (w,v) \in \chi, v \leq u, z \in w(\ZZ)\big\} =  \big\{w(k) \, : \, k \in \ZZ, (w,v) \in \chi, v \leq u \big\} \, .
	\end{equation}
	We denote by $\PP$ the law of the RI on the space $\mathcal{M}_p(W^* \times \RR_+)$ of point measures on $W^* \times \RR_+$, which we also use to denote the law of the (increasing) family of random subsets $(\mathscr{I}(u))_{u > 0}$.
	
	A key property of RI is that its trace in a finite set $K$ can be recovered from a collection of SRW trajectories --- this is a core component of the coupling. More precisely, let $N_K^u$ be a Poisson-distributed random variable with parameter $u \, \capN{K}$ and let $(X^{(i)})_{i \geq 1}$ be i.i.d.\ SRW trajectories with common law $\mathbf{P}_{\bar{e}_K}$, independent from $N_K^u$. Then we have the equality in law
	\begin{equation}\label{eq:entrelac-simul-traj}
		\mathscr{I}(u) \cap K \; \overset{(d)}{=} \; \bigcup_{i = 1}^{N_K^u} \mathcal{R}_\infty(X^{(i)}) \cap K \, .
	\end{equation}
	In particular, we deduce the following crucial identity, which characterizes the law of $\mathscr{I}(u)$ as a random subset of $\ZZ^d$ (see \cite[Remark 2.2-(2)]{sznitman2010vacant}): for any fixed $u > 0$ and any finite set $K \subset \ZZ^d$,
	\[ \PP\big( \mathscr{I}(u) \cap K  = \varnothing \big) = e^{- u \, \capN{K}} \, . \]

	\subsubsection{Tilting of random interlacements}\label{ssec:entrelac-tilte}
	
The tilted random interlacements were first introduced in \cite{teixeira2009interlacement} as a generalization of random interlacement on transient weighted graphs.
Later, \cite{liLowerBoundDisconnection2014} defined the tilting of continuous-time random interlacements as a way to locally modify the trajectories of the RI. 
It was used in particular to get large deviation principles for disconnection events, by locally densifying the RI. We refer to subsequent works of Sznitmann \cite{sznitmanDisconnectionRandomWalks2017,sznitmanMacroscopicHolesSupercritical2019a,sznitmanBulkDeviationsLocal2023} for the use of tilted interlacements to study large deviation events for interlacements such as disconnection, covering of sets or excess-type events, with some applications to simple random walk.

We choose to present here a reformulation of the first approach of tilted interlacements by \cite{teixeira2009interlacement}, with a touch of local tilting of interlacement present in \cite{liLowerBoundDisconnection2014}.
	
%	%QB: je crois qu'on n'a pas besoin de ça ici, non?
%	For $h : \ZZ^d \longrightarrow \RR$  a real-valued function on $\ZZ^d$, we define
%	\begin{equation}
%		\bar{h}(z) \defeq \frac{1}{2d} \sum_{|e| = 1} h(z+e) \quad , \quad \Delta_d h(z) \defeq \bar{h}(z) - h(z) \quad , \quad \| h \| \defeq \sum_{z \in \ZZ^d} |h(z)| \, .
%	\end{equation}
%
	
	Consider a positive function $h : \ZZ^d \longrightarrow \RR_+^*$ which satisfies $h = 1$ outside a finite set. Informally, the $h$-tilted random interlacements of level $u > 0$, denoted by $\mathscr{I}_h(u)$, is a Poisson cloud of $h$-tilted random walk trajectories, \textit{i.e.}\ random walks on $\mathbb{Z}^d$ equipped with conductances $c(x,y) = h(x)h(y)$.
	
	More precisely, we let $\Pbf^{h}_z$ denote the law of the random walk on conductances $c(x,y)= h(x)h(y)$ starting at $z \in \ZZ^d$. Consider a finite set $K \subset \ZZ^d$. Then, we can define the $h$-tilted equilibrium measure of $K$ by
	\begin{equation}\label{eq:def:titled-equilibrium-mes}
		e^h_K(z) \defeq \mathbf{P}^h_z \Big( \bar{H}_K = +\infty \Big) h(z) \bar{h}(z) \indic{z \in K}  \, .
	\end{equation}
	The measure $e^h_K$ has finite mass and is supported on $\partial K$ with total mass $\cpc^h(K)$, the tilted capacity of~$K$. We write $\bar{e}^h_K \defeq e^h_K / \cpc^h(K)$ for the corresponding (tilted harmonic) probability measure. We can then define the analogues $(\nu^h_K)$ of the measures $(\nu_K)$ for finite $K \subset \ZZ^d$, and their extension $\nu^h$ on $W^*$.
	
	The $h$-tilted RI is then defined as a Poisson cloud of trajectories that is given by a Poisson point process on $W^* \times \RR_+$ with intensity measure $u \otimes \nu^h$. It satisfies the following property, analogous to~\eqref{eq:entrelac-simul-traj}: for any finite set $K$, denote by $N^{h,u}_K$ a Poisson variable with parameter $u\, \cpc^h(K)$ and let $(X^{(i)})_{i\geq 1}$ be i.i.d.\ $h$-tilted RW trajectories with law $\Pbf^h_{\bar{e}_K^h}$. We then have
	\begin{equation}\label{eq:entrelac-tilt-simul}
		\mathscr{I}_h(u) \cap K \overset{(d)}{=} \bigg( \bigcup_{i = 1}^{N^{h,u}_K} \mathcal{R}_\infty(X^{i}) \bigg) \cap K \, .
	\end{equation}
	
\begin{remark}
		Contrary to \cite{teixeira2009interlacement}, the continuous-time tilted interlacements at level $u > 0$, introduced in \cite{liLowerBoundDisconnection2014}, was instead defined as an exponential tilting of the law of the continuous-time RI at level $u$. To do so, define the function $F:W^* \to \mathbb{R}$, as
\[ \forall w \in W^*, F(w^*) = \int_0^\infty \frac{\Delta_d h}{h}(\tilde{w}(s)) \, \dd s \, \, ,\]
where we recall that $\tilde{w}$ is defined just above \eqref{eq:mesures-nu_K}.

The continuous-time tilted RI is the Poisson point process on $W^*$ with intensity measure $e^F \nu$.
This tilted RI posess the interlacement property \eqref{eq:entrelac-tilt-simul} (see \cite[Remark 2.2]{liLowerBoundDisconnection2014}), with trajectories given by the Markov process with generator $L \varphi (x) = \frac{1}{2d} \sum_{z \sim x} \frac{h(z)}{h(x)} (\varphi(z) - \varphi(x))$. The range of these trajectories are the same as the $h$-tilted RW, hence our definition (which is sufficient for the present paper).
\end{remark}

\subsection{Main result}

Recall that $B_N^{\eps}$ is the $\eps$-enlarged version of $B_N$, see below~\eqref{eq:def-B}, and define the tilting functions
	\begin{equation*}
		\Psi_N(x) \defeq \begin{dcases}
			\phi_N(x) & \quad \text{ if }\ x \in B_N^\eps \,, \\
			1 & \quad \text{ else} \,,
		\end{dcases} 
	\end{equation*}
	as well as the law $\Pbf^{\Psi_N}_x$ of the random walk on conductances $\Psi_N(i)\Psi_N(j)$ starting at $x \in \ZZ^d$, \textit{i.e.}\ with transition kernel
	\begin{equation}\label{eq:kernel-RW-tilt}
		p_{\psi_N}(x,y) \defeq \frac{1}{2d} \frac{\Psi_N(y)}{\bar{\Psi}_N(x)} \, \indic{x\sim y}\, ,
	\end{equation}
	recall~\eqref{eq:kernel_conductances}. We will call this process the \emph{tilted random walk} (tilted RW).
	
	Note that on $B^\eps$ we have $\Psi_N = \phi_N$, 
%QB sert à rien:		and if $z \in D_N \setminus \partial D_N$, we have $\bar{\phi}_N(z) = \lambda_N \phi_N(z)$. 
thus, when considering a RW path which stays inside~$B^\eps$, this path has the same probability under $\Pbf^{\psi_N}$ as under $\Pbf^{\phi_N} = \Pbf^N$. In other words, on $B^\eps$ the tilted RW has the same law as the confined RW.

	Let us mention that it is known (see \cite{pjm/1103040107}, or Appendix \ref{sec:appendix:fonctions-propres} for details) that $\lambda_N$ satisfies
	%QB : déplacé la 2e partie de l'équation à plus tard
	\begin{equation}\label{eq:encadrement-lambda}
		\lambda_N = 1 - \frac{\lambda}{2d} \frac{1}{N^2} (1+\bar{o}(1)) \, , 
	\end{equation}
	where $\lambda=\lambda_D$ is the first Dirichlet eigenvalue of the Laplace-Beltrami operator on $D$ and $\bar{o}(1)$ is a quantity that vanishes as $N \to +\infty$.
	In particular, there exists $c_0 > 0$ a universal constant, which can be made arbitrarily close to $\lambda/(2d)$ by taking $N$ large enough, such that for any $T \geq 0$, 
	\begin{equation}
	\label{eq:encadrement-lambda2}
		1 \leq \lambda_N^{-T} \leq \rme^{c_0 T/N^2}\,.
	\end{equation}
	
	Denote by $\mathcal{R}_{\phi_N}(t_N)$ the range up to time $t_N \geq 1$ of the Markov chain with law $\Pbf^N_{\phi_N^2}$ (recall that $\Pbf^N_{\phi_N^2}$ is the confined RW starting from its invariant measure $\mu= \phi_N^2 N^{-d}$), and recall the definition~\eqref{eq:def-B} of $B_N$.
	The main result of this paper describes how $\mathcal{R}_{\phi_N}(t_N)$ is locally close to the $\Psi_N$-tilted interlacement $\mathscr{I}_{\Psi_N}$ at an appropriate level.

\begin{theorem}\label{th:couplage-CRW-entrelac-t_N}
		Let $\delta \in (0,1)$ and consider a sequence $(t_N)_{N\geq 1}$ that satisfies $ t_N/ N^{2+\delta} \to +\infty$. We define
\[
u_N\defeq \frac{t_N }{\lambda_N \| \phi_N \|_2^2} = \frac{t_N }{\lambda_N N^d} \,, \qquad \text{where we recall}\quad  \| \phi_N \|_2^2 = \sum_{x\in D_N} \phi_N^2(x) = N^d \, , \]
%the $L^2$ norm of $\phi_N$.
as well as $\eps_N \defeq N^{-\delta/4} \to 0$.
Then, there are some $\eta > 0$ and some constants $c_1,c_2>0$ (that only depend on $\alpha, \delta, \eps, D$ and $d \geq 3$) and a coupling $\QQ$ of $\mathcal{R}_{\phi_N}(t_N)$ and $\mathscr{I}_{\Psi_N}((1 \pm \eps_N)u_N)$ such that, for all~$N$ large enough we have
		\begin{equation}
			\mathscr{I}_{\Psi_N}((1-\eps_N)u_N) \cap B_N \, \subseteq\, \mathcal{R}_{\phi_N}(t_N) \cap B_N \,\subseteq \, \mathscr{I}_{\Psi_N}((1+\eps_N)u_N) \cap B_N \,,
		\end{equation}
		with $\QQ$-probability at least $ 1 - c_1 e^{-c_2 N^{\eta}}$.
		\end{theorem}

Note that taking $t_N = \lfloor u N^d \rfloor$, we get $u_N \to u > 0$, thus this is consistant with the random walk on the torus (recall \eqref{eq:couplage-entrelac-RW-tore}). It should be noted that with our choice of $t_N$, the number of interlacement trajectories hitting $B_N$ is of order $u_N N^{d-2} \asymp t_N/N^2 \to +\infty$.

\begin{remark}
Coming back to the original question of the random walk conditioned to stay in $D_N$, we can expect the same result as in Theorem~\ref{th:couplage-CRW-entrelac-t_N}. 
The challenge mostly resides in the technicalities needed to properly compare the range of the confined RW and the range of the SRW condioned to stay in $D_N$. 
With the same method as for \cite[Lemma~3.10]{dingDistributionRandomWalk2021a}, one can show that for any $\eta > 0$, for any $t \in \llbracket \eta t_N, (1-\eta) t_N \rrbracket$, the law of $S_t$ under $\mathbf{P}(\cdot \mid \mathcal{R}_{t_N} \subseteq D_N)$ is close to $\phi_N^2 / \| \phi_N\|_2^2$, that is the invariant measure of the confined RW. 
However we would need to get a coupling of the excursions inside $B_N$ --- this seems feasible, at least if $B_N$ has a sufficiently small (\textit{i.e.}\ $\bar{o}(N)$) radius.
\end{remark}

\begin{remark}
One could wonder how our choice of taking $\Psi_N \equiv 1$ on $\ZZ^d \setminus B_N^\eps$ affects our result. One could easily think of other choices (ones that may be more optimal in some sense), but from what one can gather from the proofs, changing the function $\psi_N$ outside from $B_N^{\eps}$ would only affect marginally the tilted RI on $B_N$ --- some constants should be changed here and there in the proof.
The reason is the macroscopic ``buffer'' zone $B_N^\eps \setminus B_N$ that allows the tilted RW to forget about the exterior before returning to $B_N$.
One can mention that the optimal $\Psi_N$ to take is the one which maximizes $\Pbf^{\Psi_N}_x \big( H_\Delta < \bar{H}_B < +\infty) / \Pbf^N_x \big( H_\Delta < \bar{H}_B)$.
\end{remark}

\begin{remark}
	Since $\lambda_N = 1 - \grdO(N^{-2})$, its contribution is absorbed by $\eps_N$ and thus it can be omitted without changing the theorem. Still, we state the result with the ``correct'' level for $\mathscr{I}_{\Psi_N}$, as $\lambda_N$ may become relevant when considering mesoscopic balls instead of $B_N$ (in which case $\eps_N$ can be lowered).
\end{remark}

	\begin{comment}
	As explained previously, the use of the tilted interlacement is primarly to take into account the drift that results from the conditioning, which is expressed in terms of $\phi_N$. This drift is particularly present when studying excursions in a macroscopic box. If instead one considers a microscopic box, we can expect a weaker drift on the excursions. The two following results consider that $B$ is replaced by a microscopic box $B_\theta = B(x_0,N^\theta$) with $\theta < 1$.
	
	\begin{theorem}
		Coupling $\hat{\Psi}_N$-$\Psi_N$
	\end{theorem}
	
	For microscopic boxes, the drift can be expressed using only one function.
	
	\begin{theorem}
		Coupling $\hat{\Psi}_N$-tilted and $\mathscr{I}(\phi(x_0/N))$.
	\end{theorem}
	
	For even smaller microscopic boxes, the drift is sufficiently weak to replace the tilted interlacement with the random interlacement.
	
	\subsection{Approximating the conditioning}\label{sec:condition-pseudo-h-transform}
	
\end{comment}

	\subsection{Connection with other works}
	
	The confined and tilted walks appelations also appear in a paper by Li \cite{liLowerBoundDisconnection2017a} studying disconnection events for simple random walk. The confined walk essentially designates the same process. The tilted walk is slightly different: in Li's paper the drift is only active up to a deterministic time.
	It can be noted that Li proved that the confined walk locally dominates (in stochastic sense) random interlacements.

	Chiarini \& Nietzchner \cite{chiariniLowerBoundsBulk2023} also studied excess type events for functionals of the SRW, which also involves tilted interlacements to find a lower bound on the probability of such events. The strategy to achieve these events is to behave like a random walk in a drift, much like the drift that we use for the tilted walk. Their proof used a series of coupling on mesocopic boxes (of size $N^\lambda, \lambda \in (0,1)$) between the tilted walk, simple random walk and standard random interlacements.
	
	The coupling relies on the regularity of the drift function to control the effect of the drift at these mesocopic scales. In our case, we where able to prove the regularity of $\phi_N$ (see Appendix \ref{sec:appendix:fonctions-propres}) which indicates that at such scales, we may replace the tilted RI by the standard RI. It seems natural to expect that the level $u_N$ would now depend on the average value of $\phi_N$ on this mesocopic box. Note that the regularity of $\phi_N$ is offscaled by the number of excursions of $\mathscr{I}_{\phi_N}(u_N)$ inside $B_N$, which grows as $u_N N^{d-2} \asymp t_N N^{-2} \to +\infty$, therefore the proof of \cite{chiariniLowerBoundsBulk2023} does not apply in this context.

	\subsection{Possible applications of the coupling}\label{ssec:applications}
	
	Theorem \ref{th:couplage-CRW-entrelac-t_N} being the analogue of \eqref{eq:couplage-entrelac-RW-tore} for the random walk conditioned to stay in a domain, it is natural to expect that the results concerning the random walk on the torus have an analogue for the random walk conditioned to stay in a domain. To simplify the exposition, we focus on the case where $D$ is the unit Euclidean ball in $\RR^d$ and $B_N = B(0, (1- \eps)N)$. Theorem \ref{th:couplage-CRW-entrelac-t_N} gives us a coupling of the range $\mathcal{R}_{\phi_N}(t_N)$ of the confined RW on the ball $B_N$ with a tilted RI. Let us give two examples of problems that have be tackled for the random walk on the torus using random interlacements, and what should be the corresponding ``conditioned'' results, that should rely on a better understanding of tilted RI.
	
\paragraph{Disconnection of sets}
One can expect that, similarly to the RI, the tilted RI $\mathscr{I}_{\Psi_N}$ to undergo a phase transition regarding percolation on subsets of $D_N$. 
Thanks to Theorem~\ref{th:couplage-CRW-entrelac-t_N}, this should translate for the confined RW as a phase transition for the probability of an inner set $K_N \subseteq B_N$ being disconnected from $\partial B_N^\eps$ (in the same flavor as RI, see \cite{disconnection2010}). 
We may also adapt the entropic bound used in \cite{liLowerBoundDisconnection2014} to get large deviation results in the subcritical regime.
	
\paragraph{Disconnection of cylinders}
	The disconnection time of $\mathbb{T}_N^d \times \ZZ$ by the random walk was notably studied in \cite{sznitman2009disconnection-cylinder}, obtaining that the disconnection occurs around time $N^{2d}$. In our setting, consider the random walk conditioned to stay in the cylinder $D_N \times \ZZ$, as well as the central cylinder $B_N\times \ZZ$ of radius $(1-\eps)N$ centered at $\mathset{0} \times \ZZ$.
	Write $\mathcal{V}_N^\eps(t) = (\ZZ \times B_N) \setminus \mathcal{R}(t)$ the vacant set in $B_N\times \ZZ$, and define the disconnection time
	\begin{equation}
		\mathfrak{D}_N \defeq \inf \mathset{t \geq 0 \, : \, \mathcal{V}_N^\eps(t) \text{ is not connected}} \, .
	\end{equation}
	According to Theorem \ref{th:couplage-CRW-entrelac-t_N} we can expect $\mathfrak{D}_N$ to also be of order $N^{2d}$, with a possiblity of obtaining tightness of the laws of $\mathfrak{D}_N/N^{2d}$.
	
	\paragraph{Covering time}
	Still considering $B_N \subset D_N$, we may look at the covering time of $B_N$, meaning
	\begin{equation}
		\mathfrak{C}_N \defeq \inf \mathset{t \geq 0 \, : \, B_N \subseteq \mathcal{R}(t)} = \sup \mathset{H_x \, : \, x \in B_N} \, .
	\end{equation}
	For the SRW on the torus, the covering time is asymptotically $g(0) N^d \log N$, see~\cite{aldous1983time}, and \cite{belius2013gumbel} proved that the \og{}last points\fg{} visited by the SRW are distributed as a Poisson point process with intensity proportional to the Lebesgue measure. This means that close to time $\mathfrak{C}_N$, the SRW will from time to time collect one of these points, in no particular order. We also refer to \cite{prévost2023phase} for an extensive study of those \og{}late points\fg{}. In our case, one can prove that $\mathfrak{C}_N \sim g(0) \alpha N^d \log |\Lambda_N|$ in probability as $N \to +\infty$, with $\alpha \defeq \lim_{N \to +\infty} \inf_{x \in B_N} \phi_N^2(x)$. In the case where $D$ is the unit ball and $B_N = B(0,(1-2\eps)N)$, one can also prove that the \og{}last points\fg{} distribute as a Poisson point process on the sphere.

	\subsection{Organization of the paper}

%QB : à développer un peu plus tard.

	We first explain in Section \ref{ssec:soft-local-times} the soft local times method, developed in~\cite{popov2015soft,teixeiracoupling}.
	It will be used to obtain a coupling of the entrance sites in~$B$ of excursions of the confined RW and tilted RI. This will actually be sufficient to create the coupling for the entire range.
	The key result is Theorem~\ref{th:couplage-CdM-diff-pi} below (borrowed from~\cite{teixeiracoupling}), that requires us to estimate three quantities: the invariant measure $\tilde{\pi}_\circ$ of entrance points in $B$ for both the confined RW and tilted RI (Section \ref{ssec:mes-invariante-1}), the mixing times $T^Y$ and $T^Z$ (Section \ref{ssec:tps-melange-1}) of the underlying Markov chains of entrance points, and the variances $\mathrm{Var}_{\pi} \rho^\circ_z$ (Section \ref{ssec:variance-1}) of some specific functions $\rho^\circ_z$ appearing in Theorems~\ref{th:soft-local-times}-\ref{th:couplage-CdM-diff-pi} below.
	In Section \ref{ssec:couplage-1}, we regroup the results to exhibit the coupling of Theorem \ref{th:couplage-CRW-entrelac-t_N}.

	\section{Some preliminary results}

	\subsection{The soft local time coupling method}\label{ssec:soft-local-times}
	
	The soft local time method was introduced in \cite{popov2015soft} as a way to decorrelate distant regions of the random interlacements. It was later expanded in \cite{teixeiracoupling} to couple the ranges of Markov processes on a finite state space $\Sigma$ which have the same invariant measure. In this section, we briefly summarize the method and the main coupling theorem of \cite{teixeiracoupling}, that will be the main tool for the coupling of the walk with the tilted interlacements.
	
	%QB Un peu modifié ce paragraphe : il vaut mieux à mon avis décrire le processus de simulation à partir du processus de Poisson et dire comment on en tire un couplage.
	The key idea of is to simulate the range up to time $n \geq 1$ of a Markov chain with finite state space $\Sigma$ by using a Poisson point process $\eta$ on $\Sigma \times \RR_+$.
	The coupling then simply uses the same Poisson point process to simulate the ranges up to time $n$ of two different Markov chains.
	The construction of the range is done by choosing the right intensity measure on $\Sigma \times \RR_+$ and then  exploring the points of $\eta$ in the direction $\RR_+$, while collecting the corresponding elements of $\Sigma$ (the ``explored sites''). 
	 This exploration is constructed by means of the transition kernel of the Markov chain at each step, and it correctly builds up the range of the Markov chain up to time~$n$. At this time $n$, the construction has explored a part of $\Sigma \times \RR_+$ delimited (on the right/from above) by a ``line'' $G_n : \Sigma \longrightarrow \RR_+$, which is called the soft local time at step $n \geq 1$. 
	 The range is thus constituted by the first component of points $(z,v) \in \eta(\Sigma \times \RR_+)$ such that $v \leq G_n(z)$.

	 To summarize, to each Markov chain on $\Sigma$ one can associate soft local times $(G_n)_{n\geq 0}$, and the range up to time~$n$ is simply the set $\{z \in \Sigma, (z,v)\in \eta, G_n(z) \leq v\}$ of ``explored sites''. Using the same Poisson process $\eta$ to generate ranges of two Markov chains therefore produces a coupling of the two ranges; having inequalities on soft local times then amounts to having inclusion relations for the ranges.
	 Let us now introduce the necessary elements to be able to state Theorem~3.2 in~\cite{teixeiracoupling}, which controls the error in the coupling.
	
	Let $Y$ and $Z$ be two Markov chains with values in a finite space $\Sigma$, with invariant measure $\pi^Y$ and $\pi^Z$ respectively. Let us also consider a measure $\mu$ with full support in $\Sigma$ (that can be considered as a parameter one can play with). Writing ``$\circ$'' for either $Y$ or $Z$, we define
	\begin{equation}\label{eq:def-g-et-rho}
		g^\circ(z) \defeq \frac{\pi^\circ(z)}{\mu(z)} \quad , \quad \rho_z^\circ(w) \defeq \frac{p^\circ(w,z)}{\mu(z)}
	\end{equation}
	where $p^\circ$ is the transition kernel of the corresponding Markov chain.
	Define the mixing times $T^\circ_{\mix}$ as
	\begin{equation}\label{eq:def-temps-melange}
		T^\circ_{\mix} \defeq \inf \Big\{ t \geq 0 \, : \, \sup_{x \in \Sigma} \| p_t^\circ(x, \cdot) - \pi^\circ \|_{TV} \leq \tfrac14 \Big\} \, ,
	\end{equation}
where $\| \cdot \|_{TV}$ is the total variation distance, given by $\| \lambda - \nu \|_{TV} = \tfrac12 \sum_{z \in \Sigma} |\lambda(z) - \nu(z)|$.
For $\eps>0$, let us define for any $n \geq 1$ the following event on the ranges of $Y$ and $Z$:
	\begin{equation}\label{eq:softloc-event}
		\mathcal{G}_{n,\eps} \defeq \mathset{\mathset{Z_i}_{i = 1}^{(1-\eps)n} \subseteq \mathset{Y_i}_{i = 1}^{n} \subseteq \mathset{Z_i}_{i = 1}^{(1+\eps)n}} \, .
	\end{equation}
Let us also introduce
	\begin{equation}\label{eq:softloc-keps}
		k(\eps) \defeq - \inf_{\circ = Y, Z} \inf_{z \in \Sigma} \log_2 \frac{\pi^\circ_* \eps^2 g^\circ(z)^2}{6 \mathrm{Var}_{\pi^\circ} (\rho^\circ_z)} \, ,
	\end{equation}
	with $\pi^\circ_* \defeq \min_{z \in \Sigma} \pi^\circ(z)$, and the following condition on $\eps>0$:
	\begin{equation}\label{eq:softloc-epsilon}
		0 < \eps < \inf_{\circ = Y, Z} \inf_{z \in \Sigma} \frac{\mathrm{Var}_{\pi^\circ} (\rho^\circ_z)}{2 \| \rho^\circ_z \|_\infty g^\circ(z)} \, .
	\end{equation}
	
	\begin{theorem}[{\cite[Theorem 3.2]{teixeiracoupling}}]
	\label{th:soft-local-times}
Let $Y,Z$ be two Markov chains on a finite state space~$\Sigma$, with transition matrices $P^Y,P^Z$, starting distribution $\nu_Y,\nu_Z$, and with the same invariant probability measure $\pi$.
Then there are constants $C,c>0$ and there exists a (common) probability space $(\Omega, \mathscr{F},\QQ)$
such that, for any $\eps > 0$ that satisfies \eqref{eq:softloc-epsilon}, for any $n \geq 2k(\eps) \min\{T^Y_{\mix} , T^Z_{\mix}\}$, we have
		\begin{equation}\label{eq:borne-soft-local-times}
			1-\QQ \left( \mathcal{G}_{n,\eps} \right) \leq C \sum_{\circ = Y, Z} \sum_{z \in \Sigma} \left( e^{-cn\eps^2} + e^{-c n \eps \frac{\pi^\circ(z)}{\nu_\circ(z)}} + \exp \left( - c \frac{\eps^2 g^\circ(z)^2}{\mathrm{Var}_{\pi^\circ} (\rho^\circ_z)} \frac{n}{k(\eps) T_{\mix}^\circ} \right) \right) \, .
		\end{equation}
	\end{theorem}

	In this paper, we will couple Markov chains with close, but different invariant measures. With a similar proof, we can easily get the following coupling theorem.
	
	\begin{theorem}\label{th:couplage-CdM-diff-pi}
		Consider $\eps > 0$ that satisfies \eqref{eq:softloc-epsilon} and assume that the invariant measures of $Y$ and $Z$ satisfy the following: 
		\begin{equation}\label{eq:hyp-erreur-pi}
			\text{there is a $\delta \in (0, \tfrac12 \eps)$ such that} \quad \sup_{z \in \Sigma} \bigg| \frac{\pi^Y(z)}{\pi^Z(z)} - 1 \bigg| < \delta \, .
		\end{equation}
		Then, there is a coupling of $Y$ and $Z$ such that for any $n \geq 2k(\eps) \min\{T^Y_{\mix} , T^Z_{\mix}\}$, we have
		\begin{equation}\label{eq:borne-soft-local-times-diff}
			1-\QQ \left( \mathcal{G}_{n,\eps} \right) \leq C \sum_{\circ = Y, Z} \sum_{z \in \Sigma} \left( e^{-cn\eps^2} + e^{-c n \eps \frac{\pi^\circ(z)}{\nu_\circ(z)}} + \exp \left( - c \frac{\eps^2 g^\circ(z)^2}{\mathrm{Var}_{\pi^\circ} (\rho^\circ_z)} \frac{n}{k(\eps) T_{\mix}^\circ} \right) \right) \, .
		\end{equation}
	\end{theorem}

	\begin{proof}
		Recall that, using the soft local times method, the event $\mathcal{G}_{n,\eps}$ is equivalent to having $G_{(1 - \eps) n}^Z(z) \leq G_n^Y(z) \leq G_{(1 + \eps) n}^Z(z)$.
		Following the proof of \cite[Theorem 3.1]{teixeiracoupling}, we get a coupling such that
		\begin{equation}
			(1 - \tfrac14 \eps) n g^Y(z) \leq G_n^Y(z) \leq (1 + \tfrac14 \eps) n g^Y(z)
		\end{equation}
		as well as
		\begin{equation}
			 (1-\tfrac14 \eps) (1 \pm \eps)n g^Z(z) \leq G_{(1 \pm \eps) n}^Z(z) \leq (1+\tfrac14 \eps) (1 \pm \eps) n g^Z(z)
		\end{equation}
		both hold with high probability (that is one minus the right-hand side of \eqref{eq:borne-soft-local-times}).
		Then, \eqref{eq:hyp-erreur-pi} implies that uniformly in $z \in \Sigma$, we have $(1-\delta) g^Y(z) \leq g^Z(z) \leq (1+\delta) g^Y(z)$. In particular, using \eqref{eq:hyp-erreur-pi}, under the same coupling we have
		\begin{equation}\label{eq:SLT-Z-leq-gY}
				G_{(1 - \eps) n}^Z(z) \leq (1+\tfrac14 \eps) (1 - \eps) n g^Z(z) \leq (1+\tfrac14 \eps) (1 - \eps) n (1+\delta) g^Y(z) \, .
		\end{equation}
		We then observe that provided $\delta < \tfrac12 \eps$, the right-hand side of \eqref{eq:SLT-Z-leq-gY} is less than $(1 - \tfrac14 \eps) n g^Y(z)$, and therefore we have $G_{(1 - \eps) n}^Z(z) \leq G_n^Y(z)$. In similar fashion, we also get $G_n^Y(z) \leq G_{(1 + \eps) n}^Z(z)$ and thus the full result.
	\end{proof}

	%QB changé un peu le titre
	\subsection{Confined random walk and eigenvectors estimates}

	Our global strategy for the proofs in this paper is to get back to SRW estimates, which are easier to obtain than for the confined RW, \textit{i.e.}\ under $\Pbf^N$. 
	We will mostly reduce to estimating probabilities of events that are localized in space and time: we are usually dealing with events that occur before exiting a specific finite domain in $D_N$. 
	
	Consider an event $A$ that depends on the trajectory up to $\tau_C := \min\{n, X_n \notin C \}$, the exit time of some subdomain $C \subset D_N$. 
	Then recalling the transition kernel $p_{N}(x,y)$ of $\Pbf^N$ from~\eqref{eq:def-p_N}, after telescoping the product of ratios of $\phi_N$'s, we can write
	\begin{equation}\label{eq:proba-htransform-RW-tuee}
		\Pbf^N_x (A) = \frac{1}{\phi_N(x)} \espRW{x}{\phi_N(X_{\tau_C}) \lambda_N^{-\tau_C} \mathbbm{1}_A} \, .
	\end{equation}
	Then, we usually want to show that the ratio of the $\phi_N$'s or the $\lambda_N^{-\tau_C}$ only contribute through multiplicative constants, meaning that\footnote{We use the notation $a_N \asymp b_N$ to denote the existence of positive constants $c, C$ that satisfy the following: there is a $N_0$ such that for all $N \geq N_0$, $c a_N \leq b_N \leq C a_N$.}
$\Pbf^N_x (A) \asymp \Pbf_x (A)$. To do so, we need a proper control of the ratios of $\phi_N$'s: this appears natural and should be known, but we were not able to find a reference in the literature. We will use the following result, that is a transcription of \cite[Corollary 1.12]{vecteurpropre}.

	%QB faudrait changer tes "proposal" en "proposition", non? (même dans le tex je veux dire)
	\begin{proposition}\label{prop:encadrement-ratio}
	There is a positive constant $\kappa_1$, and some $N_1 \geq 1$ such that, for all $N\geq N_1$,
		\begin{equation}
			\kappa_1 \leq \inf_{x,y \in B_N^\eps} \frac{\phi_N(x)}{\phi_N(y)} \leq \sup_{x,y \in B_N^\eps} \frac{\phi_N(x)}{\phi_N(y)} \leq \frac{1}{\kappa_1} \, .
		\end{equation}
	\end{proposition}
	
	An important consequence of Proposition \ref{prop:encadrement-ratio} is that the values of $\phi_N$ inside $B_N^\eps$ are bounded away from $0$ and from $+\infty$ by constants that are independent from $N$. More precisely, \cite[Corollary 1.11]{vecteurpropre} states that there is a constant $\kappa_2 > 0$, and some $N_2 \geq 1$ such that, for all $N\geq N_2$,
	\begin{equation}\label{eq:phiN-borne}
		\kappa_2 \leq \inf_{x \in B_N^\eps} \phi_N(x) \leq \sup_{x \in B_N^\eps} \phi_N(x) \leq \frac{1}{\kappa_2} \, .
	\end{equation}
	
	Coming back to~\eqref{eq:proba-htransform-RW-tuee}, Proposition~\ref{prop:encadrement-ratio} already proves that when $C \subseteq B^\eps$, the ratio of $\phi_N$'s can be controlled by universal constants. On the other hand, since $C \subseteq D_N$, we will have $\tau_C \asymp N^2$ and the factor $(\lambda_N)^{-\tau_C}$ should also contribute up to a multiplicative constant, recalling~\eqref{eq:encadrement-lambda}, but this is slightly technical and needs to be dealt with properly.
	
	Let $C$ be a subset of $B_N^\eps$, define $\mathcal{F}_{\tau,C}=\{A\in \mathcal{F} \colon A \cap \{\tau_C \le t\} \in \mathcal{F}^X_t\}$ the natural filtration stopped at $\tau_C$. In the proofs, we will repeatedly use the following bounds: for $x \in C$ and any event $A \in \mathcal{F}_{\tau,C}$, we have
	\begin{equation}\label{eq:encadrement-proba}
		 \kappa_1 \Pbf_x(A) \leq \Pbf^N_x (A) \leq \frac{1}{\kappa_1} \left[ \rme^{c_0} \Pbf_x(A) + \sum_{k = 1}^{+\infty} \rme^{c_0 (k+1)} \Pbf_x(A,\tau_C \in [k,k+1) N^2) \right]  \, ,
	\end{equation}
	where $\kappa_1 > 0$ is the constant of Proposition~\ref{prop:encadrement-ratio} and $c_0$ from~\eqref{eq:encadrement-lambda}-\eqref{eq:encadrement-lambda2} (recall that $c_0$ can be chosen arbitrarily close to $\lambda/2d$, provided that $N$ is chosen large enough). 
	The challenge in the proofs will be to properly decorrelate the events $A$ and $\mathset{\tau_C \geq k N^2}$ so that the sum in \eqref{eq:encadrement-proba} has at most the same order as $\Pbf_x(A)$.

	Since in the following, $C$ will generally be chosen to be a ball or an annulus, we recall the following results that will be used to control the sum in \eqref{eq:encadrement-proba}. 
	Its proof, which builds upon  classical results on random walks (that we collect from~\cite{lawlerRandomWalkModern2010}), is given in Appendix~\ref{sec:appendix:fonctions-propres} --- it can easily be adapted to more general domains.
	For $R > r > 0$ and $x \in \ZZ^d$, we will denote $A_r^R(x)$ the annulus $\{z \in \ZZ^d \, : \, r \leq |z-x| \leq R \} = (B(x,R) \setminus B(x,r)) \cup \partial B(x,r)$, and $A_r^R \defeq A_r^R(0)$.

	\begin{proposition}\label{prop:rester-ds-boule-anneau}
		There are positive constants $c_b,c_a>0$ and $C > 0$ such that, for all $T \geq 1$, for all sequences $a_N > b_N \to +\infty$ such that $\limsup_{N \to +\infty} \tfrac{a_N}{b_N} < +\infty$, provided $N$ large enough we have
		\[ \sup_{z \in B_{a_N}} \probaRW{z}{S_{[0,T]} \subseteq B_{a_N}} \leq C e^{- c_b \frac{T}{a_N^2}} \quad , \quad \sup_{z \in {B_{a_N} \setminus B_{b_N}}} \probaRW{z}{S_{[0,T]} \subseteq A_{b_N}^{a_N}} \leq C e^{- c_a \frac{T}{(a_N - b_N)^2}} \, . \]
	\end{proposition}

	We fix $\eps > 0$ that appears in the beginning of Section \ref{ssec:pseudo-h-transform} to have $c_a > \eps^2 c_0$. This choice may seem strange, since we introduced $\eps$ to create a buffer zone between $\partial D_N$ which would imply that larger $\eps$ is somewhat easier than small $\eps$. While it is true in this regard, simple RW estimates are easier for small $\eps$, hence our choice.
	
	\begin{proof}
		The first inequality is a consequence of \cite[Corollary 6.9.5 \& 6.9.6]{lawlerRandomWalkModern2010}. For the second inequality, we split $[0,T]$ into intervals of length $(a_N - b_N)^2$. Then, for all $z \in A_{b_N}^{a_N}$,
		\begin{equation}\label{eq:prop:rester-ds-anneau-boule:decoup}
			\begin{split}
				\probaRW{z}{S_{[0,T]} \subseteq A_{b_N}^{a_N}} &\leq \probaRW{z}{\forall i \in \mathset{1 , \dots , \lfloor \tfrac{T}{(a_N - b_N)^2} \rfloor}, S_{i(a_N - b_N)^2} \in A_{b_N}^{a_N}}\\
				&\leq \Big( \sup_{z \in {B_{a_N} \setminus B_{b_N}}} \probaRW{z}{S_{(a_N - b_N)^2} \in A_{b_N}^{a_N}} \Big)^{\lfloor \tfrac{T}{(a_N - b_N)^2} \rfloor}
			\end{split}
		\end{equation}
		Now, according to the local limit theorem \cite[Theorem 2.1.3]{lawlerRandomWalkModern2010}, the probability on the last line of \eqref{eq:prop:rester-ds-anneau-boule:decoup} is bounded from above by a constant $c \in (0,1)$ that only depends on the dimension. By the same theorem, this probability is also bounded from below by some $c' > 0$ that only depends on the dimension. Writing $c = e^{-c_a}$ for some $c_a > 0$ and taking $C = \tfrac{1}{c'}$ completes the proof.
	\end{proof}
	
%{\blue QB: il y a un truc qui me perturbe: si $r=1$, il me semble que c'est pas vrai que $c_a>c_b$. Il faut empêcher que $R/r$ ne tende vers $+\infty$, non? Sinon à la limite on n'a pas vraiment un anneau. Bien vérifier l'énoncé.}	

%{\blue 	Recalling that $c_0$ appearing in~\eqref{eq:encadrement-lambda2} can be taken arbitrarily close to $\lambda$, we will assume in the following that 
	%\begin{equation}		\label{eq:ineqconstantes}			c_0 < c_b <c_a	\end{equation}
%and that $N,R,r$ are chosen large enough so that~\eqref{eq:encadrement-lambda2} and Proposition~\ref{prop:rester-ds-boule-anneau} hold.}

	\subsection{Some estimates on $\Psi_N$-tilted interlacements}\label{ssec:propriete-entrelac-phiN}

	We provide in this section a few useful estimates, in particular on the tilted harmonic measure and the tilted capacities appearing in the definition of the tilted RI. To lighten the notation, we write $B$ for $B_N$.
	Recall the definition \eqref{eq:def:titled-equilibrium-mes} of the tilted equilibrium measure and capacity, as well as the fact that $\Psi_N \equiv \phi_N$ on $B^{\eps}$. For the $\Psi_N$-tilted RI, we have for any $K \subseteq B$ and any $x \in \partial K$, since $\bar{\Psi}_N(x) = \lambda_N \Psi_N(x)$ (recall \eqref{eq:phiharmonique}),
	\begin{equation}\label{eq:mes-harm-tilt-B}
		e_K^{\Psi_N}(x) = \Pbf^{\Psi_N}_x (\bar{H}_K = +\infty) \Psi_N(x) \bar{\Psi}_N(x)  =  \lambda_N \Psi_N^2(x) \Pbf^{\Psi_N}_x (\bar{H}_K = +\infty) \, ,
	\end{equation}
	with total mass $\cpc^{\Psi_N}(K)$. The tilted harmonic measure is then $\bar{e}_K^{\Psi_N} = e_K^{\Psi_N} / \cpc^{\Psi_N}(K)$. A useful identity is the so-called \og{}last exit decomposition\fg{} that we prove in the following lemma.
We introduce the Green function of the tilted RW:	
	\[
	G^{\Psi_N}(x,y)= \sum_{n \geq 0}  \Pbf^{\Psi_N}_x(X_n = y) \,.
	\]
Let us stress that this Green function is not symmetric in $x$ and $y$.

	\begin{lemma}
		Let $K \subseteq B^\eps$. We have for any $x \in B^\eps$,
		\begin{equation}\label{eq:LED-tilted}
			\frac{1}{\phi_N^2(x)} \sum_{y \in K} e_K^{\Psi_N}(y) \, G^{\Psi_N}(y,x) = \sum_{y \in K} \frac{1}{\phi_N^2(y)} e_K^{\Psi_N}(y) \,G^{\Psi_N}(x,y) = \lambda_N \Pbf^{\Psi_N}_x(H_K < +\infty) \, .
		\end{equation}
	\end{lemma}
	
	\begin{proof}
		Note that the first equality is simply given by the fact that the reversible measure is given by $\Psi_N^2 = \phi_N^2$ on~$B$. For the second inequality, consider $x \in K$ and define $L_K:=\max\{n\geq 0, X_n\in K\}$ the last time the tilted RW is in $K$. Since the $\Psi_N$-tilted RW is transient, we have $L_K < +\infty$ $\Pbf^{\Psi_N}_x$-a.s. 
		Therefore,
		\[ \begin{split} 
		\Pbf^{\Psi_N}_x(H_K < +\infty) &= \sum_{n \geq 0} \sum_{y \in K} \Pbf^{\Psi_N}_x(L_K = n, X_n = y) \\
		& = \sum_{n \geq 0} \sum_{y \in K} \Pbf^{\Psi_N}_x(X_n = y) \Pbf^{\Psi_N}_y(\bar{H}_K = +\infty) \\
			&= \sum_{y \in K} G^{\Psi_N}(x,y) \Pbf^{\Psi_N}_y(\bar{H}_K = +\infty) \, .
		\end{split} \]
		By \eqref{eq:mes-harm-tilt-B}, we have $\Pbf^{\Psi_N}_y(\bar{H}_K = +\infty) = e_K^{\Psi_N}(y) / \lambda_N \phi_N^2(y)$ for all $y \in K \subseteq B$, which proves the lemma.
	\end{proof}
	
	In the following, we also need some estimates on the \textit{tilted} equilibrium measure and capacity of $B$, comparing them with their simple random walk analogues; the results use the fact that $B$ is a ball (at distance $\eps N$ from $\partial D_N$), but could easily be extended to ``smooth'' domains whose capacity is comparable to that of a ball.

\begin{lemma}\label{lem:encadrement-mes-harm-tilt}
There is a constant $c > 0$ such that for all $x \in \partial B$
		\begin{equation*}
			c\, N^{-(d-1)} \leq \bar{e}^{\Psi_N}_B(x) \leq c^{-1} N^{-(d-1)} \, ,
			\end{equation*} 
meaning that the tilted harmonic measure $\bar{e}^{\Psi_N}_B(x)$ is roughly uniform on $\partial B$.
Moreover, we have
\begin{equation}\label{eq:encadrement-cap-tilt}
c N^{d-2} \leq \cpc^{\Psi_N}(B) \leq c^{-1} N^{d-2} \, .
\end{equation}
\end{lemma}

In the course of the proof of Lemma \ref{lem:encadrement-mes-harm-tilt}, we will use estimates on the probability for the confined/tilted RW to go from $B$ to $\partial B^\eps$ and vice-versa. The following lemma shows that these probabilities are of the same order as for the SRW.

\begin{lemma}\label{lem:arg-mart-tilt-B-Beps}
	There is a constant $C > 0$ such that for any $N$ large enough, any $x \in \partial B$, any $y \in \partial B^\eps$,
	\begin{equation}
		\frac{1}{CN} \leq \Pbf^{\Psi_N}_x(\bar{H}_B > H_{\partial B_\eps}) \leq \frac{C}{N} \quad , \quad \frac{1}{CN} \leq \Pbf^{\Psi_N}_y(H_B < \bar{H}_{\partial B^\eps}) \leq \frac{C}{N} \, .
	\end{equation}
\end{lemma}

Let us first proceed with the proof of Lemma \ref{lem:encadrement-mes-harm-tilt}.

	\begin{proof}[Proof of Lemma \ref{lem:encadrement-mes-harm-tilt}]
	Note that, by \eqref{eq:mes-harm-tilt-B}, we have
		\[ \bar{e}^{\Psi_N}_B(x) = \phi_N^2(x) \Pbf^{\Psi_N}_x (\bar{H}_B = +\infty) \Big / \sum_{z \in \partial B} \phi_N^2(z) \Pbf^{\Psi_N}_z (\bar{H}_B = +\infty) \, . \]
Our main goal is to prove that there exist constants $c,C > 0$, that do not depend on $N$ or $x \in \partial B$, such that 
		\begin{equation}\label{aprouver}
				\frac{c}{N} \leq \Pbf^{\Psi_N}_x (\bar{H}_B = +\infty) \leq \frac{C}{N} \,.
		\end{equation}
Then, thanks to Proposition \ref{prop:encadrement-ratio}, this will give that $\frac{c\kappa_2^2}{C |\partial B|} \leq \bar{e}^{\Psi_N}_B(x) \leq \frac{C}{c\kappa_2^2 |\partial B|}$ and conclude the estimates on the harmonic measure.
As far as the tilted capacity is concerned, we immediately get from~\eqref{aprouver}, recalling also~\eqref{eq:mes-harm-tilt-B}, that
		\[ \cpc^{\Psi_N}(B) = \sum_{x \in \partial B} e_B^{\Psi_N} (x) = \sum_{x \in \partial B} \phi_N^2(x) \Pbf^{\Psi_N}_x (\bar{H}_B = +\infty) \asymp \frac{1}{N} \sum_{x \in \partial B} \phi_N^2(x) \, . \]
Since $\phi_N(x) \in [\kappa_2, \kappa_2^{-1}]$ (recall \eqref{eq:phiN-borne}), we get the desired bounds on $\cpc^{\Psi_N}(B)$. Note that \eqref{eq:encadrement-cap-tilt} can also be deduced from Rayleigh's monotonicity law (see \cite[Theorem 9.12]{levin2017markov}) combined with \eqref{eq:phiN-borne}.
	
	\smallskip	
	It therefore only remains to prove~\eqref{aprouver}. Our proof consists in showing that $\Pbf^{\Psi_N}_x (\bar{H}_B = +\infty)$ is of the same order as $\Pbf^{\Psi_N}_x (\bar{H}_B > H_{\partial B^\eps})$, and then conclude using Lemma \ref{lem:arg-mart-tilt-B-Beps}.
	We rely on the use of the Markov property at time $H_{\partial B^\eps}$ ,
%	and a lower bound on $\Pbf^{\Psi_N}_{z}(H_B = +\infty)$ for $z \in \partial B^\eps$, 
which gives us the following starting point:
	\begin{equation*}\label{eq:lem:encadrement-mes-harm-tilt:encadrement-proba}
		\Pbf^{\Psi_N}_x (\bar{H}_B > H_{\partial B^{\eps}}) \inf_{z \in \partial B^{\eps}} \Pbf^{\Psi_N}_z (\bar{H}_B = +\infty) \leq \Pbf^{\Psi_N}_x (\bar{H}_B = +\infty) \leq \Pbf^{\Psi_N}_x (\bar{H}_B > H_{\partial B^\eps}) \, .
	\end{equation*}
Then, thanks to Lemma \ref{lem:arg-mart-tilt-B-Beps}, we immediately get that
	\begin{equation}
		\frac{1}{C N} \inf_{z \in \partial B^{\eps}} \Pbf^{\Psi_N}_z (\bar{H}_B = +\infty) \leq \Pbf^{\Psi_N}_x (\bar{H}_B = +\infty) \leq \frac{C}{N} \, .
	\end{equation}
		
Hence, we now only need to show that there is a constant $c > 0$ such that, uniformly in~$N$, $\inf_{z \in \partial B^{\eps}} \Pbf^{\Psi_N}_z (\bar{H}_B = +\infty) \geq c$. Fix $z \in \partial B^\eps$, and consider a tilted RW trajectory starting at $z$. This trajectory may cross $\partial B^\eps$ several times before either going to $B$ or leaving to infinity. If the ratio of the corresponding probabilities of these ``final choice'' is not too close to $0$ or $1$, the two probabilities are bounded from both sides by constants in $(0,1)$.

We first rewrite $\Pbf^{\Psi_N}_z (\bar{H}_B = +\infty)$ as
\begin{equation}\label{eq:eviter-B-ratio}
	\Pbf^{\Psi_N}_z (\bar{H}_B = +\infty) = \frac{\Pbf^{\Psi_N}_z (\bar{H}_B = +\infty)}{\Pbf^{\Psi_N}_z (\bar{H}_B < +\infty) + \Pbf^{\Psi_N}_z (\bar{H}_B = +\infty)} = \left[ 1 + \frac{\Pbf^{\Psi_N}_z (\bar{H}_B < +\infty)}{\Pbf^{\Psi_N}_z (\bar{H}_B = +\infty)} \right]^{-1} \, .
\end{equation}

Let us denote
\[
p_1 \defeq \sup_{z \in \partial B^{\eps}} \Pbf^{\Psi_N}_z(H_B < \tau_1  < +\infty) \, , \qquad q_1 \defeq \inf_{z\in \partial B^{\eps}} \Pbf^{\Psi_N}_z\big( \tau_1=+\infty \big)  \, .
\]
Then, by decomposing on the time before going to $B$ and using the Markov property, we have
\begin{equation}\label{eq:dernier-choix:toucher-B}
	\Pbf^{\Psi_N}_z (\bar{H}_B < +\infty) \leq \sum_{k \geq 0} \Pbf^{\Psi_N}_z(X_k \in \partial B^\eps, k < H_B) p_1 \, .
\end{equation}
In the same way, by decomposing on the last time the tilted RW visits $B^\eps$ and using the Markov property,
\begin{equation}\label{eq:dernier-choix:partir}
	\Pbf^{\Psi_N}_z (\bar{H}_B = +\infty) \geq \sum_{k \geq 0} \Pbf^{\Psi_N}_z(X_k \in \partial B^\eps, k < H_B) q_1 \, .
\end{equation}
Combining \eqref{eq:dernier-choix:toucher-B} and \eqref{eq:dernier-choix:partir} with \eqref{eq:eviter-B-ratio}, we deduce
\begin{equation}
	\inf_{z \in \partial B^{\eps}} \Pbf^{\Psi_N}_z (\bar{H}_B = +\infty) \geq \frac{1}{1 + \frac{p_1}{q_1}} \, .
\end{equation}

It therefore remains to estimate $p_1$ and $q_1$: we prove below that $p_1\leq c_1/N$ while $q_1\geq c_2/N$, which gives the desired bound.

\smallskip
\noindent
{\it (i) Probability $p_1$.}
According to Lemma \ref{lem:arg-mart-tilt-B-Beps}, an excursion that enters $B^\eps$ has a probability at most $\tfrac{c_1}{N}$ of hitting~$B$, regardless of its starting point. Hence, we have that $p_1 \leq c_1/N$.

	%Starting from some $z \in \partial B^\eps$, we use the Markov property to consider independently each excursion in $B^\eps$ that may hit $B$, which are $\Pbf^{\Psi_N}_z$-a.s. in finite number (before escaping from $B^\eps$). Then, taking $z \in \partial B^\eps$, $\Pbf^{\Psi_N}_z (H_B < H_{B^\eps})$ is the probability for one of these excursions starting from $z$ of not hitting $B$. According to Lemma \ref{lem:arg-mart-tilt-B-Beps}, Thus, we fix some $z \in \partial B^\eps$ and evaluate the number of said excursions.
		
		\smallskip
		\noindent
		{\it (ii) Probability $q_1$. }
		Consider $w \in \partial B^\eps$ as well as the slightly larger ball $B_\eps' \defeq B(x_0^N, (\alpha+\eps)N + 4)$. Then, we have the lower bound
		\begin{equation}\label{eq:proba-eloigner-Beps}
			\Pbf^{\Psi_N}_w(H_{B^\eps \setminus \partial B^\eps} = +\infty) \geq \Pbf^{\Psi_N}_w(\bar{H}_{B^\eps} > H_{\partial B_\eps'}) \inf_{v \in \partial B_\eps'} \Pbf^{\Psi_N}_v(\bar{H}_{B_\eps'} = +\infty) \, .
		\end{equation}
First observe that since $\forall z \in \ZZ^d, \Psi_N(z) \in [\kappa_2 \wedge 1, 1 \vee \tfrac{1}{\kappa_2}]$ (see \eqref{eq:phiN-borne}), we have that $\Pbf^{\Psi_N}_w(\bar{H}_{B^\eps} > H_{B_\eps'}) \geq c$ for some constant $c > 0$ (it is enough to take $4$ straight steps from $B^{\eps}$ to~$B_{\eps}'$).
		
		On the other hand, since $\Psi_N \equiv 1$ on $\ZZ^d \setminus B^\eps$, we have $\bar{\Psi}_N \equiv 1$ on $\partial B_\eps' \cup (\ZZ^d \setminus B_\eps')$. Thus,
		\begin{equation}\label{eq:proba-escape-Beps}
			\inf_{v \in \partial B_\eps'} \Pbf^{\Psi_N}_v(\bar{H}_{B_\eps'} = +\infty) = \inf_{v \in \partial B_\eps'} \Pbf_v(\bar{H}_{B_\eps'} = +\infty) \geq \frac{c'}{N} \, ,
		\end{equation}
		where we used \cite[Proposition 6.4.2]{lawlerRandomWalkModern2010} for the second inequality.
		Therefore, combining \eqref{eq:proba-eloigner-Beps} and \eqref{eq:proba-escape-Beps}, there is a constant $c_2 > 0$ that does not depend on $N \geq N_0$ or $w \in \partial B^\eps$ such that $\Pbf^{\Psi_N}_w (\bar{H}_{B^\eps} = +\infty) \geq c_2 / N$. In conclusion, $q_1\geq c_2/N$.

%{\it (ii) Sharp estimate on $\Pbf^{\Psi_N}_x (\bar{H}_B > H_{\partial B^\eps})$.}	
		%For the lower bound, by \eqref{eq:encadrement-proba} applied to $C = B^\eps \setminus B$ and $A = \mathset{\bar{H}_B > H_{\partial B^\eps}}$, we have
		%\[ \Pbf^{\Psi_N}_x (\bar{H}_B > H_{\partial B^\eps}) = \Pbf^{N}_x (\bar{H}_B > H_{\partial B^\eps}) \geq \kappa \probaRW{x}{\bar{H}_B > H_{\partial B^\eps}} \geq c N^{-1} \]
		%where we used Lemma \ref{lem:arg-mart-tilt-B-Beps} for the last inequality. which after inserting in \eqref{eq:lem:encadrement-mes-harm-tilt:encadrement-proba} proves the bounds on $\bar{e}_B^T(x)$. 
	\end{proof}

	\begin{proof}[Proof of Lemma \ref{lem:arg-mart-tilt-B-Beps}]
A key result in our proof is the following ``gambler's ruin'' lemma, which can be extracted from Proposition~1.5.10 in \cite{lawler2013intersections} (we refer to Appendix \ref{appendix:gambler-ruin} for details on how to get this statement).

\begin{lemma}\label{lem:gambler-ruin}
		Let $z \in (B^\eps \setminus B) \cup \partial B$ and $w \in B^\eps \setminus B$ be such that there are $\iota, \jmath \in [0,1)$ and $\eta > 0$ such that
		\[ |z - x_0^N| - \alpha N \in [\eta, \tfrac{1}{\eta}] N^\iota \quad , \quad |w - x_0^N| - (\alpha + \eps) N \in -[\eta, \tfrac{1}{\eta}] N^\jmath \, . \]
		Then, there exist explicit constants that depend only on $\alpha, \eps, d$ such that for $N$ large enough,
		\begin{equation}\label{eq:proba-atteindre-avant-SRW}
			\begin{split}
				c_1 \tfrac{\eta}{2} N^{\iota-1} + \grdO(N^{-1}) \leq \, &\Pbf_z(H_{B^\eps} < \bar{H}_B) \leq c_2 \tfrac{2}{\eta} N^{\iota-1} + \grdO(N^{-1}) \, , \\
				c_1' \tfrac{\eta}{2} N^{\jmath-1}  + \grdO(N^{-1}) \leq \, &\Pbf_w(H_B < \bar{H}_{\partial B^\eps}) \leq c_2' \tfrac{2}{\eta} N^{\jmath-1}  + \grdO(N^{-1}) \, .
			\end{split}
		\end{equation}
	\end{lemma}	
	
Then, for the lower bounds in Lemma~\ref{lem:arg-mart-tilt-B-Beps}, we use  \eqref{eq:encadrement-proba} to get
		\[ \Pbf^{\Psi_N}_x(\bar{H}_B > H_{\partial B_\eps}) \geq \kappa_1 \Pbf_x(\bar{H}_B > H_{\partial B_\eps}) \, , \qquad \Pbf^{\Psi_N}_y(H_B < \bar{H}_{\partial B^\eps}) \geq \kappa_1 \Pbf_y(H_B < \bar{H}_{\partial B^\eps}) \, , \]
then we simply apply~\eqref{eq:proba-atteindre-avant-SRW} for $\iota = 0$ and $\jmath = 0$.
%		To control these probabilities, we will use Proposition 1.5.10 in \cite{lawler2013intersections}, that can be reformulated in our case as the following statement: let $z \in (B^\eps \setminus B) \cup \partial B$ and $w \in B^\eps \setminus B$ be such that there are $\iota, \jmath \in [0,1)$ and $\eta > 0$ such that
%		\[ |z - x_0| - \alpha N \in [\eta, \tfrac{1}{\eta}] N^\iota \quad , \quad |w - x_0| - (\alpha + \eps) N \in -[\eta, \tfrac{1}{\eta}] N^\jmath \, . \]
%		Then, there exist explicit constants that depend only on $\alpha, \eps, d$ such that for $N$ large enough,
%		\begin{equation}\label{eq:proba-atteindre-avant-SRW}
%			\begin{split}
%			c_1 \tfrac{\eta}{2} N^{\iota-1} + \grdO(N^{-1}) \leq \, &\Pbf_z(H_{B^\eps} < \bar{H}_B) \leq c_2 \tfrac{2}{\eta} N^{\iota-1} + \grdO(N^{-1}) \, , \\
%			c_1' \tfrac{\eta}{2} N^{\jmath-1}  + \grdO(N^{-1}) \leq \, &\Pbf_w(H_B < \bar{H}_{\partial B^\eps}) \leq c_2' \tfrac{2}{\eta} N^{\jmath-1}  + \grdO(N^{-1}) \, .
%			\end{split}
%		\end{equation}

Let us focus on the first upper bound in Lemma~\ref{lem:arg-mart-tilt-B-Beps}.
Using \eqref{eq:encadrement-proba} with $C = B^\eps \setminus B$ and $A = \mathset{\bar{H}_B > H_{\partial B^\eps}}$, we have for $x\in \partial B$,
		\begin{equation}\label{eq:lem:arg-mart-tilt-B-Beps:UB-encadrement}
			\Pbf^{\Psi_N}_x (\bar{H}_B > H_{\partial B^\eps}) \leq \frac{c}{N} + c \sum_{k \geq 1} e^{c_0 k} \probaRW{x}{kN^2 \leq H_{\partial B^\eps} < \bar{H}_B} \, ,
		\end{equation}
where we have used~\eqref{eq:proba-atteindre-avant-SRW} with $\iota=0$ to bound the first term $\mathbf{P}_x(H_{\partial B^\eps} < \bar{H}_B)$ by $c/N$.	
We fix a $\delta \in (0,\eps)$ then bound each term in the sum by splitting on whether $\partial B^\delta$ has been reached before time $\tfrac{k}{2} N^2$.		
If it is the case, we have
		\begin{equation}\label{eq:lem:encadrement-mes-harm-tilt:up1}
			\begin{split}
				&\probaRW{x}{H_{\partial B^\eps} \wedge \bar{H}_B \geq k N^2 \, , \, H_{\partial B^\delta} \leq \frac{k}{2} N^2} \\			 
				&\qquad\qquad\qquad \leq \probaRW{x}{H_{\partial B^\delta} < \bar{H}_B} \sup_{z \in \partial B^\delta} \probaRW{z}{H_{\partial B^\eps} \wedge \bar{H}_B > \frac{k}{2} N^2} 		
				\leq \frac{C}{\delta N} e^{- \frac{c_a}{2 \eps^2} k} \, ,
			\end{split}
		\end{equation}
		where for the last inequality we used \cite[Lem.~6.3.4]{lawlerRandomWalkModern2010} to bound the first factor and Proposition~\ref{prop:rester-ds-boule-anneau} for the second factor.
		
		In order to get a bound when the tilted RW does not reach $\partial B^\delta$ before time $\tfrac{k}{2} N^2$, we first make the following claim on simple random walks (about ``gambler's ruin duration''), whose proof we postpone to Appendix~\ref{appendix:gambler-ruin} (again, the result seems standard but we did not find it in the literature).
%		Recall, for a set $K \subset \ZZ^d$, the definitions~\eqref{def:HK} of the hitting times $H_K, \bar{H}_K$.
		
		\begin{claim}
			\label{claim:esp-tps-retour-bord-anneau}
			For all $\delta \in (0,\eps)$, there is a $c_\delta > 0$ such that
			\[
			\sup_{z \in \partial B} \probaRW{z}{\bar{H}_B \wedge H_{\partial B^\delta} > (\delta N)^2} \leq \frac{c_\delta}{N} \, , \qquad \sup_{w \in \partial B^\delta} \probaRW{w}{H_B \wedge \bar{H}_{ \partial B^\delta} > (\delta N)^2 } \leq \frac{c_\delta}{N} \,.
			\]
		\end{claim}	
		
		Using the first inequality in Claim~\ref{claim:esp-tps-retour-bord-anneau} and Proposition \ref{prop:rester-ds-boule-anneau}, we get, for $x\in \partial B$,
\begin{equation}\label{eq:lem:encadrement-mes-harm-tilt:up2}
			\begin{split}
				&\probaRW{x}{H_{\partial B^\eps} \wedge \bar{H}_B \geq k N^2 \, , \, H_{\partial B^\delta} > \frac{k}{2} N^2}\\ &\qquad \qquad\qquad\leq \probaRW{x}{\bar{H}_B \wedge H_{\partial B^\delta} > \frac{k}{2} N^2} \sup_{z \in B^\delta \setminus B} \probaRW{z}{H_{\partial B^\eps} \wedge \bar{H}_B > \frac{k}{2} N^2} 
				\leq \frac{c}{N} e^{- c_a \frac{k}{2 \eps^2}} \, , 
			\end{split}
		\end{equation}
		where $c > 0$ only depends on $\delta$ and the dimension.		
		Assembling \eqref{eq:lem:encadrement-mes-harm-tilt:up1} and \eqref{eq:lem:encadrement-mes-harm-tilt:up2} with \eqref{eq:lem:arg-mart-tilt-B-Beps:UB-encadrement} yields
		\[ \Pbf^{\Psi_N}_x (\bar{H}_B > H_{\partial B^\eps}) \leq \frac{c}{N} + \frac{c}{N} \sum_{k \geq 1} \exp \left( c_0 k - \frac{k}{\eps^2} c_a \right) \, , \]
		and the last sum is finite thanks to our assumption $c_a > \eps^2 c_0$.
This proves the upper bound on $\Pbf^{\Psi_N}_x (\bar{H}_B > H_{\partial B^\eps})$ for $x \in \partial B$. 

The upper bound on $\Pbf^{\Psi_N}_y(H_B < \bar{H}_{\partial B^\eps})$ for $y \in \partial B^\eps$ follows from the same arguments, using the second part of Claim~\ref{claim:esp-tps-retour-bord-anneau}.
\end{proof}

Some easy adaptation of the previous proof actually gives the following lemma.

\begin{lemma}\label{lem:toucher-B-avant-Beps-start-Delta}
For $\gamma \in (0,1)$ fixed, for any $\eta>0$ there is a constant $c_{\eta}$ such that uniformly in $y$ satisfying $ \eta N^{\gamma} \leq d(y,B) \leq \eta^{-1} N^\gamma$, we have
		\begin{equation}\label{eq:borne-inf-proba-Delta-infini}
			c_{\eta} N^{\gamma - 1} \leq \Pbf^{\Psi_N}_y (\bar{H}_B = +\infty) \leq \Pbf^{N}_y (\bar{H}_B > H_{\partial B^\eps}) \leq c_{\eta}^{-1} N^{\gamma - 1} \, .
		\end{equation}
\end{lemma}
	
\begin{proof}
	With the same arguments as in the proof of Lemma \ref{lem:encadrement-mes-harm-tilt}, the proof boils down to proving the existence of positive constants $c,c'$, that do not depend on $N$ or $y$ such that $\eta N^{\gamma} \leq d(y,B) \leq \eta^{-1} N^\gamma$, for which:
	\begin{equation}
		c N^{\gamma - 1} \leq \Pbf^{\Psi_N}_y (\bar{H}_B > H_{\partial B^\eps}) \leq c' N^{\gamma - 1} \, .
	\end{equation}
	This amounts to small changes in the proof of Lemma \ref{lem:arg-mart-tilt-B-Beps}, namely the gambler's ruin argument \eqref{eq:proba-atteindre-avant-SRW} now gives $N^{\gamma - 1}$ instead of $N^{-1}$.
One also needs to adapt Claim~\ref{claim:esp-tps-retour-bord-anneau} to get an upper bound $\probaRW{y}{\bar{H}_B \wedge H_{\partial B^\delta} > (\delta N)^2}\leq c N^{\gamma - 1}$, but its proof is easily adapted, see Comment~\ref{rem:adapt-pt-depart-Delta}.
In particular, the inequalities~\eqref{eq:lem:encadrement-mes-harm-tilt:up1}-\eqref{eq:lem:encadrement-mes-harm-tilt:up2} still hold, with a prefactor $N^{\gamma - 1}$ instead of $N^{-1}$.
\end{proof}
	
\begin{remark}
	As it was previously mentioned, we could have taken $B$ to be another shape than a ball, provided this shape is ``$s N$-regular'' for some $s > 0$, see Comment~\ref{rem:shape}.
%	: there is a $s > 0$ such that for any $N$ large enough, for any point $x \in \partial B_N$, there are balls $B^{\mathrm{in}} \subseteq B_N$, $B^{\mathrm{out}} \subseteq \ZZ^d \setminus B_N \cup \partial B_N$ of radius $s N$ that are both tangent to $B_N$ at point $x$. 
Then, Lemma~5.5 and Proposition~5.6 of \cite{teixeiracoupling} ensure that the estimates on SRW used in the proofs of Lemmas~\ref{lem:encadrement-mes-harm-tilt} and~\ref{lem:arg-mart-tilt-B-Beps} are still valid.
\end{remark}

%	\section{Coupling CRW - interlacement}\label{sec:couplage-CRW-entrelac}
%	
	%QB Coupé la section 3 en deux parties (et changé ce titre)
	\section{Preliminary estimates for the coupling}\label{ssec:chaines-markov}

%	This section is devoted to the proof of Theorem \ref{th:couplage-CRW-entrelac-t_N}. 
	The main tool to prove Theorem \ref{th:couplage-CRW-entrelac-t_N} will be the soft local time coupling Theorem~\ref{th:soft-local-times}, applied to the Markov chain that is constituted by the entrance and exit points of excursions from $\partial B$ to an exterior set $\Delta$ at a given distance from $B$.
	
	Recall that $B$ is a discrete Euclidean ball centered at $x_0^N \in D_N$ with radius $\alpha N$, that satisfies that $\partial B$ is at distance at least $2\eps N$ from $\partial D_N$.
	In Theorem \ref{th:couplage-CRW-entrelac-t_N} we fix $\delta > 0$ to have
	\[ t_N/N^{2+\delta} \to +\infty \quad , \quad \eps_N = N^{-\delta/4} \to 0 \, . \]
	We also fix $\gamma$ that satisfies $1 - \frac{\delta}{4(d-2)} < \gamma < 1$, as well as $\Delta$ by
	\begin{equation}\label{eq:def-Delta}
		\Delta = \Delta_N =  \mathset{x \in D_N \, : \, d(x,B) > N^\gamma} \, .
	\end{equation}
	In all the following, constants will be usually denoted by $C, c, c'$ and will always only depend on $D, \alpha, \eps, \delta, d$ and $\gamma$.
	
	For reasons that will become clear later, we define the boundary of $\Delta$ as only the points of the actual boundary that are \og{}truly\fg{} inside $D_N$, that is $\partial \Delta \defeq \mathset{x \in \Delta \, : \, \exists y \in \Delta^c \cap D_N, x \sim y}$. In another word, $\partial \Delta = \partial B(x_0^N, \alpha N + N^\gamma)$.

\subsection{Notation}
\label{ssec:defYZ}

We will couple the entrance and exit points of confined RW and tilted RI random walk excursions from $B$ to $\partial \Delta$, so we need to introduce some notation for both these objects.
For convenience, we adopt similar notation as in~\cite{teixeiracoupling}.

\medskip
First, consider the confined RW, with transitions given by $p_N$. We define the set of hitting times $\mathcal{R} \defeq \mathset{R_i}_{i \geq 1}$ of $B$ and exit time $\mathcal{D} \defeq \mathset{D_i}_{i \geq 0}$ on $\partial \Delta$ as $D_0 = H_\Delta$ and for $i \geq 1$,
		\begin{equation}\label{eq:def-instants}
			R_i \defeq H_B \circ \theta_{D_{i-1}} + D_{i-1} \, , \quad D_i = H_\Delta \circ \theta_{R_i} + R_i \, ,
		\end{equation}
	with $\theta_t$ the time-shift operator at time $t$.
	The process $Y_i \defeq (X_{R_i},X_{D_i})$ for $i \geq 1$ is a Markov chain on $\Sigma:=\partial B \times \partial \Delta$,
	with transition probabilities
		\begin{equation}\label{eq:noyau-Y}
			p^Y((z,w),(x,y)) = \Pbf_x^N(Y_{i+1} = (x,y) \, | \, Y_i = (z,w)) = \Pbf^N_w (X_{H_B} = x) \Pbf_x^N (X_{H_\Delta} = y) \, ,
		\end{equation}
	and starting distribution, under $\Pbf^N_{\phi_N^2}$,
	\begin{equation}
		\label{def:nuY}
		 \nu_Y(x,y) \defeq \Pbf^N_{\phi_N^2} (Y_1 = (x,y)) = \frac{1}{\| \phi_N\|_2^2} \sum_{z \in D_N} \phi_N^2(z) \Pbf^N_z (X_{R_1} = x) \times \Pbf^N_x (X_{H_\Delta} = y) \, .
		\end{equation}

\medskip
	On the other hand, we want to create a Markov chain $Z$ which will contain the information about the tilted interlacement $\mathscr{I}_{\Psi_N}(u_N)$ restricted to $B$. 
	To this end, recall \eqref{eq:entrelac-tilt-simul}: $\mathscr{I}_{\Psi_N}(u_N) \cap B$ can be decomposed into independent trajectories $\mathset{w^j ; j \geq 0}$ with law $\Pbf^{\Psi_N}_{\bar{e}_B^{\Psi_N}}$.
	 
	\noindent Since there is an almost surely finite number $N_B^{\Psi_N, u_N}$ of trajectories $(w^j)_{j \geq 1}$ that touch $B$ we can order them by their label $(u_j)_{j \geq 1}$. 
	To each trajectory $w^j$ correspond sets $\mathcal{R}^j$ and $\mathcal{D}^j$ which collect the entrance and exit times of $w^j$ as in \eqref{eq:def-instants}:
	$D^j_0 = H_\Delta(w^j)$ and for $i \geq 1$,
	\begin{equation}\label{eq:def-instants-entrelac}
		R^j_i \defeq H_B(w^j) \circ \theta_{D^j_{i-1}} + D^j_{i-1} \, , \quad D^j_i = H_\Delta(w^j) \circ \theta_{R^j_i} + R^j_i \, .
	\end{equation}
	Write $T^{j} = |\mathcal{R}^j| = |\mathcal{D}^j| - 1$ for the number of excursions $(w^j_i)_{1 \leq i \leq T^j}$ of the $j$-th trajectory. We can define for all $k \in \NN$ the random variable
	\begin{equation}
		Z_k = (w^j_i(0),w^j_i(H_\Delta)) \, , \quad \text{with} \quad k = i + \sum_{p = 1}^{j-1} T^p \, .
	\end{equation}
The resulting process $Z = (Z_k)_{k \geq 0}$ is a Markov chain on $\Sigma$ with transition probabilities
		\begin{equation}\label{eq:noyau-Z}
			p^Z((z,w),(x,y)) = \left( \Pbf^{\Psi_N}_w(H_B < +\infty , X_{H_B} = x) + \Pbf^{\Psi_N}_w(H_B = +\infty) \bar{e}^{\Psi_N}_B(x) 	\right) \Pbf_x^N (X_{H_\Delta} = y) \, ,
		\end{equation}
	and starting distribution 
\begin{equation}
\label{def:nuZ}
\nu_Z(x,y) = \bar{e}^{\Psi_N}_B(x) \Pbf_x^N (X_{H_\Delta} = y) \,.
\end{equation}	

\begin{figure}[h]\label{fig:excursions}
	\centering
	\caption{Some excursions of the confined RW (left) and of two trajectories $w^j, w^{j'}$ of the tilted RI (right).}
	\includegraphics[width=7cm]{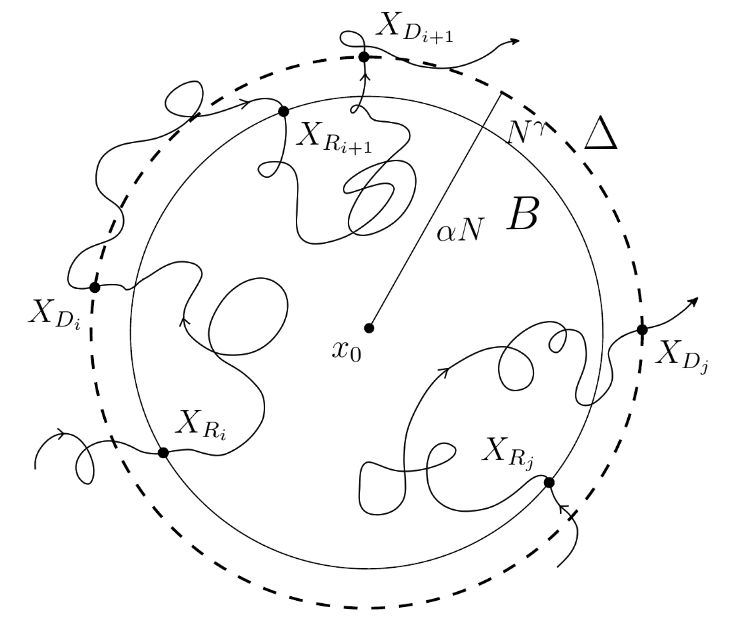}
	\includegraphics[width=7.5cm]{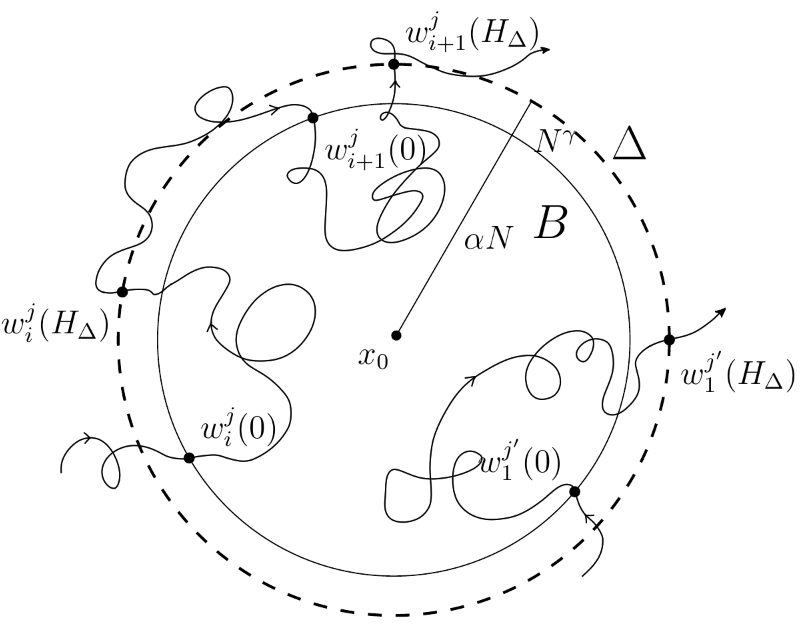}
\end{figure}

	The proof of Theorem \ref{th:couplage-CRW-entrelac-t_N} then mostly consists of applying Theorem \ref{th:soft-local-times} to the Markov chains~$Y$ and~$Z$. To do so, we need the following four results:
	\begin{itemize}
		\item In Section \ref{ssec:mes-invariante-1}, we work out the invariant measures of the Markov chains $Y$ and $Z$, and we control their difference in order to apply Theorem \ref{th:couplage-CdM-diff-pi}.
		\item We prove bounds on the mixing times (Section \ref{ssec:tps-melange-1}) and variances (Section \ref{ssec:variance-1}) that appear in the upper bound \eqref{eq:borne-soft-local-times-diff} of the theorem.
	\end{itemize}

After that, we estimate in Section \ref{ssec:nb-excursions} the required size of the range of $Y$ and $Z$, i.e. the number of excursions of the confined RW before time $t_N$ and of the tilted RI $\mathscr{I}_{\Psi_N}(u_N)$. This will give us the time $n$ in Theorem \ref{th:soft-local-times}.

\begin{definition}
	We define the $(\phi_N, \Delta)$-equilibrium measure $e_\Delta^{\phi_N}$ on $\partial B$ and its total mass $\cpc_\Delta^{\phi_N}(B)$ by
	\begin{equation}\label{eq:def-mes-harm-capacite-phi}
		e_\Delta^{\phi_N}(x) = \phi_N^2(x) \Pbf^N_x (H_\Delta < \bar{H}_B) \mathbbm{1}_{\partial B}(x) \quad , \quad \cpc_\Delta^{\phi_N}(B) = \sum_{x \in \partial B} \phi_N^2(x) \Pbf^N_x (H_\Delta < \bar{H}_B) \, .
	\end{equation}
	We write $\tilde{e}_\Delta^{\phi_N}$ for the probability measure $e_\Delta^{\phi_N} / \cpc_\Delta^{\phi_N}(B)$.
\end{definition}

	\subsection{Invariant measure}\label{ssec:mes-invariante-1}
	
	As is explained in Section \ref{ssec:soft-local-times}, the main requirement to couple the two Markov chains~$Y$ and~$Z$ is for them to have the same invariant measure. In this section, we prove this statement and give some related estimates that we will use further in the paper.
	
	\begin{lemma}\label{lem:mes-inv-Y}
		The Markov chain $Y$ admits an invariant probability $\tilde{\pi}_Y$ on $\partial B \times \partial \Delta$ given by
		\begin{equation}
				\tilde{\pi}_Y(x,y) \defeq \tilde{e}_\Delta^{\phi_N}(x) \Pbf^N_x (X_{H_\Delta} = y) \, .
		\end{equation}
	\end{lemma}
	
	\begin{remark}
		Using \eqref{eq:proba-htransform-RW-tuee}, we can see that for $x\in \partial B$,
		\[
		e_\Delta^{\phi_N}(x) = \phi_N(x) \espRW{x}{\phi_N(X_{H_\Delta})  \lambda_N^{-H_\Delta} \indic{H_\Delta < \bar{H}_B}} 
		\]
		and
		\[
		\Pbf^N_x (X_{H_\Delta} = y) = \frac{\phi_N(y)}{\phi_N(x)} \espRW{x}{\lambda_N^{-H_\Delta} \indic{X_{H_\Delta} = y}} \,,
		\]
		so that 
		\[ \tilde{\pi}(x,y) = \frac{1}{\cpc_\Delta^{\phi_N}(B)} \espRW{x}{\phi_N(X_{H_\Delta})  \lambda_N^{-H_\Delta} \indic{H_\Delta < \bar{H}_B}} \phi_N(y) \espRW{x}{\lambda_N^{-H_\Delta} \indic{X_{H_\Delta} = y}} \, . \]
	\end{remark}

	\begin{proof}
		Consider the stationary random walk $(X_k)_{k \in \ZZ}$ under $\Pbf^N_{\phi_N^2}$. With the definition of the return set $\mathcal{R}$, by a last-passage decomposition, we have
		\[ \Pbf^N_{\phi_N^2}(n \in \mathcal{R}, X_n = x) = \Pbf^N_{\phi_N^2}(\exists w \in \partial \Delta, \exists m\geq 1, X_{n-m} = w, X_{[n-m+1,n-1]} \subseteq (B \cup \Delta)^c, X_n = x) \, , \]
		where we have introduced the notation $X_{[a,b]} = \{X_a,X_{a+1}, \ldots, X_b\}$ for $a,b \in \ZZ$.

		Therefore, decomposing on every $m \geq 1$ and $w \in \partial \Delta$, since the events are disjoint we obtain
		\begin{equation}\label{eq:mes-inv-Y:retour-decomp}
			\Pbf^N_{\phi_N^2}(n \in \mathcal{R}, X_n = x) = \sum_{w \in \partial \Delta} \sum_{m \geq 1} \Pbf^N_{\phi_N^2}(X_{n-m} = w, X_{[n-m+1,n-1]} \subseteq (B \cup \Delta)^c, X_n = x) \, .
		\end{equation}
		Since $\phi_N^2$ is the stationary measure of $(X_k)_{k\in \mathbb{Z}}$, we have
		\begin{equation}\label{eq:lem:mes-inv-Y:phiN2-w}
			\Pbf^N_{\phi_N^2}(X_{n-m} = w, X_{[n-m+1,n-1]} \subseteq (B \cup \Delta)^c, X_n = x) = \frac{\phi_N^2(w)}{\| \phi_N\|_2^2} \Pbf^N_{w}(X_{[1,m-1]} \subseteq (B \cup \Delta)^c, X_{m} = x) \, .
		\end{equation}
		Moreover, since $\phi_N^2$ is a reversible measure for the confined RW (recall that random walks on conductances are reversible),
		\begin{equation}\label{eq:lem:mes-inv-Y:reverse}
			\Pbf^N_{w}(X_{[1,m-1]} \subseteq (B \cup \Delta)^c, X_{m} = x) = \frac{\phi_N^2(x)}{\phi_N^2(w)} \Pbf^N_{x}(X_{[1,m-1]} \subseteq (B \cup \Delta)^c, X_{m} = w) \, .
		\end{equation}
		Therefore, injecting \eqref{eq:lem:mes-inv-Y:reverse} in \eqref{eq:lem:mes-inv-Y:phiN2-w} and then in \eqref{eq:mes-inv-Y:retour-decomp}, we get
		\begin{equation}\label{eq:proba-retour-temps-site}
			\begin{split}
				\Pbf^N_{\phi_N^2}(n \in \mathcal{R}, X_n = x) &= \frac{\phi_N^2(x)}{\| \phi_N\|_2^2} \sum_{w \in \partial \Delta} \sum_{m \geq 1} \Pbf^N_{x}(X_{[1,m-1]} \subseteq (B \cup \Delta)^c, X_{m} = w)\\
				&= \frac{\phi_N^2(x)}{\| \phi_N\|_2^2} \Pbf^N_x(\bar{H}_B > H_\Delta) \, ,
			\end{split}
		\end{equation}
which does not depend on $n$ and is proportional to~$e_B^{\phi_N}(x)$.
		Using the ergodic theorem, $\tilde{\pi}$ satisfies
		\begin{equation}\label{eq:pi-th-ergodique}
			\tilde{\pi}_Y(\mathset{x} \times \partial \Delta) = \lim_{k \to +\infty} \frac{\sum_{i = 1}^k \indic{i \in \mathcal{R}, X_i = x}}{\sum_{i = 1}^k \indic{i \in \mathcal{R}}} = \frac{\Pbf^N_{\phi_N^2} (n \in \mathcal{R}, X_n = x)}{\Pbf^N_{\phi_N^2} (n \in \mathcal{R})} \, ,
		\end{equation}
		which proves that $\tilde{\pi}_Y(\mathset{x} \times \partial \Delta) = \tilde{e}_B^{\phi_N}(x)$. Then, using the Markov property, we see that $\tilde \pi(x,y) = \tilde{\pi}_Y(\mathset{x} \times \partial \Delta) \Pbf^N_x (X_{H_\Delta} = y)$, which concludes the proof.
	\end{proof}

	\begin{lemma}\label{lem:mes-inv-Z}
		The Markov chain $Z$ admits the invariant measure
		\begin{equation}
			\tilde{\pi}_Z(x,y) = \Pbf^{\Psi_N}_x \big( H_\Delta < \bar{H}_B < +\infty) \Pbf^N_x \big( X_{H_\Delta} = y \big) \mathbbm{1}_{x \in \partial B} \, .
		\end{equation}
		Moreover, there is a constant $c > 0$ such that for all $N$ large enough,
		\begin{equation}\label{eq:erreur-mes-invariante}
			\sup_{(x,y) \in \Sigma} \Big| \frac{\tilde{\pi}_Y(x,y)}{\tilde{\pi}_Z(x,y)} - 1 \Big| \leq c N^{\gamma - 1} \, .
		\end{equation}
	\end{lemma}
	
	\begin{proof}
		Recall that $Z$ is constructed thanks to a collection of i.i.d.\ random walks on conductances, whose trajectories are denoted $(w^j)_{j\ge 1}$. Therefore, choosing $x  \in \partial B$, $\pi_Z(\mathset{x} \times \partial \Delta)$ can be expressed by the average number of visits to $x$ between two consecutive visits to a reference site. Here, the reference site is a state $+\infty$ added to $B$, which corresponds to count the average number of visits at $x \in \partial B$ by a single trajectory $w^j$ of the tilted RI.
		We end up with
		\begin{equation}\label{eq:mes-inv-entrelac}
			\tilde{\pi}_Z( \mathset{x} \times \partial \Delta) = \frac{1}{\mathbf{E}_{\bar{e}^{\Psi_N}_B}^{\Psi_N} \left[ T^{1} \right]} \mathbf{E}_{\bar{e}^{\Psi_N}_B}^{\Psi_N}  \left[ \sum_{i = 1}^{T^{1}} \indic{X_{{R_i}} = x} \right] \, ,
		\end{equation}
		where we recall that $T^{1}$ is the number of excursions of the first trajectory.
%		{\red [JP: en effet, c'est la formule classique de la loi stationnaire pour une chaîne de Markov, exprimée en fonction du nombre moyen de passages en chaque site entre deux visites consécutives à un site fixé de référence. Ici, le site de référence est un peu caché ; il correspond à un site ``$+\infty$" ajouté à l'ensemble $B$.]}

Then, as for the conclusion of the proof of Lemma~\ref{lem:mes-inv-Y}, by the Markov property we get that 
\[
\tilde{\pi}_Z(x,y) = \tilde{\pi}_Z( \mathset{x} \times \partial \Delta) \mathbf{P}_{x}^{\Psi_N} \left( X_{H_\Delta} = y \right) = \tilde{\pi}_Z( \mathset{x} \times \partial \Delta)  \mathbf{P}_{x}^{N} \left( X_{H_\Delta} = y \right)\,,
\] 
since $\Psi_N =\phi_N$ on $\Delta^c$.

%		Note that on $\Delta^c$, we have $\mathbf{P}_{x}^{T,N} \left( X_{H_\Delta} = y \right) = \Pbf^N_x \left( X_{H_\Delta} = y \right)$. 
		Thus, we only need to prove that  \eqref{eq:mes-inv-entrelac}, or more precisely the second term on its right-hand side, is proportional to $e_\Delta^{\phi_N}(x)$ (normalizing constants are unimportant).
		We first rewrite it using the definition of $R_i$: with a similar last-passage decomposition as in \eqref{eq:mes-inv-Y:retour-decomp}, the last expectation in \eqref{eq:mes-inv-entrelac} is equal to
		\begin{equation}
			\label{eq:mes-inv-entrelac2}
\begin{split}		
 \sum_{y \in B} &  \bar{e}^{\Psi_N}_B(y) \sum_{k \geq 0}  \sum_{w \in \partial \Delta} \mathbf{P}_{y}^{\Psi_N} \Big( X_k = x, \exists j \geq  1, X_{k-j} = w, X_{[k-j+1,k-1]} \subseteq (B \cup \Delta)^c \Big) \\
 & = \sum_{y \in B} \bar{e}^{\Psi_N}_B(y) \sum_{k \geq 0} \sum_{j=1}^{k} \sum_{w \in \partial \Delta} \mathbf{P}_{y}^{\Psi_N}\left( X_{k-j} = w \right) \Pbf^{\Psi_N}_{w} \left( X_{j} = x, X_{[1,j-1]} \subseteq (B \cup \Delta)^c \right) \,,
		\end{split}
		\end{equation}
where the last line follows by the Markov property.
Reversing both paths, similarly to~\eqref{eq:lem:mes-inv-Y:reverse} we have
		\[
		\mathbf{P}_{y}^{\Psi_N}\left( X_{k-j} = w \right) = \frac{\Psi_N(w)^2}{\Psi_N(y)^2} \mathbf{P}_{w}^{\Psi_N}\left( X_{k-j} = y \right) 
		\]
		and
		\[
		\Pbf^{\Psi_N}_{w} \left( X_{j} = x, X_{[1,j-1]} \subseteq (B \cup \Delta)^c \right)
		= \frac{\Psi_N(x)^2}{\Psi_N(w)^2} \Pbf^{\Psi_N}_{x} \left( X_{j} = w, X_{[1,j-1]} \subseteq (B \cup \Delta)^c \right) \,.
		\]
		Hence, using also that $\Psi_N = \phi_N$ on $\Delta^c \subset B^{\eps}$ (and Fubini's theorem), \eqref{eq:mes-inv-entrelac2} is proportional to
		\[
		\begin{split}
			\sum_{j \geq 1}  \sum_{w \in \partial \Delta} \phi_N^2(x)  \Pbf^N_{x} \big( X_j = w, X_{[1,j-1]} \subseteq (B \cup \Delta)^c \big)  \sum_{y\in B} \sum_{k' \geq 1} e^{\Psi_N}_B(y) \frac{1}{\phi_N^2(y)} \mathbf{P}_{w}^{\Psi_N} \left( X_{k'} = y \right)  &  \\
			 = \phi_N^2(x)  \sum_{w \in \partial \Delta} \mathbf{P}^N_{x} \big( H_\Delta < \bar{H}_B, X_{H_{\Delta}} = w \big)  \sum_{y\in B}  \frac{1}{\phi_N^2(y)} e^{\Psi_N}_B(y) G^{\Psi_N}(w,y) & \, .
			\end{split}
		\]
		
%		where we used the fact that $\mathbf{P}_{x}^{T,N} = \Pbf_x^N$ on $\Delta^c \subset B^\eps$. 
		Using the last exit decomposition for the tilted walk \eqref{eq:LED-tilted}, we get that
		\begin{equation}\label{eq:esperance-pi-Z-terme-final}
			\begin{split}
			 	\mathbf{E}_{\bar{e}^{\Psi_N}_B}^{\Psi_N} \left[ \sum_{i = 1}^{T^1} \indic{X_{{R_i}} = x} \right] &= \lambda_N \phi_N^2(x) \sum_{w \in \partial \Delta} \mathbf{P}^N_{x} \big( H_\Delta < \bar{H}_B, X_{H_{\Delta}} = w \big) \Pbf_w \big( H_B < +\infty \big) \\
			 	&= \lambda_N \phi_N^2(x) \mathbf{P}^{\Psi_N}_{x} \left( H_\Delta < \bar{H}_B < +\infty \right) \, ,
			\end{split}
		\end{equation}
where the last inequality stems from the strong Markov property.

To get the second part of the lemma, we rewrite \eqref{eq:esperance-pi-Z-terme-final} as
\begin{equation}
	\tilde{\pi}_Z(x,y) = \mathbf{P}^{N}_{x} \left( H_\Delta < \bar{H}_B \right) - \sum_{w \in \partial \Delta} \mathbf{P}^N_{x} \big( H_\Delta < \bar{H}_B, X_{H_{\Delta}} = w \big) \Pbf_w \big( H_B = +\infty \big)
\end{equation}
Observe that a use of Lemma \ref{lem:toucher-B-avant-Beps-start-Delta} yields that $\Pbf_w \big( H_B = +\infty \big) \asymp N^{\gamma - 1}$ uniformly in $x \in \partial B, y \in \partial \Delta$, in particular
\begin{equation}
	\frac{\tilde{\pi}_Z(x,y)}{\tilde{\pi}_Y(x,y)} - 1 \asymp \frac{1}{\mathbf{P}^{N}_{x} \big( H_\Delta < \bar{H}_B \big)} \sum_{w \in \partial \Delta} \mathbf{P}^N_{x} \big( H_\Delta < \bar{H}_B, X_{H_{\Delta}} = w \big) N^{\gamma - 1} = N^{\gamma - 1} \, ,
\end{equation}
hence proving the second part of the lemma.
%		Then, using the Markov property, $\pi_Z(x,y) = \tilde{e}_\Delta^{\phi_N}(x) \Pbf^N_x (X_{H_\Delta} = y) = \tilde{\pi}(x,y)$.
	\end{proof}
	
	In the rest of the paper, we will simply write $\tilde{\pi} = \tilde{\pi}_Y$, since according to \eqref{eq:erreur-mes-invariante} $\tilde{\pi}_Z$ is close to $\tilde{\pi}_Y$.
	
	To end this section, we provide estimates on $\tilde{\pi}$ that will be useful later. Namely, we will need the probability under $\tilde{\pi}$ that $y \in \partial \Delta$ is an exit site, as well as estimates on $\tilde{e}_\Delta^{\phi_N}$.

	\begin{lemma}\label{lem:encadrement-sortie-Delta}
	There are universal constants $c,C>0$ such that, uniformly in $y \in \partial \Delta$, we have
		\begin{equation}
		c N^{1-d} \leq \tilde{\pi}(\partial B \times \mathset{y}) \leq C N^{1-d} \, .
		\end{equation}
	\end{lemma}
	
	\begin{proof}
		With the same method as in the proof of Lemma \ref{lem:mes-inv-Y} (see e.g.~\eqref{eq:proba-retour-temps-site}), we get, for $y\in \partial \Delta$
		\[
		\begin{split}
			\Pbf^N_{\phi_N^2}&(X_n = y , n \in \mathcal{D})  = \sum_{x \in \partial B} \sum_{m \geq 1} \frac{\phi_N^2(x)}{\| \phi_N\|_2^2} \Pbf^N_{x}(X_{[1,m-1]} \subseteq (B \cup \Delta)^c, X_{m} = y) \\
			& \qquad \qquad = \sum_{x \in \partial B} \sum_{m \geq 1} \frac{\phi_N^2(y)}{\| \phi_N\|_2^2} \Pbf^N_{y}(X_{[1,m-1]} \subseteq (B \cup \Delta)^c, X_{m} = x)  = \frac{\phi_N^2(y)}{\| \phi_N\|_2^2} \Pbf^N_{y}( H_B < \bar{H}_\Delta)\,,
		\end{split}
		\]
	where we have reversed time in the second identity, similarly as in~\eqref{eq:lem:mes-inv-Y:reverse}.
		Now, thanks to the ergodic theorem, we obtain, similarly as in \eqref{eq:pi-th-ergodique},
		\begin{equation*}
			\tilde{\pi}(\partial B \times \mathset{y}) = \frac{\Pbf^N_{\phi_N^2}(X_n = y , n \in \mathcal{D})}{\Pbf^N_{\phi_N^2}(n \in \mathcal{D})} = \frac{\phi_N^2(y) \Pbf^N_y (H_B < \bar{H}_\Delta)}{\sum_{w \in \partial \Delta} \phi_N^2(w) \Pbf^N_w (H_B < \bar{H}_\Delta)} \, .
		\end{equation*}

We now estimate $\Pbf^N_y (H_B < \bar{H}_\Delta)$ uniformly in $y\in \partial \Delta$, with an easy adaptation of the proof of Lemma \ref{lem:arg-mart-tilt-B-Beps} and thanks to the small width of $(B \cup \Delta)^c$.
		Applying \eqref{eq:encadrement-proba} with the set $C=\Delta^c \setminus B$ and the event $A = \big\{H_B < \bar{H}_\Delta \big\}$, we have
		\begin{equation}\label{eq:decomp-toucher-B-avant-Delta}
		\begin{split}
			\kappa_1 \Pbf_y (H_B < \bar{H}_\Delta) & \leq \Pbf^N_y (H_B < \bar{H}_\Delta) \\
			& \qquad  \leq \frac{1}{\kappa_1} \left[ e^{c_0} \Pbf_y(H_B < \bar{H}_\Delta) + \sum_{k = 1}^{+\infty} e^{c_0 (k+1)} \Pbf_y(k N^2 \leq H_B < \bar{H}_\Delta) \right] \,.
			\end{split}
		\end{equation}
		With Lemma \ref{lem:gambler-ruin}, see~\eqref{eq:proba-atteindre-avant-SRW}), we have that $\Pbf_y (H_\Delta < \bar{H}_B) \asymp N^{-\gamma}$, uniformly in $y\in \partial \Delta$. 
		On the other hand, with Proposition~\ref{prop:rester-ds-boule-anneau}, we have $\Pbf_y(k N^2 \leq H_B < \bar{H}_\Delta) \leq e^{-c_a kN^{2(1-\gamma)}}$. Thus, the last sum in \eqref{eq:decomp-toucher-B-avant-Delta} is a $\bar{o}(N^{-\gamma})$ and we have proved that $\Pbf^N_y (H_B < \bar{H}_\Delta) \asymp N^{-\gamma}$, uniformly in $y\in \partial \Delta$. With Proposition~\ref{prop:encadrement-ratio}, this proves that
		\begin{equation*}
			\tilde{\pi}(\partial B \times \mathset{y}) = \frac{\phi_N^2(y) \Pbf^N_y (H_B < \bar{H}_\Delta)}{\sum_{w \in \partial \Delta} \phi_N^2(w) \Pbf^N_w (H_B < \bar{H}_\Delta)} \asymp \frac{\Pbf^N_y (H_B < \bar{H}_\Delta)}{\sum_{w \in \partial \Delta} \Pbf^N_w (H_B < \bar{H}_\Delta)} \asymp \frac{N^{-\gamma}}{N^{d-1-\gamma}} = N^{1-d} \, , 
		\end{equation*}
	again uniformly in $y\in \partial \Delta$.
	\end{proof}

	\begin{lemma}\label{lem:encadrement-mes-harmonique}
		There are universal constants $c, C > 0$ such that for all $x \in \partial B$,
		\begin{equation}
			c N^{1-d} \leq \tilde{e}_\Delta^{\phi_N}(x) \leq C N^{1-d} \quad , \quad c N^{d-1-\gamma} \leq \cpc_\Delta^{\phi_N}(B) \leq C N^{d-1-\gamma} \, .
		\end{equation}
	\end{lemma}
	
	\begin{proof}
		Again, applying \eqref{eq:encadrement-proba} with the set $C=\Delta^c \setminus B$ and the event $A = \big\{ \bar{H}_\Delta <H_B \big\}$, we have for $x \in \partial B$
		\begin{equation}\label{eq:decomp-toucher-Delta-avant-B}
		\begin{split}
			\kappa_1 \Pbf_x (H_\Delta < \bar{H}_B) &\leq \Pbf^N_x (H_\Delta < \bar{H}_B)\\
			&\qquad  \leq \frac{1}{\kappa_1} \left[ e^{c_0} \Pbf_x(H_\Delta < \bar{H}_B) + \sum_{k = 1}^{+\infty} e^{c_0 (k+1)} \Pbf_x(k N^2 \leq H_\Delta < \bar{H}_B) \right] \,.
			\end{split}
		\end{equation}
		Then, for the same reason as in the proof of Lemma \ref{lem:encadrement-sortie-Delta}, we have $\Pbf_x (H_\Delta < \bar{H}_B) \asymp N^{-\gamma}$ and the last sum in \eqref{eq:decomp-toucher-Delta-avant-B} is a $\bar{o}(N^{-\gamma})$, thus showing that $\Pbf^N_x (H_\Delta < \bar{H}_B) \asymp N^{-\gamma}$ uniformly in $x\in \partial B$.
		Combined with Proposition \ref{prop:encadrement-ratio}, this proves that
		\[ \tilde{e}_\Delta^{\phi_N}(x) = \frac{\phi_N^2(x) \Pbf^N_x (H_\Delta < \bar{H}_B)}{\sum_{z \in \partial B} \phi_N^2(z) \Pbf^N_z (H_\Delta < \bar{H}_B)} \asymp \frac{N^{-\gamma}}{ |\partial B| N^{-\gamma}} \asymp N^{1-d} \, . \]
		On the other hand,
		\begin{equation}\label{eq:cap-phi-somme-phi-surface}
			\cpc_\Delta^{\phi_N}(B) = \sum_{x \in \partial B} \phi_N^2(x) \Pbf^N_x (H_\Delta < \bar{H}_B) \asymp N^{-\gamma} \sum_{x \in \partial B} \phi_N^2(x) \,,
		\end{equation}
%		where we used the estimates in the proof of Lemma \ref{lem:encadrement-sortie-Delta}. 
		and with \eqref{eq:phiN-borne} we have
		\begin{equation}\label{eq:asymp-somme-phi-surface}
			\sum_{x \in \partial B} \phi_N^2(x) \asymp |\partial B| \asymp N^{d-1} \, .
		\end{equation}
		Combining \eqref{eq:cap-phi-somme-phi-surface} and \eqref{eq:asymp-somme-phi-surface}, this concludes the lemma.	
	\end{proof}

	\subsection{Upper bound on mixing times}\label{ssec:tps-melange-1}
	
		In this section, we prove bounds on the mixing times $T^Y_{\mix}$ and $T^Z_{\mix}$ that we will use in the upper bound of Theorem \ref{th:soft-local-times}. To bound the mixing times of a Markov chain $X$ on $\Sigma$, we exhibit a coupling $\QQ_{x,\tilde{x}}$ of two copies $X$ and $\tilde{X}$ starting at $x$ and $\tilde{x}$ respectively, such that $X_i = \tilde X_i$ implies $X_j = \tilde X_j$ for all $j\ge i$. Then, we evaluate their coalescing time under $\QQ_{x,\tilde{x}}$, which we use to get an upper bound on the mixing time $T_{\mix}$ with
		\begin{equation}\label{eq:borne-tps-mixage}
				T_{\mix} \leq \inf \bigg\{ k \in \NN \, : \, \sup_{x,\tilde{x} \in \Sigma} \QQ_{x,\tilde{x}}(X_k \neq \tilde{X}_k) \leq \tfrac14 \bigg \} \, .
		\end{equation}
	We refer to \cite[Corollary 5.3]{levin2017markov} for the proof of \eqref{eq:borne-tps-mixage} and more details on coalescing couplings.
	In the two following sections, we establish such coupling for the Markov chains $Y$ and $Z$, and prove the following result.
	
	\begin{proposition}\label{prop:temps-mixage}
		There is a constant $c > 0$ such that for all $N$ large enough,
		\begin{equation}
			T^Y_{\mix} \, \vee \, T^Z_{\mix} \leq c N^{(1-\gamma)} \, .
		\end{equation}
	\end{proposition}
	
	\subsubsection{Upper bound on the mixing time of the confined random walk}
	
	Consider the ball that is concentric with $B$ with half radius:
	 \[ G \defeq B(x_0^N, \tfrac{\alpha}{2} N) = \Big\{ z \in B \, : \, d(z,\partial B) \geq \tfrac12 \alpha N \Big\} \, . \]
	Consider two copies $Y$ and $\tilde{Y}$ with the law of the Markov chain with transitions $p^Y$, and write~$X$ and $\tilde{X}$ for their underlying confined RW. We couple $Y$ and $\tilde{Y}$ by making $X$ and $\tilde{X}$ go to $G$ during the same $i$-th excursion (between times $R_i$ and $D_i$). Then, trajectories will have a positive probability of being coupled before reaching again $\Delta$, \textit{i.e.}\ before time $D_i$.

\smallskip
\noindent
{\it Step 1. Probability of reaching $G$. }
	First, we estimate the probability that a confined RW starting at a point $x \in \partial B$ hits $G$ before exiting through $\Delta$. This will help control the probability of the event where $X$ and $\tilde{X}$ touch $G$ between $R_i$ and $D_i$ for the same $i \geq 1$.
	\begin{lemma}\label{lem:atteindre-G-avant-Delta}
		We have
		\begin{equation}
			\inf_{x \in \partial B} \Pbf^N_x (H_G < H_\Delta) \geq c N^{\gamma - 1} \, .
		\end{equation}
	\end{lemma}
	
	\begin{proof}
		This follows from $\Pbf^N_x (H_G < H_\Delta) \geq \kappa_1 \Pbf_x (H_G < H_\Delta)$ (recall \eqref{eq:encadrement-proba}) and Lemma \ref{lem:gambler-ruin}.
	\end{proof}
	
	With Lemma \ref{lem:atteindre-G-avant-Delta}, we see that reaching $G$ before $\Delta$ is a Bernoulli trial with success probability at least $c N^{\gamma - 1}$.

\smallskip
\noindent
{\it Step 2. Probability of being coupled starting from $\partial G$. }	
	For $y \in \partial \Delta$, we define 
	\[
	\mu_N(y) = \inf_{x \in \partial G} \Pbf^N_x (X_{H_\Delta} = y) \,,
	\] 
	which is a sub-probability measure on $\partial \Delta$. Starting at $z \in \partial G$, the law of $X_{H_\Delta}$ can be written as $\nu_z + \mu_N$ with $\nu_z \defeq \Pbf^N_z (X_{H_\Delta} \in \cdot) - \mu_N$. Defining $W$ a random variable with law $\mu_N/ \mu_N(\partial \Delta)$, then our coupling will consist in choosing $X_{D_i} = \tilde{X}_{D_i} = W$ with probability $\mu_N(\partial \Delta)$, while sampling $X_{D_i}$ and $\tilde{X}_{D_i}$ independently otherwise, according to the relevant exit point distributions.
	
	We prove that $\mu_N(\partial \Delta)$ is bounded from below by a constant, which means that at each simultaneous visit to $G$, there is a positive probability for the walks to coalesce on their way to $\Delta$.
	\begin{lemma}\label{lem:masse-inf-mesure-sortie}
		There is a $c_\mu \in (0,1)$, such that $\mu_N(\partial \Delta) \geq c_\mu$ for all $N$ large enough.
	\end{lemma}
	
	\begin{proof}
		With \eqref{eq:encadrement-proba},
		\[ \mu_N (\partial \Delta) = \sum_{y \in \partial \Delta} \inf_{x \in \partial G} \Pbf^N_x (X_{H_\Delta} = y) \geq \kappa_1 \sum_{y \in \partial \Delta} \inf_{x \in \partial G} \probaRW{x}{X_{H_\Delta} = y} \, . \]
		Since $x \mapsto \probaRW{x}{X_{H_\Delta} = y}$ is harmonic on $G$, by the Harnack inequality \cite[Theorem 6.3.9]{lawlerRandomWalkModern2010} there is $c > 0$ that depends neither on $N$ nor on $y$ such that for any $x_*\in G$, $\inf_{x \in \partial G} \probaRW{x}{X_{H_\Delta} = y} \geq c \probaRW{x_*}{X_{H_\Delta} = y}$, which means that
		\[ \mu_N(\partial \Delta) = \sum_{y \in \partial \Delta} \inf_{x \in \partial G} \Pbf^N_x (X_{H_\Delta} = y) \geq c \kappa_1 \sum_{y \in \partial \Delta} \probaRW{x_*}{X_{H_\Delta} = y} = c \kappa_1 \, . \qedhere \]
	\end{proof}

%	Considering two points $x,x'$ on $\partial G$, Lemma \ref{lem:masse-inf-mesure-sortie} states that our coupling of two tilted RW starting at $x$ and $x'$ have a probability at least $c_\mu$ of exiting through the same site on $\Delta$. 

\smallskip
\noindent
{\it Conclusion. }	
Putting together Step 1 and Step 2, the coalescing time is dominated by a geometric random variable with success parameter at least $c c_\mu N^{\gamma-1}$, which gives the bound on $T^Y_{\mix}$.

Let us write the coupling in more detail.
For all $i \geq 1$, we define two independent sequences of i.i.d Bernoulli random variables: $(\xi^1_i)_{i \geq 1}$ with parameter $c N^{\gamma-1}$ and $(\xi^2_i)_{i\geq 1}$ with parameter $c_\mu$. We couple $(Y_i=(X_{R_i},X_{D_i}))_{i\geq 0}$ and~$(\tilde{Y}_i=(\tilde{X}_{R_i},\tilde{X}_{D_i}))_{i\geq 0}$ in the following way: if $X_{R_i} = \tilde{X}_{R_i}$, then $X_{D_i} = \tilde{X}_{D_i}$; if on the other hand $X_{R_i}\neq \tilde X_{R_i}$, then
		\begin{itemize}
			\item If $\xi^1_i = 0$, we independently sample $X_{D_i} \in \partial \Delta$ and $\tilde{X}_{D_i} \in \partial \Delta$  according to the respective probabilities $\Pbf^N_{X_{R_i}} (X_{H_\Delta} \in \cdot)$ and $\Pbf^N_{\tilde{X}_{R_i}} (X_{H_\Delta} \in \cdot)$.
			\item If $\xi^1_i = 1$, we independently sample $U_i\in \partial G$ and $\tilde{U}_i\in \partial G$ according to the respective probabilities $\Pbf^N_{X_{R_i}} (X_{H_G} \in \cdot \, | \, H_G < H_\Delta)$ and $\Pbf^N_{X_{R_i}} (X_{H_G} \in \cdot \, | \, H_G < H_\Delta)$ --- this imposes that $X$ and $\tilde X$ \textit{both} reach $G$. Then:
			\begin{itemize}
				\item If $\xi^2_i = 0$, we independently sample $X_{D_i}$, $\tilde{X}_{D_i}$ according to the respective probabilities $\bar{\nu}_{U_i}$, $\bar{\nu}_{\tilde{U}_i}$, where $\bar \nu_z = \nu_z/\nu_z(\partial \Delta)$  with $\nu_z:= \Pbf^N_z (X_{H_\Delta} \in \cdot) - \mu_N$ as defined above.
				
				\item If $\xi^2_i = 1$, we sample $X_{D_i} = \tilde{X}_{D_i}$ according to $\bar{\mu}_N = \mu_N/\mu_N(\partial \Delta)$.
			\end{itemize}
		\end{itemize}
After that, if $X_{D_i} \neq \tilde{X}_{D_i}$, we independently sample $X_{R_{i+1}}$ and $\tilde{X}_{R_{i+1}}$ according to $\Pbf^N_{X_{D_i}} (X_{H_B} \in \cdot)$ and $\Pbf^N_{\tilde{X}_{D_i}} (X_{H_B} \in \cdot)$ respectively;
if on the other hand $X_{D_i} = \tilde{X}_{D_i}$, then $X_{R_{i+1}}=\tilde{X}_{R_{i+1}}$, and in practice $Y_j=\tilde{Y}_j$ for all $j \geq i+1$.

To summarise, under this coupling, the set $\mathset{ i \geq 1 \, | \, \xi^1_i = 1}$ collects $i$'s for which both $X$ and $\tilde X$ hit $G$ before touching $\Delta$. For such $i$'s, $\xi^2_i$ indicates if we force $X$ and $\tilde X$ to attain $\Delta$ on the same site. The coalescence time is thus no larger than the first $i \geq 1$ for which $\xi^1_i = \xi^2_i = 1$.	
In particular, under this coupling, that we denote by $\mathbf{Q}$, the probability of $Y$ and $\tilde{Y}$ not coalescing before step $k$ can be bounded from above by
		\[ \mathbf{Q}(Y_k \neq \tilde{Y}_k) \leq \mathbf{Q}\big(\forall i \leq k, \xi^1_i \xi^2_i = 0 \big) \leq \big( 1 - c_\mu c N^{\gamma-1} \big)^k \leq e^{-c' k N^{\gamma-1}} \, , \]
		in which $c'$ does not depend on $N$. Therefore, taking $k \geq \tfrac{1}{c'}N^{1-\gamma} \log 4$, the right-hand side is less than $\tfrac14$, thus proving the upper bound for $T^Y_{\mix}$ using \eqref{eq:borne-tps-mixage}.

	\subsubsection{Upper bound on the mixing time of interlacements}
	
	To get a coalescing coupling in the case of tilted interlacements, recall that $Z$ is given by excursions~$w_i^j$ of independent tilted RW trajectories $w^j$ starting from $\bar{e}_B$. In the same way, a copy $\tilde{Z}$ is given by excursions $\tilde{w}_i^j$ and trajectories $\tilde{w}^j$ starting from $\bar{e}_B$.
	We will couple two copies $Z$ and $\tilde{Z}$ in a way to have some $k$ such that both $Z_k$ and $\tilde{Z}_k$ are given by the first excursion of a trajectory, that is by $w_0^j$ and $\tilde{w}_0^{\tilde{\jmath}}$ for some $j,\tilde{\jmath}$; put otherwise, if $X$ and $\tilde{X}$ are the underlying tilted RW trajectories, we will force $X,\tilde X$ to escape to infinity for the same $(k-1)$-th excursion. By \eqref{eq:borne-inf-proba-Delta-infini}, escaping to infinity when starting from $z \in \partial \Delta$ is a Bernoulli trial with success probability $\Pbf^{\Psi_N}_z (\bar{H}_B = +\infty) \geq c N^{\gamma - 1} = p_N$. Using the same trial for both $w_i^j$ and $\tilde{w}_i^j$ thus allows for a quick coupling.
	\smallskip
	
Again, let us write the coupling in more details. 
We write $Z_{i}=(X_{R_i},X_{D_i})$ and $\tilde Z_{i} = (\tilde X_{R_i},\tilde X_{D_i})$. Let $(\xi_i)_{i \geq 1}$ be a sequence of i.i.d Bernoulli random variables with parameter $p_N$. The coupling of $Z$ and $\tilde{Z}$ proceeds as follows: for all $i \geq 1$,
		\begin{itemize}
			\item If $\xi_i = 0$, independently sample $Z_{R_{i+1}}$ according to $\nu_{Z_{D_i}}$ and $\tilde{Z}_{R_{i+1}}$ with $\nu_{\tilde{Z}_{D_i}}$, where for $w\in \partial \Delta$, the law $\nu_w$ is defined by: for any $z \in \partial B$
			\[ (1-p_N) \nu_w(z) = \Pbf^{\Psi_N}_{w} (X_{H_B} = z \, , \, \bar{H}_B < +\infty) + \left[ \Pbf^{\Psi_N}_w (\bar{H}_B = +\infty) - p_N \right] \bar{e}^T_B(z) \, . \] 
			\item If $\xi_i = 1$, sample $Z_{R_{i+1}} = \tilde{Z}_{R_{i+1}}$ according to $\bar{e}^T_B$. Afterwards, set $Z_j = \tilde{Z}_j$ for all $j > i$.
		\end{itemize}
		Under this coupling, the first $i \geq 1$ that satisfies $\xi_i = 1$ is such that $Z_{R_{i+1}} = \tilde{Z}_{R_{i+1}}$ is the entrance point in $B$ for trajectories of the two copies of the tilted RI. The probability of not having coalesced before step $k$ is then given by
		\[ \mathbf{Q}(Z_k \neq \tilde{Z}_k) \leq \mathbf{Q}\big( \forall i \leq k, \xi_i = 0 \big) = \big(1 - p_N \big)^k \leq \big(1 - c N^{\gamma-1})^k \leq e^{-c k N^{\gamma-1}} \, , \]
		where $c > 0$ is independent from $N$. Again, taking $k \geq \tfrac{1}{c}N^{1-\gamma} \log 4$ proves the upper bound for $T^Z_{\mix}$ using \eqref{eq:borne-tps-mixage}.

	\subsection{Controlling the variances for Theorem \ref{th:soft-local-times}}\label{ssec:variance-1}
	
	Recall the terms $\mathrm{Var}_{\pi^\circ}(\rho_{(x,y)}^\circ)$ that appear in the upper bound of Theorem \ref{th:soft-local-times}, where
	\[ \rho_{(x,y)}^\circ(z,w) \defeq \frac{p^\circ((z,w),(x,y))}{\mu(x,y)} \, , \]
	with $\mu$ a measure on $\Sigma$ having full support. A natural measure to take is $\mu(x,y) \defeq \Pbf^N_x(X_{H_\Delta} = y)$ since it appears in both transition probabilities $p^Y$ and $p^Z$, see~\eqref{eq:noyau-Y} and \eqref{eq:noyau-Z}. With this choice, we therefore have
	\[ \rho^Y_{(x,y)}(z,w) = \Pbf^N_w(X_{H_B} = x) \, , \]
	\[ \rho^Z_{(x,y)}(z,w) = \Pbf^{\Psi_N}_w(X_{H_B} = x, H_B < +\infty) + \bar{e}_B^T(x) \Pbf^{\Psi_N}_w(H_B = +\infty) \, . \]	
	To control the variances $\mathrm{Var}_{\tilde{\pi}}(\rho_{(x,y)}^\circ)$, we need to control these quantities. Since $\bar{e}_B^T(x)$ and $\Pbf^{\Psi_N}_w(H_B = +\infty)$ were both studied in Section \ref{ssec:propriete-entrelac-phiN}, we now seek estimates on the law of the return point to $B$.
	
	\begin{proposition}[Return point for $Y$]\label{prop:point-retour-B}
		For all $y \in \partial \Delta$ and $x \in \partial B$, we have
		\begin{equation}\label{eq:prop-point-retour-B:upper}
			\Pbf^N_y (X_{H_B} = x) \leq c N^{-\gamma (d-1)} \, .
		\end{equation}
		Moreover, for any $x \in \partial B$, there  are at least $c^{-1} N^{\gamma(d-1)}$ points $y \in \partial \Delta$ such that
		\begin{equation}\label{eq:prop-point-retour-B:lower}
			\Pbf^N_y (X_{H_B} = x) \geq c N^{-\gamma (d-1)} \, . \qedhere
		\end{equation}
	\end{proposition}

	\begin{proof}
		We first write, using~\eqref{eq:proba-htransform-RW-tuee}
		\begin{equation}\label{eq:preuve-prop-retour-B:proba}
			\Pbf^N_y (X_{H_B} = x) = \frac{\phi_N(x)}{\phi_N(y)} \mathbf{E}_y\bigg[ \lambda_N^{-H_B} \mathbbm{1}_{\big\{H_B < H_{D_N^c} ,\ X_{H_B} = x\big\}} \bigg] \, .
		\end{equation}
	
\smallskip
\noindent
\textit{Lower bound. }	
	Since $\lambda_N < 1$, we get a first lower bound using Proposition~\ref{prop:encadrement-ratio}:
		\[ \Pbf^N_y (X_{H_B} = x) \geq \kappa_1 \probaRW{y}{H_B < H_{D_N^c}, X_{H_B} = x} = \kappa_1 \Big( \Pbf_y (X_{H_B} = x) - \Pbf_y(H_B > H_{D_N^c}, X_{H_B} = x) \Big). \]
		By \cite[Lemma 5.3]{teixeiracoupling}, there are at least $c^{-1} N^{\gamma(d-1)}$ points $y \in \partial \Delta$ such that $\probaRW{y}{X_{H_B} = x} \geq c N^{-\gamma (d-1)}$. On the other hand, using the (strong) Markov property,
		\[ \probaRW{y}{H_B > H_{D_N^c}, X_{H_B} = x} = \Ebf_{y}\bigg[ \mathbbm{1}_{\big\{H_B > H_{D_N^c}\big\}} \probaRW{X_{H_{D_N^c}}}{X_{H_B} = x} \bigg] \, . \]
		Since $\partial D_N^c \subseteq (B_N^\eps)^c$, using \cite[Prop. 6.5.4]{lawlerRandomWalkModern2010} and the notation of \eqref{eq:def-cap-mes-eq-RI} we get
		\[ \sup_{z \in \partial D_N^c} \probaRW{z}{X_{H_B} = x} = \sup_{z \in \partial D_N^c} \bar{e}_B(x) \left[ 1 + c \frac{\alpha N}{d(z,B)}(1 + \bar{o}(1)) \right] \leq c \bar{e}_B(x) \leq c N^{-(d-1)} \, . \]
		In particular, since $\gamma < 1$ we get
		\[ \Pbf^N_y (X_{H_B} = x) \geq \kappa_1 \probaRW{y}{X_{H_B} = x} (1 + \bar{o}(1)) \]
		for the $c^{-1} N^{\gamma (d-1)}$ points $y \in \partial \Delta$, thus proving \eqref{eq:prop-point-retour-B:lower}.
		
\smallskip
\noindent	
\textit{Upper bound. }
	We first split $\Pbf^N_y (X_{H_B} = x)$ according to the value of $H_B$ and according to whether or not the walk reaches $\partial B^\delta$ for some $\delta \in (0, \eps)$ (whose value will be fixed later):
		\begin{equation}
			\begin{split}
				\Pbf^N_y (X_{H_B} = x) = \Pbf^N_y (H_B < N^2, X_{H_B} = x) &+ \Pbf^N_y (H_{\partial B^\delta} > H_B \geq N^2, X_{H_B} = x) \\
				&\quad + \Pbf^N_y (H_{\partial B^\delta} < H_B, H_B \geq N^2, X_{H_B} = x) \, .
			\end{split}
			\label{eq:preuve-prop-retour-B:upper}
		\end{equation}
		
		The first term can be bounded from above using~\eqref{eq:proba-htransform-RW-tuee} (similarly to \eqref{eq:preuve-prop-retour-B:proba}) together with Proposition \ref{prop:encadrement-ratio} and \eqref{eq:encadrement-lambda2}:
		\[ \Pbf^N_y (H_B < N^2, X_{H_B} = x) \leq \frac{c}{\kappa_1} e^{c_0} \probaRW{y}{H_B < H_{D_N^c} , X_{H_B} = x} \leq c N^{-\gamma(d-1)} \, , \]
		where we used \cite[Lem. 8.6]{popov2015soft} to get a bound on $\probaRW{y}{X_{H_B} = x}$ uniformly for $y\in \partial \Delta$.

		For the second term, we claim that uniformly in $x \in \partial B, y \in \partial \Delta$
		\begin{equation}\label{eq:prop:point-retour-B:claim-1}
			\Pbf^N_y (H_{\partial B^\delta} > H_B \geq N^2, X_{H_B} = x) \leq c N^{-\gamma(d-1)} \, .
		\end{equation}
		
		For the third term, we claim that uniformly in $x \in \partial B, y \in \partial \Delta$
		\begin{equation}\label{eq:prop:point-retour-B:claim-2}
			\Pbf^N_y (H_{\partial B^\delta} < H_B, H_B \geq N^2, X_{H_B} = x) \leq c N^{\gamma-1} \sup_{z \in \partial \Delta} \Pbf^N_z (X_{H_B} = x) \, .
		\end{equation}
		Plugging these estimates in~\eqref{eq:preuve-prop-retour-B:upper}, we get
		\[ \sup_{x \in \partial B, y \in \partial \Delta} \Pbf^N_y (X_{H_B} = x) \leq C N^{-\gamma (d-1)} + c N^{\gamma-1} \sup_{x \in \partial B, y \in \partial \Delta} \Pbf^N_y (X_{H_B} = x) \, . \]
		Taking $N$ large enough so that $c N^{\gamma-1} < 1/2$ proves \eqref{eq:prop-point-retour-B:upper}.

		We now prove the two claims, starting with \eqref{eq:prop:point-retour-B:claim-1}. Applying \eqref{eq:encadrement-proba} with $C = B^\delta \setminus B$, we get that
		\begin{equation}
\label{eq:prop:point-retour-B:1ere-eq-claim-1}		
		\Pbf^N_y (H_{\partial B^\delta} > H_B \geq N^2, X_{H_B} = x) \leq \frac{c}{\kappa_1} \sum_{k \geq 1} e^{c_0(k+1)}P_N(\delta,k)  \, . 
		\end{equation}
	where we have set
	$P_N(\delta,k)  := \probaRW{y}{H_B \in [k,k+1[ N^2, H_{\partial B^\delta} > k N^2 \wedge H_B, X_{H_B} = x}$.
	We now treat $P_N(\delta,k)$ for every $k \geq 1$. We use the Markov property at time $\tfrac12 k N^2$, from which two cases appear: either the walk does not return to $\Delta$ before hitting $B$, or it returns to $\Delta$ in which case since $H_B < H_{\partial B^\delta}$, the walk will cross $\partial \Delta$ again.

	By Proposition~\ref{prop:rester-ds-boule-anneau} (more precisely its second inequality) applied to $(B \cup \Delta)^c$, we have
	\begin{equation}
		\probaRW{y}{H_B \wedge H_{\partial B^\delta} \geq k N^2,  \forall\, t \in [k/2, k)N^2, X_t \in (B\cup\Delta)^c } \leq C e^{- c_a \frac12 k N^{2(1-\gamma)}} \, .
	\end{equation}
	We also get
	\begin{multline*}
		\probaRW{y}{H_B \wedge H_{\partial B^\delta} \geq k N^2, \exists t \in [0,k/2)N^2, X_t \in \Delta, X_{H_B} = x} \\ \leq \probaRW{y}{H_B \wedge H_{\partial B^\delta} \geq \tfrac12 k N^2} \sup_{w \in \Delta \cap B^\delta} \Pbf_w(H_B < H_{\partial B^\delta} , X_{H_B} = x) \, ,
	\end{multline*}
	Thus, using again Proposition~\ref{prop:rester-ds-boule-anneau} and using the Markov property at time $H_{\partial \Delta}$, we get
	\begin{equation*}
		 \probaRW{y}{H_B \wedge H_{\partial B^\delta} \geq k N^2, \exists t \in [0,k/2)N^2, X_t \in \Delta, X_{H_B} = x} \leq c e^{- c_a k/2 \delta^2} \sup_{z \in \partial \Delta} \probaRW{z}{X_{H_B} = x} \, .
	\end{equation*}
Now, with \cite[Lem. 8.6]{popov2015soft}, we have $\sup_{z \in \partial \Delta} \probaRW{z}{X_{H_B} = x} \leq c N^{-\gamma(d-1)}$ uniformly in $x \in \partial B$. Thus, injecting these estimates in \eqref{eq:prop:point-retour-B:1ere-eq-claim-1} yields
		\begin{equation}\label{eq:retour-en-B-long}
			\Pbf^N_y (H_{\partial B^\delta} > H_B \geq N^2, X_{H_B} = x) \leq \sum_{k \geq 1} e^{c_0(k+1)} \left[ c N^{-\gamma(d-1)} e^{- c_a k/2\delta^2} + C e^{- c_a \frac12 k N^{2(1-\gamma)}} \right] \, .
		\end{equation}
		Taking $\delta^2 < c_a/2c_0$ and $N > c_a c_0^{1/2(1-\gamma)}$, this proves that the right-hand side of \eqref{eq:retour-en-B-long} is bounded by a constant times $N^{-\gamma(d-1)}$, and the claim \eqref{eq:prop:point-retour-B:claim-1} follows.

		To get the second claim \eqref{eq:prop:point-retour-B:claim-2}, we forget about the condition $H_B \geq N^2$ and use the Markov property, first at time $H_{\partial B^\delta}$, then at time $H_{\partial \Delta}$: we obtain
		\[ \begin{split}
			\Pbf^N_y (H_{\partial B^\delta} < H_B, H_B \geq N^2, X_{H_B} = x) &\leq \Ebf^{N}_y \left[ \indic{H_{\partial B^\delta}  < H_B} \Ebf^{N}_{X_{H_{\partial B^\delta}}} \left[ \Pbf^N_{X_{H_{\partial \Delta}}} ( X_{H_B} = x)  \right] \right]\\
			&\leq \Pbf^N_y (H_{\partial B^\delta} < H_B) \sup_{z \in \partial \Delta} \Pbf^N_z (X_{H_B} = x) \, .
		\end{split} \]
		With \eqref{eq:borne-inf-proba-Delta-infini}, we have $\Pbf^N_y (H_{\partial B^\delta} < H_B) \leq c N^{\gamma-1}$ uniformly in $y \in \partial \Delta$, which completes the proof of~\eqref{eq:prop:point-retour-B:claim-2}.
	\end{proof}

	\begin{corollary}[Return point for $Z$]
	\label{cor:point-retour-entrelac}
		Proposition \ref{prop:point-retour-B} holds for $\Pbf^{\Psi_N}( \, \cdot \, \cap \mathset{H_B < +\infty})$ instead of $\Pbf^N( \, \cdot \, )$.
	\end{corollary}
	
	\begin{proof}
For the lower bound, observe that
		\[ 
		\Pbf^{\Psi_N}_y(X_{H_B} = x)  \geq  \Pbf^{\Psi_N}_y(H_B < H_{\partial B^\eps}, X_{H_B} = x)   = \Pbf^N_y(H_B < H_{\partial B^\eps}, X_{H_B} = x) \, , 
		\]
		since $\Pbf^{\Psi_N}_y$ and $\Pbf^N_y$ coincide on $B^\eps$. With the lower bound \eqref{eq:encadrement-proba} we get
		\begin{multline*}
\Pbf^N_y(H_B < H_{\partial B^\eps}, X_{H_B} = x)  \geq \kappa_1 \Pbf_y(H_B < H_{\partial B^\eps}, X_{H_B} = x) \\
 = \kappa_1 (\Pbf_y(X_{H_B} = x) - \Pbf_y(H_B > H_{\partial B^\eps}, X_{H_B} = x)) \, .
		\end{multline*}
		Similarly as in the proof of the lower bound in Proposition \ref{prop:point-retour-B} we get the desired bound.
		
\smallskip
For the upper bound, we decompose the probability according to whether $H_B < H_{\partial B^{\eps}}$ or $H_B > H_{\partial B^{\eps}}$.
First, notice that 
		\[ \Pbf^{\Psi_N}_y(H_B < H_{\partial B^\eps}, X_{H_B} = x) = \Pbf^N_y(H_B < H_{\partial B^\eps}, X_{H_B} = x) \leq \Pbf^N_y(X_{H_B} = x) \leq c N^{-\gamma(d-1)} \, , \]
		where we have used \eqref{eq:prop-point-retour-B:upper}.
On the other hand, in order to avoid managing the effect of $\Psi_N$ jumping to $1$ at $\partial B^\eps$, we choose $\delta \in (0,\eps)$. Then, using the Markov property at times $H_{\partial B^\delta}$ and $H_{\partial \Delta} \circ \theta_{H_{\partial B^\delta}}$ yields
		\[ \Pbf^{\Psi_N}_y(H_B > H_{\partial B^\eps}, X_{H_B} = x) \leq \Pbf^{\Psi_N}_y(H_B > H_{\partial B^\delta}) \sup_{w \in \partial \Delta} \Pbf^{\Psi_N}_w(X_{H_B} = x) \, . \]
		Since $B^\delta \subset B^\eps$, we have $\Pbf^{\Psi_N}_y(H_B > H_{\partial B^\delta}) = \Pbf^N_y(H_B > H_{\partial B^\delta})$ (recall the comment under \eqref{eq:kernel-RW-tilt}): thus, Lemma \ref{lem:toucher-B-avant-Beps-start-Delta} (replacing $\eps$ with $\delta$) implies that $\Pbf^{\Psi_N}_y(H_B > H_{B^\delta}) \asymp N^{\gamma - 1}$. Assembling the above estimates therefore yields
		\[ \Pbf^{\Psi_N}_y(X_{H_B} = x) \leq C N^{-\gamma(d-1)} + c N^{\gamma-1} \sup_{w \in \partial \Delta} \Pbf^{\Psi_N}_w(X_{H_B} = x) \, . \]
Then, taking $N$ large enough so that $c N^{\gamma-1}\leq 1/2$ gives the result.
	\end{proof}
	
	Recall the expressions for $\rho^\circ_{(x,y)}(z,w)$ at the beginning of this section
	\[ \begin{split}
		\rho^Y_{(x,y)}(z,w) &= \Pbf^N_w(X_{H_B} = x) \, , \\
		\rho^Z_{(x,y)}(z,w) &= \Pbf^{\Psi_N}_w(X_{H_B} = x, H_B < +\infty) + \bar{e}_B^{\Psi_N}(x) \Pbf^{\Psi_N}_w(H_B = +\infty) \, .
	\end{split} \]
	Combining Proposition \ref{prop:point-retour-B}, Corollary \ref{cor:point-retour-entrelac}, Lemma \ref{lem:encadrement-mes-harm-tilt} and Lemma \ref{lem:toucher-B-avant-Beps-start-Delta} we get the following upper bounds for $N$ large enough: uniformly for $(x,y),(z,w)\in \Sigma$,
	\begin{equation}\label{eq:rho-upper-bound}
		\rho^Y_{(x,y)}(z,w) \leq cN^{-\gamma(d-1)} \, , \qquad \rho^Z_{(x,y)}(z,w) \leq cN^{-\gamma(d-1)} + c N^{1-d} N^{\gamma - 1} \leq c' N^{-\gamma(d-1)} \, .
	\end{equation}
	For lower bounds, we can use again Proposition \ref{prop:point-retour-B} and Corollary \ref{cor:point-retour-entrelac} to get that, for every $x\in \partial B$, there are at least $c^{-1} N^{\gamma(d-1)}$ points $w\in\partial \Delta$ such that
	\begin{equation}\label{eq:rho-lower-bound}
		\rho^\circ_{(x,y)}(z,w) \geq cN^{-\gamma(d-1)} \, .
	\end{equation}
	
	As announced at the beginning of this section, we use these bounds to control the variances.
	
	\begin{proposition}\label{prop:variance-1}
		There is some constant $c>0$ such that uniformly in $x \in \partial B$ and $y \in \partial \Delta$, we have
		\begin{equation*}
		c^{-1} N^{-(\gamma+1)(d-1)} \leq \mathrm{Var}_{\tilde{\pi}_\circ}(\rho_{(x,y)}^\circ) \leq c N^{-(\gamma+1)(d-1)} \, .
		\end{equation*}
	\end{proposition}

	\begin{proof}
		According to \eqref{eq:erreur-mes-invariante}, we only need to prove the lemma for $\tilde{\pi} = \tilde{\pi}^Y$. Recall that $\Sigma = \partial B \times \partial \Delta$. In the proof, we write for simplicity
		\[
		\tilde{\pi}\big[\rho_{(x,y)}^\circ \big] \defeq \sum_{(z,w) \in \Sigma}  \tilde{\pi}(z,w)  \rho^\circ_{(x,y)}(z,w)\,,\qquad \tilde{\pi}\Big[ \big(\rho_{(x,y)}^{\circ}\big)^2 \Big] \defeq	\sum_{(z,w) \in \Sigma} \tilde{\pi}(z,w) \rho^\circ_{(x,y)}(z,w)^2 \,.
		\]
For the upper bound, we simply write	
		\[ \mathrm{Var}_{\tilde{\pi}}(\rho_{(x,y)}^\circ) \leq \tilde{\pi}\Big[ \big(\rho_{(x,y)}^{\circ}\big)^2 \Big]  \leq  cN^{-\gamma(d-1)}\tilde{\pi}(\rho_{(x,y)}^\circ),  \]
where we have used~\eqref{eq:rho-upper-bound} for the second inequality.
Then, notice that, by definition of $\rho^{\circ}$ and since $\tilde{\pi}$ is the invariant measure of $Y$, 
		\begin{equation}\label{eq:integral-rho-sur-pi}
		\tilde{\pi}\big(\rho_{(x,y)}^\circ\big) = \frac{1}{\Pbf^N_x(X_{H_\Delta} = y)} \sum_{(z,w)} p^\circ((z,w),(x,y)) \tilde{\pi}(z,w) = \frac{\tilde{\pi}(x,y)}{\Pbf^N_x(X_{H_\Delta} = y)} = \tilde{e}_\Delta^{\phi_N}(x) \, ,
		\end{equation}
where the last identity comes from Lemma~\ref{lem:mes-inv-Y}.
Then, we can use Lemma \ref{lem:encadrement-mes-harmonique} to obtain that $\tilde{\pi}(\rho_{(x,y)}^\circ) =\tilde{e}_\Delta^{\phi_N}(x) \leq c N^{-(d-1)}$, which concludes the proof of the upper bound.
 
%		To bound the sum in \eqref{eq:variance-leq-l2-rho}, notice that
%		\[ N^{1-d} \sum_{w \in \partial \Delta} \rho^\circ_{(x,y)}(z,w) \asymp \sum_{w \in \partial \Delta} \rho^\circ_{(x,y)}(z,w) {\tilde{\pi}}(\partial B \times \mathset{w}) = \tilde{e}_\Delta^{\phi_N}(x) \asymp N^{1-d} \, . \]
%		Thus, there is a constant $K$ that does not depend on $N$ such that $\sum_{w \in \partial \Delta} \rho^\circ_{(x,y)}(z,w) \leq K$, which proves the upper bound on the variance.
		
		To get a lower bound on the second moment, we write
		\begin{equation*}\label{eq:moment-2-lower-bound-1}
	\tilde{\pi}\Big[ \big(\rho_{(x,y)}^{\circ}\big)^2 \Big]  \geq c N^{-2\gamma(d-1)} \tilde{\pi}\Big(\mathset{(z,w) \in \Sigma \, : \, \rho^\circ_{(x,y)}(z,w) \geq c N^{-\gamma(d-1)}} \Big) \, .
		\end{equation*}
		Using Lemma \ref{lem:encadrement-sortie-Delta} and \eqref{eq:rho-lower-bound}, we get (note that $\rho^\circ_{(x,y)}(z,w)$ does not depend on $z$)
		\begin{equation*}
		\tilde{\pi}\Big(\mathset{(z,w) \in \Sigma \, : \, \rho^\circ_{(x,y)}(z,w) \geq c N^{-\gamma(d-1)}} \Big) \geq c N^{1-d} c' N^{\gamma(d-1)}\, ,
		\end{equation*}
	so that 
		\[
		\tilde{\pi}\Big[ \big(\rho_{(x,y)}^{\circ}\big)^2 \Big]  \geq c   N^{- (1+\gamma)(d-1)}\,.
		\]
On the other hand, using again \eqref{eq:integral-rho-sur-pi} and Lemma \ref{lem:encadrement-mes-harmonique}, we get $\tilde{\pi}\big(\rho_{(x,y)}^\circ\big)^2 \leq c N^{-2(d-1)} = \bar{o}(N^{- (1+\gamma)(d-1)})$.
All together, we end up with
		\[ \mathrm{Var}_{\tilde{\pi}}(\rho_{(x,y)}^\circ) = \tilde{\pi}\Big[ \big(\rho_{(x,y)}^{\circ}\big)^2 \Big]  -  \tilde{\pi}\big[\rho_{(x,y)}^\circ\big]^2  \geq (1-\bar{o}(1)) c N^{- (1+\gamma)(d-1)}   \, , \]
which concludes the proof of the lower bound.
	\end{proof}

	\section{Coupling between the confined RW and tilted RI}\label{ssec:couplage-1}
	
	In this section, we prove Theorem \ref{th:couplage-CRW-entrelac-t_N} in three steps. First we control the number of excursions, \textit{i.e.}\ we estimate how many times the confined RW enters $B$ and leaves through $\Delta$ before time $t_N$, and similarly for the tilted RI $\mathscr{I}_{\Psi_N}(u_N)$. Then, we use Theorem \ref{th:soft-local-times} to couple the Markov chains~$Y$ and~$Z$ defined in Section~\ref{ssec:defYZ}. Finally, we detail how to couple the actual excursions in~$B$ to get Theorem \ref{th:couplage-CRW-entrelac-t_N}.
	
	Let us first recall some notation. We fix a $\delta > 0$ and we consider a sequence of times $(t_N)_{N \geq 1}$ that satisfies $t_N / N^{2+\delta} \to +\infty$. We also consider a vanishing sequence $(\eps_N)_{N \geq 1}$ defined by $\eps_N = N^{-\delta/4}$.
	Recall the definitions~\eqref{eq:def-instants} of the return times $R_i$ for the $\phi_N$-walk, as well as~\eqref{eq:def-instants-entrelac} for the corresponding ``returns'' $R_i^j$ to $B$ for the interlacements.

	\subsection{Number of excursions}\label{ssec:nb-excursions}

	\subsubsection{For the confined RW}	
	Write $\mathcal{N}(t) \defeq \sup \mathset{i \geq 1 \, : \, R_i < t}$ for the number of excursions in $B$ before time $t$ made by the random walk under $\Pbf^N_{\phi_N^2}$.
	
	\begin{proposition}\label{prop:nb-excursions-tN}
		There exist positive constants $c, c', C$ depending only on $d,\gamma$ such that for every $N \geq 1$.
		\begin{equation}\label{prop:nb-excursions-tN:eq}
			\Pbf^N_{\phi_N^2} \left(\left|\mathcal{N}(t_N) - t_N \frac{\cpc_\Delta^{\phi_N}(B)}{\| \phi_N\|_2^2} \right| > \eps_N t_N \frac{\cpc_\Delta^{\phi_N}(B)}{\| \phi_N\|_2^2} \right) \leq C \exp (- c' N^c) \, .
		\end{equation}
	\end{proposition}
	
	Let us first give an expression for the expectation of $\mathcal{N}(t)$.
	
	\begin{lemma}\label{lem:esp-nb-retour-walk}
		For any $t \geq 1$,
		\begin{equation}
			\Ebf_{\phi_N^2}^{N} \left[ \mathcal{N}(t) \right] = t \, \frac{\cpc_\Delta^{\phi_N}(B)}{\| \phi_N\|_2^2} \, .
		\end{equation}
	\end{lemma}
	
	\begin{proof}
		This follows from \eqref{eq:proba-retour-temps-site} and the definition of $\cpc_\Delta^{\phi_N}(B)$ in~\eqref{eq:def-mes-harm-capacite-phi}: we have that
		\[
		\Ebf_{\phi_N^2}^{N}[\mathcal{N}(t)] = \sum_{k = 1}^t \Pbf^N_{\phi_N^2} (k \in \mathcal{R}) = \sum_{k = 1}^t \sum_{x \in \partial B} \frac{\phi_N^2(x)}{\| \phi_N\|_2^2} \Pbf^N_x(\bar{H}_B > H_\Delta) = t \frac{\cpc_\Delta^{\phi_N}(B)}{\| \phi_N\|_2^2} \, . \qedhere
		\]
	\end{proof}

	\begin{proof}[Proof of Proposition \ref{prop:nb-excursions-tN}]
		Note that for any $t > 0$ and $b > 0$
		\begin{equation*}
			\mathset{ \big| \mathcal{N}(t) - \Ebf_{\phi_N^2}^{N}\left[ \mathcal{N}(t) \right] \big| > b} \subseteq \Big\{ R_{\lceil \Ebf_{\phi_N^2}^{N}\left[ \mathcal{N}(t) \right] - b \rceil}  >t \Big\} \cup \Big\{ R_{\lfloor \Ebf_{\phi_N^2}^{N}\left[ \mathcal{N}(t) \right] + b \rfloor} < t \Big\} \, .
		\end{equation*}
		In particular, if we write $k_\pm = \lceil (1 \pm \eps_N) t_N \| \phi_N\|_2^{-2} \cpc_\Delta^{\phi_N}(B) \rceil$, we have the upper bound
		\begin{equation}\label{eq:nb-traj-RW-Rk}
			\Pbf^N_{\phi_N^2} \left(\left|\mathcal{N}(t_N) - t_N \frac{\cpc_\Delta^{\phi_N}(B)}{\| \phi_N\|_2^2} \right| >  \eps_N t_N \frac{\cpc_\Delta^{\phi_N}(B)}{\| \phi_N\|_2^2} \right) \leq \Pbf^N_{\phi_N^2} (R_{k_-} > t_N) + \Pbf^N_{\phi_N^2} (R_{k_+} < t_N) \, .
		\end{equation}
		
		In order to get a bound on \eqref{eq:nb-traj-RW-Rk}, we  control the deviations of $R_{k_\pm}$. Indeed, we claim that $\Ebf^N_{\phi_N^2}[R_{k_\pm}] \approx (1 \pm \eps_N) t_N$: we postpone its proof until the end of the section.
		\begin{claim}\label{claim:esp-temps-k-retour}
			There is a constant $c>0$ such that, for $N$ large enough,
			\[
			\Ebf^N_{\phi_N^2}[R_{k_{\pm}}] = k_{\pm} \frac{\| \phi_N\|_2^2 } {\cpc_\Delta^{\phi_N}(B)} + \grdO(N^{-c}) \eps_N t_N \,.
			\]
			We also have the following convergences:
			\begin{equation}\label{eq:esp-m-ell-retour}
				\lim_{N \to +\infty} \frac{\Ebf^N_{\phi_N^2}[R_{k_\pm}]}{(1 \pm \eps_N) t_N} = \lim_{N \to +\infty}  \frac{\cpc_\Delta^{\phi_N}(B)}{\| \phi_N\|_2^2 k_\pm} \Ebf^N_{\phi_N^2}[R_{k_\pm}]  = 1 \, .
			\end{equation}
		\end{claim}
		
		As a consequence of Claim \ref{claim:esp-temps-k-retour}, we have that $\Ebf^N_{\phi_N^2}[R_{k_{-}}]  \leq t_N - \frac12 \eps_N t_N$ and $\Ebf^N_{\phi_N^2}[R_{k_{+}}]  \geq t_N + \frac12 \eps_N t_N$ for $N$ large enough.
		Using \eqref{eq:nb-traj-RW-Rk} to prove Proposition~\ref{prop:nb-excursions-tN}, we need to show that 
		\begin{equation}
			\label{prop:nb-excursions-tN:eq2}
			\Pbf^N_{\phi_N^2} \Big( \big| R_{k_\pm} - \Ebf^N_{\phi_N^2}[R_{k_\pm}] \big| > \tfrac12 \eps_N \Ebf^N_{\phi_N^2}[R_{k_\pm}] \Big) \leq C \exp(- c'N^{c''}) \,.
		\end{equation}
		
		%QB		applied to $\eta = \eps_N$ {\blue (QB: $\eta$ n'est pas un bon choix parce que c'est dans la définition de $k_{\pm}$)}, which requires an exponential control of the deviations of $R_{k_\pm}$.		
		
		The way of controlling these deviations is through the Azuma inequality, see e.g.\ \cite[Ch.~1.1]{inegalite-concentration}. We write $R_{k_\pm}$ as a sum of time increments that are roughly i.i.d.	Let $\varsigma > 0$ and recall the definition \eqref{eq:def-temps-melange} of the mixing time $T^Y_{\mix}$. Set $\ell = N^\varsigma T^Y_{\mix}$ and $m_\pm = k_\pm \ell^{-1}$, and define $\mathscr{G}_i$ the sigma-algebra associated to $R_{i\ell}$. With the strong Markov property and mixing properties (see \cite[Section 4.5]{levin2017markov}), we have the bound
		\begin{equation}
			\| \Pbf^N_{\phi_N^2} \left((X_{R_{i\ell}}, X_{D_{i\ell}}) \in \cdot \, | \, \mathscr{G}_{i-1}\right) - \tilde{\pi}_Y(\cdot) \|_{\mathrm{TV}} \leq 2^{-N^\varsigma} \, .
		\end{equation}
		Summing over the value of $X_{D_{i\ell}}$ and since $\tilde{\pi}_Y(\mathset{x} \times \partial \Delta) = \tilde{e}_\Delta^{\phi_N}(x) \asymp N^{1-d}$ uniformly in $x \in \partial B$ (recall Lemma \ref{lem:encadrement-mes-harmonique}), we obtain
		\begin{equation}\label{eq:loi-retour-sachant-Gj}
			\sup_{x\in\partial B} \left| 1 - \frac{\Pbf^N_{\phi_N^2} (X_{R_{i\ell}} = x \, | \, \mathscr{G}_{i-1})}{\tilde{e}_\Delta^{\phi_N}(x)} \right| \leq c 2^{-N^{\varsigma/2}} \, .
		\end{equation}

		This tells us that after $\ell$ returns to $B$, the starting points of the next sequence of $\ell$ returns has a distribution very close to $\tilde{e}_\Delta^{\phi_N}$. This helps controlling the time spent between $R_{(i-1)\ell}$ and $R_{i\ell}$, which is useful to apply the Azuma--Hoeffding inequality.
		Write $R_{m\ell} = \sum_{j = 1}^m W_j$ with $W_j = R_{j\ell} - R_{(j-1)\ell}$ and $R_0 = 0$. Then, writing
		\[
		\Pbf^N_{\phi_N^2} (W_j > t \, | \, \mathscr{G}_{j-2}) = \sum_{x \in \partial B} \Pbf^N_x (R_{\ell} > t) \Pbf^N_{\phi_N^2} (X_{R_{(j-1)\ell}}  = x \, | \, \mathscr{G}_{j-2})
		\]
		and using \eqref{eq:loi-retour-sachant-Gj}, we have
		\begin{equation}\label{eq:proba-cond-Wj-geq-t}
			\big| \Pbf^N_{\phi_N^2} (W_j > t \, | \, \mathscr{G}_{j-2}) - \Pbf^N_{\tilde{e}_\Delta^{\phi_N}} (R_{\ell} > t) \big|  \leq c2^{-N^{\varsigma/2}}\Pbf^N_{\tilde{e}_\Delta^{\phi_N}} (R_{\ell} > t)  \,.
		\end{equation}
		
		Note that using~\eqref{eq:proba-cond-Wj-geq-t} and also $\Ebf^N_{\phi_N^2} [W_j] = \Ebf^N_{\phi_N^2} [\Ebf^N_{\phi_N^2} [W_j \, | \, \mathscr{G}_{j-2}]]$, we have
		\begin{equation}\label{eq:controle-esp-Zj}
			\left| \Ebf^N_{\phi_N^2} [W_j \, | \, \mathscr{G}_{j-2}] - \Ebf^N_{\tilde{e}_\Delta^{\phi_N}} [R_{\ell}] \right| \,,\, \left| \Ebf^N_{\phi_N^2} [W_j] - \Ebf^N_{\tilde{e}_\Delta^{\phi_N}} [R_{\ell}] \right|  \leq c 2^{-N^{\varsigma/2}} \Ebf^N_{\tilde{e}_\Delta^{\phi_N}} [R_{\ell}] \, .
		\end{equation}	
		%QB: Je ne suis pas sûr que ça soit vraiment utile
		%Moreover, by the Markov property at time $R_{(j-1)\ell}$ and using again~\eqref{eq:loi-retour-sachant-Gj}, we have
		%		\begin{equation}\label{eq:esp-Zj-Rl}
			%			\Ebf^N_{\phi_N^2} [W_j] = \Ebf^N_{\phi_N^2} [\Ebf^N_{X_{R_{(j-1)\ell}}}[R_\ell]] \geq (1 - c2^{-N^{\varsigma/2}}) \Ebf^N_{\tilde{e}_\Delta^{\phi_N}}[R_\ell] \, .
			%		\end{equation}
		Therefore, 
		%combining \eqref{eq:controle-esp-Zj} and \eqref{eq:esp-Zj-Rl}, 
		we also have
		$\left| \Ebf^N_{\phi_N^2} [W_j] - \Ebf^N_{\phi_N^2} [W_j \, | \, \mathscr{G}_{j-2}] \right| \leq c 2^{-N^{\varsigma/2}} \Ebf^N_{\phi_N^2} [W_j]$,
		and in particular
		\begin{equation}\label{eq:Wj-esp-vs-esp-cond}
			\bigg| \sum_{j = 1}^m (W_j - \Ebf^N_{\phi_N^2}[W_j]) - \sum_{j = 1}^m (W_j - \Ebf^N_{\phi_N^2}[W_j \, | \, \mathscr{G}_{j-2}]) \bigg| \leq c 2^{-N^{\varsigma/2}} \Ebf^N_{\phi_N^2} [R_{m\ell}] \, .
		\end{equation}
		
		With these estimates, we will get a bound on the deviations of $R_{m\ell}$, with $m = m_+$ or $m = m_-$. Writing $R_{m \ell} = \sum_{j = 1}^m W_j$ and separating even and odd $j$'s (and $j = 1$), using \eqref{eq:Wj-esp-vs-esp-cond} we obtain that for $N$ large enough,
		\begin{equation}\label{eq:separation-1-pairs-impairs}
			\begin{split}
				\Pbf^N_{\phi_N^2} &\Big( \big| R_{m\ell} - \Ebf^N_{\phi_N^2}[R_{m\ell}] \big| > \tfrac12\eps_N \Ebf^N_{\phi_N^2}[R_{m\ell}] \Big) = \Pbf^N_{\phi_N^2} \bigg( \Big| \sum_{j = 1}^m \big(W_j - \Ebf^N_{\phi_N^2}[W_j]\big) \Big| >  \tfrac12\eps_N  \Ebf^N_{\phi_N^2}[R_{m\ell}] \bigg) \\
				& \qquad \leq \sum_{n \in \mathset{0,1}} \Pbf^N_{\phi_N^2} \bigg( \Big| \sum_{\substack{2 \leq j \leq m\\ j = n \, \mathrm{mod}\, 2}} \big(W_j - \Ebf^N_{\phi_N^2} [W_j \, | \, \mathscr{G}_{j-2}]\big) \Big| > \tfrac14 \eps_N  \Ebf^N_{\phi_N^2}[R_{m\ell}] \bigg) \\[-5pt]
				& \hspace{8cm} + \Pbf^N_{\phi_N^2} \Big(W_1 \geq \tfrac14 \eps_N \Ebf^N_{\phi_N^2}[R_{m\ell}]\Big) \, .
			\end{split}
		\end{equation}
		
		Let us start with the last term in~\eqref{eq:separation-1-pairs-impairs}.
		Using a union bound, we have
		\begin{equation}\label{eq:proba-W_1-grand}
			\Pbf^N_{\phi_N^2} (W_1 > t) = \Pbf^N_{\phi_N^2} (R_{\ell} > t) \leq \ell \Pbf^N_{\phi_N^2} (R_1 > t/\ell) \, .
		\end{equation}
		We now control~\eqref{eq:proba-W_1-grand} thanks to the following claim, whose proof is postponed.
		
		\begin{claim}\label{claim:proba-tps-retour-grand}
			There are constants $C, c' > 0$ such that for all $a > 0$ and $s \geq 1$,
			\[\Pbf^N_{\phi_N^2} (R_1 > s N^{2+a}) \leq C e^{-c's N^{a}} \,,\qquad  \Pbf^N_{\tilde{e}_\Delta^{\phi_N}} (R_1 > s N^{2+a}) \leq C e^{- c's N^{a}}   \, . \]
		\end{claim}
		
		Recall that Claim~\ref{claim:esp-temps-k-retour} states that $\Ebf^N_{\phi_N^2}[R_{\ell m_{\pm}}] \sim (1 \pm \eps_N) t_N$, we claim that the last term in \eqref{eq:separation-1-pairs-impairs} is bounded as follows: there exists some $a > 0$ such that
		\begin{equation}\label{eq:borne-premier-retour-grape}
			\Pbf^N_{\phi_N^2} \Big(W_1 \geq \tfrac14 \eps_N \Ebf^N_{\phi_N^2}[R_{m\ell}]\Big) \leq \ell \Pbf^N_{\phi_N^2} \big( R_1 > \tfrac18 \eps_N t_N \ell^{-1} \big) \leq \ell e^{- c'N^{ a}}  \, .
		\end{equation}
		To get \eqref{eq:borne-premier-retour-grape}, we have used Claim \ref{claim:proba-tps-retour-grand} for the last inequality, recalling that $\ell \defeq N^{\varsigma} T^Y_{\mix} \leq c N^{\varsigma + \frac{\delta}{4(d-2)}}$ (recall Proposition~\ref{prop:temps-mixage} and the assumptions on $\gamma$) and that $\eps_N t_N \geq N^{2+\delta(1 - \frac{1}{4})}$, so that $\tfrac18 \eps_N t_N \ell^{-1} \geq N^{2+a}$ provided that $a>0$ is small enough (and that we have fixed $\varsigma$ small enough) as well as $N$ large enough.

		We now turn to the first term in~\eqref{eq:separation-1-pairs-impairs}. In order to apply the Azuma-Hoeffding inequality to the first term in \eqref{eq:separation-1-pairs-impairs}, our first task is to bound the increments appearing in the sum inside the probability.	
		
		First of all, notice that using~\eqref{eq:proba-cond-Wj-geq-t} and a union bound, we get similarly to~\eqref{eq:proba-W_1-grand}	that
		\begin{equation}\label{eq:proba-W_j-grand}
			\Pbf^N_{\phi_N^2} (W_j > t \, | \, \mathscr{G}_{j-2}) \leq (1 + c2^{-N^{\varsigma/2}}) \Pbf^N_{\tilde{e}_\Delta^{\phi_N}} (R_{\ell} > t) \leq 2 \ell \Pbf^N_{\tilde{e}_\Delta^{\phi_N}} (R_1 > t/\ell) \, .
		\end{equation}			
		Then, applying~\eqref{eq:proba-W_j-grand} with Claim~\ref{claim:proba-tps-retour-grand}, we get that $\Pbf^N_{\phi_N^2} (W_j > \ell N^{2+a}) \leq \ell \exp(-c'N^{a})$, so that using a union bound for the event that one of the $W_j$'s is larger than $\ell N^{2+a}$, we get, for $n\in\{0,1\}$
		\begin{align}\label{eq:restriction-hat-Wj}
			&\Pbf^N_{\phi_N^2} \bigg( \Big| \sum_{\substack{1 \leq j \leq m\\ j = n \, \mathrm{mod}\, 2}} \big(W_j - \Ebf^N_{\phi_N^2} [W_j \, | \, \mathscr{G}_{j-2}] \big) \Big| > \tfrac{\eps_N}{4} \Ebf^N_{\phi_N^2}[R_{m\ell}] \bigg) \\  &\leq m\ell  e^{-c' N^{a}} + \Pbf^N_{\phi_N^2} \bigg( \Big| \sum_{\substack{2 \leq j \leq m\\ j = n \, \mathrm{mod}\, 2}} \big(W_j - \Ebf^N_{\phi_N^2} [W_j \, | \, \mathscr{G}_{j-2}]\big) \Big| > \tfrac{\eps_N}{4} \Ebf^N_{\phi_N^2}[R_{m\ell}] \, , \, \max_{2\leq j \leq n} W_j \leq \ell N^{2+a} \bigg) \, . \notag
		\end{align}
		For $j \in \mathset{2, \dots , m}$, we define $\hat{W}_j \defeq W_j \wedge \ell N^{2+a}$ and note that we have
		\begin{equation*}
			\begin{split}
				\left| \Ebf^N_{\phi_N^2} [W_j \, | \, \mathscr{G}_{j-2}] - \Ebf^N_{\phi_N^2} [\hat{W}_j \, | \, \mathscr{G}_{j-2}] \right| &= \Ebf^N_{\phi_N^2} [W_j \indic{W_j > \ell N^{2+a}} \, | \, \mathscr{G}_{j-2}] \\
				&\hspace{-1cm} = \int_{\ell N^{2+a}}^{+\infty} \Pbf^N_{\phi_N^2} (W_j > t \, | \, \mathscr{G}_{j-2}) \, dt  \,  + \ell N^{2+a}\Pbf^N_{\phi_N^2}(W_j >\ell N^{2+a}\, | \, \mathscr{G}_{j-2}).
			\end{split}
		\end{equation*}
		Let us deal with the first term (the second term is dealt with in a similar fashion). Using~\eqref{eq:proba-W_j-grand} and then Claim \ref{claim:proba-tps-retour-grand}, we get that
		\[
		\int_{\ell N^{2+a}}^{+\infty} \Pbf^N_{\phi_N^2} (W_j > t \, | \, \mathscr{G}_{j-2}) \, dt \leq 2\ell \, (\ell N^{2+a}) \int_{1}^{+\infty} \Pbf^N_{\tilde{e}_\Delta^{\phi_N}}(R_1 > u N^{2+a} ) \, du
		\leq C' e^{- c' N^{a}} \,.
		\]

		Since $e^{-c' N^{ a}} = \bar{o}(\eps_N t_N)$, we obtain thanks to Claim \ref{claim:esp-temps-k-retour} that $C' e^{- c'N^{ a}} \leq \tfrac{\eps_N}{8} \Ebf^N_{\phi_N^2}[R_{m\ell}]$ for~$N$ large enough.
		Therefore, we conclude that the last term in the right-hand side of \eqref{eq:restriction-hat-Wj} is bounded from above by
		\begin{equation}\label{eq:azuma-proba-a-borner}
			\Pbf^N_{\phi_N^2} \bigg( \Big| \sum_{\substack{2 \leq j \leq m\\ j = n \, \mathrm{mod}\, 2}} (\hat{W}_j - \Ebf^N_{\phi_N^2} [\hat{W}_j \, | \, \mathscr{G}_{j-2}]) \Big| > \tfrac{\eps_N}{8} \Ebf^N_{\phi_N^2}[R_{m\ell}] \bigg) \, .
		\end{equation}
		We are now in the position of using Azuma--Hoeffding inequality to bound the probability in \eqref{eq:azuma-proba-a-borner}. Since $|\hat{W}_j - \Ebf^N_{\phi_N^2} [\hat{W}_j \, | \, \mathscr{G}_{j-2}]| \leq 2 \ell N^{2+a}$ by definition, we have
		\begin{equation}\label{eq:azuma-1ere-apply}
			\Pbf^N_{\phi_N^2} \bigg( \Big| \sum_{\substack{2 \leq j \leq m\\ j = n \, \mathrm{mod}\, 2}} (\hat{W}_j - \Ebf^N_{\phi_N^2} [\hat{W}_j \, | \, \mathscr{G}_{j-2}]) \Big| > \tfrac{\eps_N}{8} \Ebf^N_{\phi_N^2}[R_{m\ell}] \bigg) \leq \exp \bigg( - c \frac{2 (\tfrac{\eps_N}{8} \Ebf^N_{\phi_N^2}[R_{m\ell}])^2}{m (\ell N^{2+a})^2} \bigg) \, ,
		\end{equation}
		where $c > 0$ does not depend on $N$ large enough.
		Now recall that $\Ebf^N_{\phi_N^2}[R_{m\ell}] \sim  t_N$ (see Claim~\ref{claim:esp-temps-k-retour}) and that $m \ell = k_{\pm}\sim \frac{t_N}{ \| \phi_N\|_2^2} \cpc_\Delta^{\phi_N}(B)$ by definition, with $\| \phi_N\|_2^2 = N^{d}$ (recall our choice of normalization \eqref{def:eigenfunction}) and $\cpc_\Delta^{\phi_N}(B) \asymp N^{d-1-\gamma}$  (see Lemma~\ref{lem:encadrement-mes-harmonique}): the upper bound in~\eqref{eq:azuma-1ere-apply} then becomes
		\[
		\exp \bigg( - c \frac{2 (\tfrac{\eps_N}{8} \Ebf^N_{\phi_N^2}[R_{m\ell}])^2}{m (\ell N^{2+a})^2} \bigg) 
		\leq \exp\Big( - c' \eps_N^2 t_N N^{1+\gamma} \ell^{-1} N^{-2(2+a)}\Big) \leq 
		\exp\Big( - c' N^{\frac12 \delta +2(\gamma-1) -\varsigma -2 a}\Big) \,,
		\]	
		where we have used that $\eps_N^2 t_N \geq N^{2 + \frac12 \delta}$ and $\ell \defeq N^{\varsigma} T^Y_{\mix} \leq c N^{\varsigma +1-\gamma}$ (see Proposition~\ref{prop:temps-mixage}).

		Injecting all of these estimates in \eqref{eq:restriction-hat-Wj}, we finally get
		\begin{equation}
			\Pbf^N_{\phi_N^2} \Big(|R_{m_\pm \ell} - \Ebf^N_{\phi_N^2}[R_{m_\pm\ell}]| > \eps_N \Ebf^N_{\phi_N^2}[R_{m\ell}] \Big) \leq m\ell  e^{-N^{c'a}} + \exp(-c' N^{\frac 12 \delta - 2 (1 - \gamma) -\varsigma -2 a}) \, ,
		\end{equation}
		which concludes the proof of~\eqref{prop:nb-excursions-tN:eq2}, since by assumption $2 (1 - \gamma) < \frac{\delta}{2(d-2)}$, provided that $\varsigma,a$ has been fixed close enough to $0$ (that is such that $\varsigma + 2 a < \tfrac12 \delta - 2 (1 - \gamma)$).
		%		Taking $\eta = \eps_N = N^{2+\delta'}/t_N$, rewriting $m_\pm \ell = k_\pm$ and injecting $1-\gamma < \delta'/2(d-1) \leq \delta'/4$ as well as Claim \ref{claim:esp-temps-k-retour} we finally get
		%		\[ \Pbf^N_{\phi_N^2} (|R_{k_\pm} - \Ebf^N_{\phi_N^2}[R_{k_\pm}]| > \eps_N t_N) \leq cme^{-N^{c'a}} + c' \exp \left( - c'' N^{\tfrac12 \delta' - \varsigma - 2a} \right) \, . \]
		%		Taking $a > 0$ and $\varsigma > 0$ sufficiently small, we have $\tfrac12 \delta' - \varsigma - 2a \geq \tfrac14 \delta'$, from which we deduce \eqref{prop:nb-excursions-tN:eq} thanks to \eqref{eq:nb-traj-RW-Rk}.
	\end{proof}

	\begin{proof}[Proof of Claim \ref{claim:proba-tps-retour-grand}]
		Recall that $R_1$ is the first hitting time of $B$ \underline{after} visiting $\Delta$. In particular, since $\tilde{e}_\Delta^{\phi_N}$ has support in $\partial B$, for the second part of the claim we get
		\begin{equation}\label{eq:proba-retour-grand-bord-B}
			\Pbf^N_{\tilde{e}_\Delta^{\phi_N}} (R_1 > s N^{2+a}) \leq \Pbf^N_{\tilde{e}_\Delta^{\phi_N}} (H_\Delta > \tfrac12 s N^{2+a}) + \sup_{y \in \partial \Delta} \Pbf^N_{y} (H_B > \tfrac12 s N^{2+a}) \, ,
		\end{equation}
		while for the first part, we write
		\begin{equation}\label{eq:proba-retour-grand-phi2}
			\begin{split} \Pbf^N_{\phi_N^2} (R_1 > s N^{2+a}) & \leq \frac{1}{\| \phi_N\|_2^2} \sum_{x \in D_N \setminus \Delta} \phi_N^2(x) \Big[\Pbf^N_{x} (H_\Delta > \tfrac12 s N^{2+a}) + \sup_{y \in \partial \Delta} \Pbf^N_{y} (H_B > \tfrac12 s N^{2+a}) \Big] \\
				& \hspace{4cm} + \frac{1}{\| \phi_N\|_2^2} \sum_{x \in \Delta} \phi_N^2(x) \Pbf^N_{x} (H_B > s N^{2+a}) \, .
			\end{split}
		\end{equation}
		Using the expression~\eqref{eq:def-p_N} of the transition kernel $p_N$ and Proposition \ref{prop:encadrement-ratio}, we have
		\begin{equation}\label{eq:attentionJP}
			\Pbf^N_{x} (H_\Delta > \tfrac12 s N^{2+a}) \leq \kappa_1 \sum_{k \geq 0} (\lambda_N)^{- \tfrac12 s N^{2+a}  - (k+1) N^2 } \Pbf_x \big( H_\Delta >  \tfrac12 s N^{2+a}  + k N^2 \big) \, .
		\end{equation}
		Using Proposition \ref{prop:equivalent-proba-sortie-eigen}-\eqref{rem:encadrement-proba-sortie-eigen} and since $N^{2+a} \gg N^2 \log N$, we have that 
		\[
		\Pbf_x \big( H_\Delta >  \tfrac12 s N^{2+a}  + k N^2 \big) \leq c \lambda(D_N \setminus \Delta)^{\tfrac12 s N^{2+a}  + k N^2} \, , 
		\]
		where, if $K \subset \ZZ^d$, we set $\lambda(K)$ as the first eigenvalue of the transition matrix of the simple random walk killed on $\partial K$. By eigenvalues comparisons (a consequence of the min-max theorem, see \textit{e.g.} \cite[Chap. 11, Th. 1]{strauss2007partial}) and the expression of eigenvalues on a blowup (see~\eqref{eq:vp-blowup}), we have that
		\[
		\frac{\lambda(D_N \setminus \Delta)}{\lambda_N}	 \leq  \frac{\lambda(D_N \setminus \Delta)}{\lambda(B^\eps)} = \frac{1 - c_d \frac{\lambda_b}{(\alpha N + N^\gamma)^2}(1 + \bar{o}(1))}{1 - c_d \frac{\lambda_b}{(\alpha N + \eps N)^2}(1 + \bar{o}(1))} = \exp \Big( - \frac{c_d \lambda_b}{(\alpha N)^2} \Big[ \frac{\eps(2\alpha+\eps)}{(\alpha+\eps)^2} + \bar{o}(1) \Big] \Big) \, . \]
		In particular, we get an upper bound provided $N$ large enough: for any $x \in D_N \setminus \Delta$,
		\begin{equation}\label{eq:proba-rester-bcp-dans-B}
			\Pbf^N_{x} (H_\Delta > \tfrac12 s N^{2+a}) \leq c \kappa_1 \sum_{k \geq 0} \exp \Big[ - \frac{c_d \lambda_b}{2\alpha^2} \Big( \frac{\eps(2\alpha+\eps)}{(\alpha+\eps)^2} + \bar{o}(1) \Big) (s N^a + k) \Big] \leq C e^{-c s N^a} \, .
		\end{equation}
		Consider now some $x \in \Delta$. We have
		%for any $b > 0$:
		\begin{equation}\label{eq:proba-eviter-B-depuis-Delta}
			\Pbf^N_{x} (H_B >  \tfrac12  s N^{2+a}) \leq \frac{1}{\phi_N(x)} \sum_{k \geq 0} (\lambda_N)^{- \tfrac12 s N^{2+a}  - (k+1) N^2} \Pbf_x \big( H_B > \tfrac12 s N^{2+a}  + k N^2 \big) \, .
		\end{equation}
		Let $\lambda_j^{\ZZ}(K)$ stand for the $j$-th eigenvalue of the transition matrix of the walk killed on $\partial K$, and let $\lambda_j^\RR(\Omega)$ be the $j$-th Dirichlet eigenvalue of (minus) the Laplacian on the domain $\Omega \subset \RR^d$. By \eqref{eq:vp-blowup} and $x_0^N = N x_0$, we have that
		\[ \forall j \geq 1 \,, \quad \lambda_j^\ZZ( D_N \setminus B_N) = 1 - \frac{c_d}{N^2} \lambda_j^\RR(D \setminus B(x_0,\alpha))(1 + \bar{o}(1)) \, . \]			
		In particular, the transition kernel of the random walk killed on $\partial B_N \cup \partial D_N$ admits a spectral gap of order $N^{-2}$. With Lemma \ref{lem:vp-retire-boule}, there is a positive $\eta_B = \eta_B(x_0,d,\alpha,D)$ such that $\lambda_1^\RR(D \setminus B(x_0,\alpha)) > \lambda + \eta_B$, where we recall that $\lambda$ is the first eigenvalue of the Laplace-Beltrami operator on $D$. In particular, using again Proposition \ref{prop:equivalent-proba-sortie-eigen}-\eqref{rem:encadrement-proba-sortie-eigen}, we have for any $x \in \Delta$:
		\begin{equation}\label{eq:proba-rester-bcp-dans-Delta}
			\begin{split}
				\Pbf^N_{x} (H_B > \tfrac12 s N^{2+a}) &\leq \frac{1}{\phi_N(x)} \sum_{k \geq 0} (\lambda_N)^{- \tfrac12 s N^{2+a}  - (k+1) N^2} \Pbf_x \big( H_B > \tfrac12 s N^{2+a}  + (k+1) N^2 \big)\\
				&\leq \frac{C}{\phi_N(x)} \sum_{k \geq 0} \exp \Big( - \eta_B (\tfrac12 s N^{a}  + k) \Big) \leq \frac{C}{\phi_N(x)} e^{- c s N^a} \, .
			\end{split}
		\end{equation}
		If we restrict our choices of $x$ to $\partial \Delta$, we have $\phi_N(x) \geq \kappa_1$ (recall Proposition~\ref{prop:encadrement-ratio}). Therefore, injecting \eqref{eq:proba-rester-bcp-dans-B} and \eqref{eq:proba-rester-bcp-dans-Delta} in \eqref{eq:proba-retour-grand-bord-B} gives the first part of the claim.
		
		For the second part, also injecting \eqref{eq:proba-rester-bcp-dans-B} and \eqref{eq:proba-rester-bcp-dans-Delta} this time in \eqref{eq:proba-retour-grand-phi2}, we get
		\begin{equation*}
			\begin{split} \Pbf^N_{\phi_N^2} (R_1 > s N^{2+a}) &\leq \frac{1}{\| \phi_N\|_2^2} \sum_{x \in D_N \setminus \Delta} \phi_N^2(x) \Big[C e^{-c s N^a} + \frac{C}{\kappa_1} e^{- c u N^a} \Big] + \frac{1}{\| \phi_N\|_2^2} \sum_{x \in \Delta} \phi_N(x) C e^{- c s N^a} \\
				&\leq \Big[ 1 + \tfrac{\| \phi_N \|}{\| \phi_N\|_2^2} \Big] C' e^{- c s N^a} \, .
			\end{split}
		\end{equation*}
		With \eqref{eq:phiN-borne} we can easily prove that $\| \phi_N \| \asymp N^d$, and we recall $\| \phi_N\|_2^2 = N^d$ by \eqref{def:eigenfunction} This means that $1 + \tfrac{\| \phi_N \|}{\| \phi_N\|_2^2}$ is bounded from above by a constant provided $N$ large enough, therefore finishing the proof of the first part of the claim.
	\end{proof}

	\begin{proof}[Proof of Claim \ref{claim:esp-temps-k-retour}]
		By the Markov property and \eqref{eq:loi-retour-sachant-Gj}, we have
		\[
		\Ebf^N_{\phi_N^2}[R_{m \ell}] = \Ebf^N_{\phi_N^2}[R_{2 \ell}] + (1 + \grdO(2^{-N^{\varsigma/2}})) \Ebf^N_{\tilde{e}_\Delta^{\phi_N}}[R_{(m-2) \ell}] \, .
		\]
		Since $\tilde{e}_\Delta^{\phi_N}$ is stationary for the entrance site in $B$ of the confined RW, we have
		\[ \Ebf^N_{\tilde{e}_\Delta^{\phi_N}}[R_{(m-2) \ell}] = (m-2) \ell \cdot \Ebf^N_{\tilde{e}_\Delta^{\phi_N}}[R_1] \, . \]
		To get an expression for this expectation, we observe that $(X_{R_i}, R_i)_{i \geq 1}$ forms a Markov renewal process. The renewal theorem for Markov renewal processes (see \textit{e.g.} \cite[Section 6.a]{cinlar1969markov}) then states that for any $N \geq 1$, under $\Pbf_{\phi_N^2}^N$,
		\[ \frac{\mathcal{N}(t)}{t} \xrightarrow[t \to +\infty]{L^1} \frac{1}{\Ebf^N_{\tilde{e}_\Delta^{\phi_N}}[R_1]} \, . \]
		Therefore, using Lemma \ref{lem:esp-nb-retour-walk}, we have
		\[ \Ebf^N_{\tilde{e}_\Delta^{\phi_N}}[R_{(m-2) \ell}] = (m-2) \ell \Ebf^N_{\tilde{e}_\Delta^{\phi_N}}[R_1] = (m-2) \ell \left[ \lim_{t \to +\infty} \frac{1}{t} \Ebf_{\phi_N^2}^N \left[ \mathcal{N}(t) \right] \right]^{-1} = (m-2) \ell \frac{\| \phi_N\|_2^2}{\cpc_\Delta^{\phi_N}(B)} \, . \]
		
		With $m = m_-$ or $m = m_+$, we thus proved that
		\[
		\Ebf^N_{\tilde{e}_\Delta^{\phi_N}}[R_{k_{\pm}}] = (1 + \grdO(2^{-N^{\varsigma/2}}))  k_{\pm} \frac{\| \phi_N\|_2^2}{\cpc_\Delta^{\phi_N}(B)} +  \Ebf^N_{\phi_N^2}[R_{2 \ell}] -2 \ell  \frac{\| \phi_N\|_2^2}{\cpc_\Delta^{\phi_N}(B)} \,.
		\]
		Note that $ k_{\pm} \frac{\| \phi_N\|_2^2}{\cpc_\Delta^{\phi_N}(B)} \asymp t_N$ by definition and that $\ell := N^{\varsigma} T^Y_{\mix} \leq c N^{\varsigma +1-\gamma}$ (see Proposition~\ref{prop:temps-mixage}), $\| \phi_N\|_2^2 = N^{d}$ (recall \eqref{def:eigenfunction}) and $\cpc_\Delta^{\phi_N}(B) \asymp N^{d-1-\gamma}$  (see Lemma~\ref{lem:encadrement-mes-harmonique}), so that
		\begin{equation}
			\label{eq:separation-esp-ke-retour}
			\Ebf^N_{\tilde{e}_\Delta^{\phi_N}}[R_{k_{\pm}}] = k_{\pm} \frac{\| \phi_N\|_2^2}{\cpc_\Delta^{\phi_N}(B)} +  \Ebf^N_{\phi_N^2}[R_{2 \ell}] + \grdO(2^{-N^{\varsigma/2}}) t_N + \grdO(N^{2+\varsigma}) \,.
		\end{equation}
		Note that since $\eps_N t_N \geq N^{2 + \frac14 \delta}$, taking $\varsigma$ small enough we have $\grdO(N^{2+\varsigma}) = \grdO(N^{-c}) \eps_N t_N$. Furthermore, observe that $\grdO(2^{-N^{\varsigma/2}}) = \bar{o}(\eps_N)$.
		
		\smallskip
		It then remains to control $\Ebf^N_{\phi_N^2}[R_{2 \ell}]$.
		%		
		%		\begin{equation}\label{eq:asymp-esp-retours_>2}
			%			\lim_{N \to +\infty} \frac{1}{t_N} \Ebf^N_{\tilde{e}_\Delta^{\phi_N}}[R_{(m_\pm-2) \ell}] = 1 \pm \eta \, .
			%		\end{equation}
		%		
		%		Let us now estimate $\Ebf^N_{\phi_N^2}[R_{2 \ell}]$. 
		We first observe that because of the differences in the starting point of the confined RW trajectory, we have to consider separately the start (before $R_1$) and the following visits to $B$. Therefore,
		\begin{equation}
			\label{eq:allexpectations1}
			\Ebf^N_{\phi_N^2}[R_{2 \ell}] \leq \Ebf^N_{\phi_N^2}[R_1] + (2\ell-1) \sup_{x \in \partial B} \Ebf^N_{x}[R_1] \, .
		\end{equation}
		Recall that $R_1$ is the first hitting time of $B$ \underline{after} visiting $\Delta$, therefore
		\begin{equation}
			\label{eq:allexpectations2}
			\Ebf^N_{\phi_N^2}[R_1] \leq \frac{1}{\| \phi_N\|_2^2} \left[ \sum_{x \in \Delta} \phi_N^2(x) \Ebf^N_x[H_B] +  \sum_{x \in D_N \setminus \Delta} \phi_N^2(x) \Big(\Ebf^N_x[H_\Delta] + \sup_{y \in \partial \Delta} \Ebf_y[H_B] \Big) \right] \, .
		\end{equation}
		With the same proof as Claim \ref{claim:proba-tps-retour-grand}, we can easily prove that for any fixed $r>0$, the expectations appearing in \eqref{eq:allexpectations2}, and thus in~\eqref{eq:allexpectations1}, are bounded from above by $C N^{2 + r}$. We thus conclude that	
		\[  \Ebf^N_{\phi_N^2}[R_{2 \ell}] \leq 2C \ell N^{2+r}  \leq c N^{2 + r + \varsigma + (1-\gamma)} \, , \]
		where we used Proposition \ref{prop:temps-mixage} for the last inequality.
		
		All together, provided that $r, \varsigma$ are fixed close enough to $0$, recalling that $\eps_N t_N \geq N^{2+ \frac34 \delta}$, we obtain that 
		\[
		\Ebf^N_{\tilde{e}_\Delta^{\phi_N}}[R_{k_{\pm}}] = k_{\pm} \frac{\| \phi_N\|_2^2}{\cpc_\Delta^{\phi_N}(B)} +   \grdO(N^{-\frac14 \delta}) \eps_N t_N  \,,
		\]
		which completes the proof of the claim.
	\end{proof}

	\subsubsection{For the tilted RI}	
	
	We now estimate the number of excursions in $B$ of the tilted RI at some level $u > 0$.	Recall notation from Section~\ref{ssec:defYZ} that each trajectory $\omega^j$ hitting $B$ has $T^j$ excursions $(\omega^j_i)_{i \leq T^j}$ starting at $B$ and ending at $\partial \Delta$. Let $J^N_{u}$ be a Poisson random variable with parameter $u \,\cpc^{\Psi_N}(B)$.
	Given $(T^j)_{j \geq 1}$ we define the number of excursions at level $u$ of the $\Psi_N$-tilted interlacement as
	\begin{equation}
		\mathcal{N}_{\Psi_N}(u) = \sum_{j = 1}^{J^N_u} T^j \, .
	\end{equation}
	
	\begin{proposition}\label{prop:nb-excursions-entrelac}
		There are constants $c,C > 0$ such that for $N$ large enough,
		\begin{equation}\label{eq:prop:nb-excursions-entrelac}
			\proba{ \big| \mathcal{N}_{\Psi_N}(u) - u \, \lambda_N \cpc_\Delta^{\phi_N}(B) \big| \geq \eps_N u \lambda_N  \cpc_\Delta^{\phi_N}(B)} \leq C e^{-c  u\, \eps_N^2  N^{d-2}} \, .
		\end{equation}	
	\end{proposition}
	
	\begin{proof}
		Since $J_u^N$ is a Poisson random variable, we can use the Chernov bound (see e.g.\ \cite[Section 2.2]{inegalite-concentration}) to get that
		\begin{equation*}
			\proba{ \big| J^N_u - u \cpc^{\Psi_N}(B) \big| \geq \tfrac12 \eps_N \, u\, \cpc^{\Psi_N}(B)} \leq \exp\Big( - u \, \cpc^{\Psi_N}(B)\, [h(\tfrac12 \eps_N) + h(-\tfrac12 \eps_N)] \Big) \, ,
		\end{equation*}
		where $h(x) := (1+x) \log(1+x) - x$. We easily see with a Taylor expansion that for $N$ large enough, $h(\tfrac12 \eps_N) + h(-\tfrac12 \eps_N) \geq (\tfrac14 \eps_N)^2$. Therefore,
		\begin{equation*}
			\proba{ \big| J^N_u - u \cpc^{\Psi_N}(B) \big| \geq \tfrac12 \eps_N \, u \, \cpc^{\Psi_N}(B)} \leq \exp\Big( -c(\tfrac14 \eps_N)^2 u\, \cpc^{\Psi_N}(B) \Big) \leq \exp\Big( -c' \eps_N^2 u\, N^{d-2} \Big)  \, ,
		\end{equation*}
		using also that  $\cpc^{\Psi_N}(B) \asymp N^{d-2}$, recall Lemma~\ref{lem:encadrement-mes-harm-tilt}.	
		
		Notice that the random variables $(T^j)_{j \geq 1}$ are i.i.d, with distribution dominated by a geometric variable. Summing \eqref{eq:esperance-pi-Z-terme-final} over $x \in \partial B$, we obtain
		\begin{equation}
			\mathbf{E}_{\bar{e}^{\Psi_N}_B}^{\Psi_N} \big[ T^{1} \big] = \frac{\lambda_N}{\cpc^{\Psi_N}(B)} \sum_{x \in \partial B}  \phi_N(x)^2 \Pbf^N_x(H_{\Delta} < \bar{H}_B) = \lambda_N \frac{\cpc_\Delta^{\phi_N}(B)}{\cpc^{\Psi_N}(B)} \, .
		\end{equation}
		Again, with the Chernov bound (since $T^j$ has an exponential moment because of the domination), we have
		\begin{equation}\label{eq:deviation-somme-Tj}
			\proba{\left|\sum_{j = 1}^{v} T^j - v \lambda_N \frac{\cpc_\Delta^{\phi_N}(B)}{\cpc^{\Psi_N}(B)} \right| \geq \frac12 v \lambda_N \frac{\cpc_\Delta^{\phi_N}(B)}{\cpc^{\Psi_N}(B)}} \leq C \exp \left(-c \eta^2 v \lambda_N \frac{\cpc_\Delta^{\phi_N}(B)}{\cpc^{\Psi_N}(B)} \right) \, .
		\end{equation}
		In particular, applying the bound of \eqref{eq:deviation-somme-Tj} with $v_{N,u}^\pm = (u \pm \tfrac12 \eps_N) \cpc^{\Psi_N}(B)$, and taking $N$ large enough so that $\lambda_N \geq \tfrac12$, we have
		\begin{equation}
			\begin{split}
				\proba{\Big|\frac{\sum_{j = 1}^{v_{N,u}^\pm} T^j}{(u \pm \tfrac12 \eps_N) \lambda_N \cpc_\Delta^{\phi_N}(B)} - 1 \Big| \geq \frac12 \eps_N} &\leq C \exp \left(-c \eps_N^2 \lambda_N \cpc_\Delta^{\phi_N}(B) \right) \\
				&\leq C \exp \left(-c' \eps_N^2 N^{d-1-\gamma} \right) \, ,
			\end{split}
		\end{equation}
		where we used $\cpc_\Delta^{\phi_N}(B) \asymp N^{d-1-\gamma}$, recall Lemma \ref{lem:encadrement-mes-harm-tilt}.
		
		Therefore, we can control the deviations of $\mathcal{N}_{\Psi_N}(u)$:
		\begin{equation}
			\begin{split}
				\proba{\Big|\frac{\mathcal{N}_{\Psi_N}(u)}{u \lambda_N \, \cpc_\Delta^{\phi_N}(B)} - 1\Big| \geq \eps_N} &\leq \PP \Big( \Big| \frac{J^N_u}{u \cpc^{\Psi_N}(B)} - 1 \Big| \geq \tfrac12 \eps_N \Big) + 2C \exp \left(-c' \eps_N^2 N^{d-1-\gamma} \right) \\
				&\leq C e^{-c u \eps_N^2 N^{d-2}} + 2C \exp \left(-c' \eps_N^2 N^{d-1-\gamma} \right) \, .
			\end{split}
		\end{equation}
		Since $\gamma < 1$, we get \eqref{eq:prop:nb-excursions-entrelac} for $N$ large enough.
	\end{proof}
	
	Observe that since we want to have $\mathcal{N}(t_N) \approx \mathcal{N}_{\Psi_N}(u_N)$, combining Propositions \ref{prop:nb-excursions-tN} \& \ref{prop:nb-excursions-entrelac} imposes to take $u_N = t_N/ \lambda_N \| \phi_N\|_2^2$.
	
	\subsection{Coupling of entrance and exit points}

	Let us recall that $\mathcal{N}(t_N)$ is the number of excursions from $\partial B$ to $\partial \Delta$ done by the confined RW up to time $t_N$, and $\mathcal{N}_{\Psi_N}(u_n)$ is the same quantity but for the tilted RI $\mathscr{I}_{\Psi_N}(u_N)$.
	
	We want to apply Theorem \ref{th:soft-local-times} to get a coupling of the ranges of the entrance and exit points $Y,Z$ up to a well-chosen time $n$. By definition, we want to choose $n$ in order to have $n \approx \mathcal{N}(t_N)$ with high probability. According to Propositions \ref{prop:nb-excursions-tN} and \ref{prop:nb-excursions-entrelac}, we take 
	\begin{equation}\label{eq:tps-soft-local-times}
		n = u_N \cpc_\Delta^{\phi_N}(B) = t_N \| \phi_N\|_2^{-2} \cpc_\Delta^{\phi_N}(B) \asymp t_N N^{-1-\gamma} \, ,
	\end{equation}
	where we used the definition of $u_N$ in Theorem \ref{th:couplage-CRW-entrelac-t_N} for the second identity, and Lemma~\ref{lem:encadrement-mes-harmonique} to get the last asymptotics (also using $\| \phi_N\|_2^2 = N^d$).
	
	We will apply Theorem \ref{th:couplage-CdM-diff-pi} with the following parameters:
	\begin{equation}\label{eq:parametres-soft-local-times}
		\begin{split}
			&|\Sigma| = |\partial B| \times |\partial \Delta| \asymp N^{2(d-1)} \, ,\\
			&T^Y_\mix \vee T^Z_\mix \leq c N^{1-\gamma} \quad \text{(Proposition \ref{prop:temps-mixage})}\, ,\\
			&\pi_Y(x,y) = \tilde{e}_\Delta^{\phi_N}(x) \Pbf^N_x(X_{H_\Delta} = y) \sim \pi_Z(x,y) \quad \text{(Lemma \ref{lem:mes-inv-Y} \& \ref{lem:mes-inv-Z})} \, , \\
			&g^Z(x,y) \sim g^Y(x,y) = \tilde{e}_\Delta^{\phi_N}(x) \asymp N^{1-d} \quad \text{(Lemma \ref{lem:encadrement-mes-harmonique})} \, ,\\
			&\mathrm{Var}_{\pi_Y}(\rho_{(x,y)}^Y), \mathrm{Var}_{\pi_Z}(\rho_{(x,y)}^Z) \asymp N^{1-d} N^{-\gamma(d-1)}  \quad \text{(Proposition \ref{prop:variance-1})} \, ,\\
			&\| \rho_{(x,y)}^Y \|_\infty, \| \rho_{(x,y)}^Z \|_\infty \asymp N^{-\gamma(d-1)} \quad \text{(Equations \eqref{eq:rho-upper-bound} \& \eqref{eq:rho-lower-bound})} \, .
		\end{split}
	\end{equation}
	With all these estimates, we are able to control all quantities appearing in Theorem \ref{th:soft-local-times}.
	
	First, we write
	\begin{equation}
		\label{eq:borne-exp-time}
		\exp \left( - n \eps_N^2 \right) \leq \exp \left( - c  \eps_N^2 t_N N^{-1-\gamma}  \right) \leq \exp \left( - c  \eps_N^2 t_N N^{-2}  \right) \leq \exp(-cN^{\delta/2}) \eqdef \delta_{\mathrm{time}}^{\gamma}(N) \, .
	\end{equation}
	
	Then, we want to get a lower bound on the ratio between the invariant and starting distributions, that is $\pi_\circ/\nu_\circ$, where $\circ\in\{Y,Z\}$, which appears in the upper bound of~\eqref{eq:borne-soft-local-times}. Let us first consider $Y$: recalling \eqref{def:nuY}, we have
	\begin{equation*}\label{eq:ratio-init-Y}
		\frac{\pi_Y(x,y)}{\nu_Y(x,y)} = \frac{\tilde{e}_\Delta^{\phi_N}(x) \| \phi_N\|_2^2}{\sum\limits_{z \in D_N} \phi_N^2(z) \Pbf^N_z(X_{R_1} = x)} \geq \frac{\tilde{e}_\Delta^{\phi_N}(x)}{\sup\limits_{z \in D_N} \Pbf^N_z(X_{R_1} = x)} \geq \frac{cN^{1-d}}{N^{-\gamma(d-1)}} = cN^{-(1-\gamma)(d-1)} \, ,
	\end{equation*}
	where we have used Lemma~\ref{lem:encadrement-mes-harmonique} and Proposition \ref{prop:point-retour-B} for the last lower bound.
	For $Z$, recalling \eqref{def:nuZ} and \eqref{eq:erreur-mes-invariante}, we have
	\begin{equation*}
		\label{eq:ratio-init-Z}
		\frac{\pi_Z(x,y)}{\nu_Z(x,y)} \geq (1 - c N^{\gamma - 1}) \frac{\tilde{e}_\Delta^{\phi_N}(x)}{\bar{e}_B^{\Psi_N}(x)} \geq C \, ,
	\end{equation*}
	where we have used Lemma~\ref{lem:encadrement-mes-harm-tilt} and Lemma~\ref{lem:encadrement-mes-harmonique} for the last lower bound.
	%From \eqref{eq:ratio-init-Y} and \eqref{eq:ratio-init-Z}, 
	We therefore end up with the following bound:
	\begin{equation}\label{eq:borne-exp-mesure-depart}
		\sup_{\circ = Y, Z} \sup_{z \in \Sigma} \, \exp \left( -c n \eps_N \frac{\pi_\circ(z)}{\nu_\circ(z)} \right) \leq \exp \left(- c \, \eps_N\, t_N N^{-2 -(1-\gamma)(d-2)} \right) \eqdef \delta_{\mathrm{init}}^\gamma(N) \, .
	\end{equation}

	Recall the definition of $k(\cdot)$ in~\eqref{eq:softloc-keps}, then since $\pi_\ast$ decays polynomially in $N$, see Lemma~\ref{lem:encadrement-mes-harmonique}, and  $\frac{g^\circ(z)^2}{\mathrm{Var}_{\pi_\circ} (\rho^\circ_z)} \asymp N^{-(1-\gamma)(d-1)}$, see~\eqref{eq:parametres-soft-local-times}, we have
	\begin{equation*}\label{eq:asymp-k_eps}
		k(\eps_N) \leq c \log N - 2 \log \eps_N \leq c' \log \big( N \big) \, .
	\end{equation*}
	Combining this with the  estimates in \eqref{eq:parametres-soft-local-times}, we have the upper bound
	\begin{equation}
		\label{eq:borne-exp-var}
		\sup_{\circ = Y, Z} \sup_{z \in \Sigma} \exp \left( - c \frac{\eps_N^2 g^\circ(z)^2}{\mathrm{Var}_{\pi_\circ} (\rho^\circ_z)} \frac{n}{k(\eps) T^\circ_\mix} \right) \leq \exp \left( - c \eps_N^2 \, t_N \, \frac{N^{-2- (1-\gamma)(d-3) } }{\log \big( N \big)} \right) \eqdef \delta_{\mathrm{Var}}^{\gamma}(N) \, .
	\end{equation}
	%	with $\theta = \gamma(d-1) - 1$ that thus satisfies $d-2 - \delta' < \theta < d-2$.
	Finally, we check the assumptions on $\eps_N$ and $n$ to apply Theorem \ref{th:soft-local-times}. With the estimates in~\eqref{eq:parametres-soft-local-times}, we see that there exists a constant $c_V \in (0,1)$ such that
	\begin{equation}
		\inf_{\circ = Y, Z} \inf_{z \in \Sigma} \frac{\mathrm{Var}_{\pi^\circ} (\rho^\circ_z)}{2 \| \rho^\circ_z \|_\infty g^\circ(z)} \geq c_V \, ,
	\end{equation}
	hence $\eps_N \to 0$ ensures that \eqref{eq:softloc-epsilon} is verified for $\eps_N$ provided $N$ large enough. We also check that
	\begin{equation}
		n \asymp t_N N^{-(1+\gamma)} \gg N^{\delta}N^{1-\gamma} \ge 2 k(\eps_N)  T^Y_\mix \wedge T^Z_\mix.
	\end{equation}
	
	We can now use Theorem \ref{th:soft-local-times}. Recall that we fixed $\delta > 0$ to have $t_N/N^{2+\delta} \to +\infty$, as well as $\eps_N = N^{-\delta/4}$. Let $N$ be sufficiently large so that $\eps_N \leq c_V$ and $n \geq 2 k(\eps_N)  T^Y_\mix \wedge T^Z_\mix$, and define the event
	\begin{equation}\label{eq:event-couplage-pts-entrees-sorties}
		\mathcal{G}^{\delta,\gamma}_{n,\eps_N} \defeq \Bigg\{ \bigcup_{i = 1}^{(1-\eps_N) n} Z_i \subseteq \bigcup_{i = 1}^{n} Y_i \subseteq \bigcup_{i = 1}^{(1+\eps_N) n} Z_i \Bigg\} \, .
	\end{equation}
	Applying Theorem \ref{th:couplage-CdM-diff-pi} with $n$ as in \eqref{eq:tps-soft-local-times} and $\eps=\eps_N$, we get that for any $N$ large enough, there is a coupling of $Y$ and $Z$, denoted by $\QQ_N$, such that
	\begin{equation}\label{eq:borne-couplage-points-3-termes}
		\begin{split}
			\QQ_N \left( (\mathcal{G}^{\delta,\gamma}_{n,\eps_N})^c \right)& \leq C N^{2(d-1)} \Big[ \delta_{\mathrm{time}}^{\gamma}(N) + \delta_{\mathrm{init}}^\gamma(N) + \delta_{\mathrm{Var}}^{\gamma}(N) \Big]  \\
			& \leq C N^{2(d-1)} \Big[ e^{- c N^{\frac12 \delta}} + e^{-c N^{\frac34 \delta - (1-\gamma)(d-2)}} +  e^{-c N^{\frac12 \delta - (1-\gamma)(d-3)}  / \log (N)}  \Big] \,, 
		\end{split}
	\end{equation}
	where we have used the bounds~\eqref{eq:borne-exp-time}-\eqref{eq:borne-exp-mesure-depart}-\eqref{eq:borne-exp-var}, together with $\eps_N = N^{-\delta/4}$ and $t_N \geq N^{2+\delta}$.
	Recall that we fixed $\gamma$ sufficiently close to $1$ so that $(1-\gamma)(d-2) < \delta(d-2)/4(d-2) = \tfrac14 \delta$. Therefore, we end up with $\QQ_N ( (\mathcal{G}^{\delta,\gamma}_{n,\eps_N})^c ) \leq C N^{2(d-1)} \exp( - c N^{\frac18 \delta})$, for $N$ sufficiently large.
	
	%	
	%	Observe that
	%	\begin{equation}\label{eq:comparaison-erreur}
		%		\frac{\log \delta_{\mathrm{time}}^{\gamma}(N)}{\log \delta_{\mathrm{Var}}^{\gamma}(N)} \asymp N^{(1 - \gamma) d} \log N
		%			\quad , \quad
		%		\frac{\log \delta_{\mathrm{init}}^{\gamma}(N)}{\log \delta_{\mathrm{Var}}^{\gamma}(N)} \asymp \frac{N^{1 - \gamma}}{\eps_N} \log N =  \frac{N^{1 - \gamma}}{\eps_N}
		%	\end{equation}
	%	 we can see that both of these ratios go polynomially fast to infinity when $N \to +\infty$, meaning that $\delta_{\mathrm{time}}^{\gamma}(N)$ and $\delta_{\mathrm{init}}^\gamma(N)$ are negligible compared to $\delta_{\mathrm{Var}}^{\gamma}(N)$. Injecting $\eps_N^2 t_N = N^{2+\delta'}$, we have
	%	\begin{equation}\label{eq:borne-couplage-points-1-terme}
		%		\QQ_N \left( (\mathcal{G}^{\delta,\gamma}_{n,\eps_N})^c \right) \leq C N^{2(d-1)} \delta_{\mathrm{Var}}^{\gamma}(N) \leq C N^{2(d-1)} \exp \left( - c' \frac{N^{2+\delta'}}{N^d} \frac{N^\theta}{\log N} \right)  \, .
		%	\end{equation}
	%	where we used $\| \phi_N^2 \| \asymp N^d$ (recall Lemma \ref{lem:norme-phiN2}). Since by construction $\theta > d-2-\delta'$, we can reformulate \eqref{eq:borne-couplage-points-1-terme} 
	This can be reformulated as follows: there exist $\eta > 0$ and constants $c_1,c_2 > 0$ such that for $N$ large enough,
	\begin{equation}\label{eq:proba-couplage-sites}
		\QQ_N \left( \mathset{Z_i}_{i = 1}^{(1-\eps_N) n} \subseteq \mathset{Y_i}_{i = 1}^{n} \subseteq \mathset{Z_i}_{i = 1}^{(1+\eps_N) n} \right) \geq 1 - c_1 e^{-c_2 N^{\eta}} \, ,
	\end{equation}
	where $n = u_N \cpc_\Delta^{\phi_N}(B)$.

	\subsection{Coupling the excursions}
	
	Let us first summarize what we achieved and what is left to do to get Theorem \ref{th:couplage-CRW-entrelac-t_N}.
	
	We defined at the beginning of Section \ref{ssec:chaines-markov} two Markov chains $Y$ and $Z$ on the space state $\Sigma = \partial B \times \partial \Delta$, which correspond to the entrance and exit points of the confined RW and the tilted RI excursions respectively.
	The Markov chain $Y$ is defined with the use of stopping times $(R_i)_{i \geq 1}, (D_i)_{i \geq 1}$ (recall \eqref{eq:def-instants}) for the confined RW, in a way that for all $i \geq 1$, $Y_i = (X_{R_i}, X_{D_i}) \eqdef (Y^R_i, Y^D_i)$. The Markov chain $Z$ is defined by ordering the trajectories and excursions of the tilted RI that hit $B$, and for $j \geq 1$, $Z_j = (w_j(0), w_j(H_{\Delta}))$ where $w_j$ is the $j$-th excursion.
	
	We obtained above in \eqref{eq:proba-couplage-sites} that the ranges of $Y$ and $Z$ roughly coincide up to a small time correction $\eps_N$. Observe that for any $(x,y) \in \Sigma$, the excursions $(X_{R_i}, X_{R_i+1}, \dots , X_{D_i})$ under $\Pbf^N_x( \cdot \, | \, X_i = (x,y))$ -- meaning a confined walk bridge -- have the same law as an excursion of a trajectory of the tilted RI -- meaning a tilted RW trajectory -- conditioned to enter $B$ at $x$ and exit through $\partial \Delta$ at $y$ . Therefore, if we know that $Y_i = Z_j$ for some $i,j \in \mathset{0, \dots, n}$, our coupling consists of taking the same path for $(X_{R_i}, X_{R_i+1}, \dots , X_{D_i})$ and for $w_j = (w_j(0), w_j(1), \dots , w(H_\Delta))$.
	
	In all the following, we work conditionally on the entrance and exit sites $Y$ and $Z$.
	We first explain how to construct two sets of excursions $(\mathcal{E}^{\mathrm{in}}_i)_{i \geq 1}$ and $(\mathcal{E}^{\mathrm{out}}_i)_{i \geq 0}$ of the confined RW. Write $Y_i = (Y_i^R, Y_i^D)$, then under $\QQ_N$ the excursions set satisfies the following properties:
	\begin{itemize}
		\item Conditionally on the values of $Y$ and $Z$, the collections $(\mathcal{E}^{\mathrm{in}}_i)$ and $(\mathcal{E}^{\mathrm{out}}_i)$ are independent sequences of independent variables.
		
		\item For any $i \geq 1$, the random variable $\mathcal{E}^{\mathrm{in}}_i$ is a nearest-neighbour path with law
		\[ \QQ_N (\mathcal{E}^{\mathrm{in}}_i \in \cdot \, | \, Y,Z) = \Pbf^N_{Y^R_i} \left( \mathcal{E}^{\mathrm{in}}_i \in \cdot \, | \, X_{H_\Delta} = Y^D_i \right) \]
		that is $\mathcal{E}^{\mathrm{in}}_i$ has the law of a confined RW path starting at $Y^R_i$ and ending at $Y^D_i$ when it first hits $\Delta$.
		
		\item For any $i \geq 1$, the random variable $\mathcal{E}^{\mathrm{out}}_i$ is a nearest-neighbour path with law
		\[ \QQ_N (\mathcal{E}^{\mathrm{out}}_i \in \cdot \, | \, Y,Z) = \Pbf^N_{Y^D_i} \left( \mathcal{E}^{\mathrm{out}}_i \in \cdot \, | \, X_{H_B} = Y^R_{i+1} \right) \, , \]
		that is $\mathcal{E}^{\mathrm{out}}_i$ has the law of a confined RW path starting at $Y^D_i$ that runs until going back to~$B$, with hitting site $Y^D_i$.
		
		\item The random variable $\mathcal{E}^{\mathrm{out}}_0$ is a nearest-neighbour path with law
		\[ \QQ_N (\mathcal{E}^{\mathrm{out}}_0 \in \cdot \, | \, Y,Z) = \Pbf^N_{\phi_N^2} \left( \mathcal{E}^{\mathrm{out}}_0  \in \cdot \, | \, X_{R_1} = Y^R_{1} \right) \, , \]
		that is $\mathcal{E}^{\mathrm{out}}_0$ is a confined RW path starting from distribution $\phi_N^2$ that ends when it visits~$B$ for the first time after leaving $\Delta$, with first hitting site $Y^R_{1}$.
	\end{itemize}
	
	With this construction, we have for all $i \geq 1$,
	\[ \mathcal{E}^{\mathrm{in}}_i \overset{(d)}{=} (X_{R_i}, X_{R_i+1}, \dots , X_{D_i}) \quad , \quad \mathcal{E}^{\mathrm{out}}_i \overset{(d)}{=} (X_{D_i}, X_{D_i+1}, \dots , X_{R_i}) \,.\]
	
	%{\red [Etrange: tu définis la loi des $\mathcal{E}^{\mathrm{in}}_i$ et $\mathcal{E}^{\mathrm{out}}_i$ sous $\QQ_N$ par rapport aux lois des mêmes variables sous $\Pbf^N_{Y^R_i}$, etc (avec les bons conditionnements), mais celles-ci n'ont pas vraiment été définies non plus auparavant... Tu sembles dire d'après l'intro de la section que $\mathcal{E}^{\mathrm{in}}_i=(X_{R_i}, X_{R_i+1}, \dots , X_{D_i})$ mais sans vraiment le dire... (et on devine que $\mathcal{E}^{\mathrm{out}}_i=(X_{D_i}, X_{D_i+1}, \dots , X_{R_i})$)] }
	
	The process $X$ defined on $(\Omega_N, \mathscr{F}_N, \QQ_N)$ as the concatenation of $\mathcal{E}^{\mathrm{out}}_0, \mathcal{E}^{\mathrm{in}}_1, \mathcal{E}^{\mathrm{out}}_1, \mathcal{E}^{\mathrm{in}}_2, \dots$ therefore has the law of the confined RW started from its stationary distribution $\phi_N^2$. With this construction, we have (recall the definition of $\mathcal{N}$ above Proposition \ref{prop:nb-excursions-tN})
	\begin{equation}
		\mathcal{N}(t) = \sup \mathset{p \geq 1 \, : \, \ell(\mathcal{E}^{\mathrm{out}}_0) + \sum_{i = 1}^p (\ell(\mathcal{E}^{\mathrm{in}}_i) + \ell(\mathcal{E}^{\mathrm{out}}_i)) \leq t} \, ,
	\end{equation}
	where $\ell(\gamma)$ is the length of the nearest-neighbour path $\gamma = (\gamma_0, \dots , \gamma_{\ell(\gamma)})$. Similarly, the range $\mathcal{R}_{\phi_N}( t )$ of the confined RW up to time $t$ is given by the process $X$ after the concatenation of the $\mathcal{N}(t)$ first excursions with the start of the $(\mathcal{N}(t)+1)$-th excursion which we stop at the appropriate time.
	
	Since it is possible to have $\mathcal{E}^{\mathrm{out}}_0 \cap B \neq \varnothing$, we need to consider this first path as a part of an excursion of the confined RW. 
	To do so, we fix $\beta > 0$ and we note that since $X$ is stationary, we have the equality in law 
	\[ \mathcal{R}_{\phi_N}( t ) \cap B \overset{(d)}{=}  \mathset{X_i \, : \, \beta t \leq i \leq (1+\beta) t} \cap B \, \]
	under $\QQ_N$. Therefore, we have
	\begin{equation}\label{eq:range-decale-beta}
		\bigcup_{i = \mathcal{N}(\beta t) + 1}^{\mathcal{N}((1+\beta)t)-1} (\mathcal{E}^{\mathrm{in}}_i \cap B) \subseteq \mathcal{R}_{\phi_N}( t ) \cap B \subseteq \bigcup_{i = \mathcal{N}(\beta t)}^{\mathcal{N}((1+\beta)t)} (\mathcal{E}^{\mathrm{in}}_i \cap B) \, .
	\end{equation}
	In short, by shifting the times we consider (using $\beta > 0$), we can have $\mathcal{E}^{\mathrm{out}}_0 \cap B \subseteq \mathcal{E}^{\mathrm{in}}_{\mathcal{N}(\beta t)}$. Since $\mathcal{N}(\beta t)$ is the index of the last excursion before time $\beta t$, this means that the true range $\mathcal{R}_{\phi_N}(t)$ contains a range that starts later (after it exits through $\Delta$), and is contained in a range that started earlier in time (that hit $B$ before).
	
	%{\blue (QB: pas tout à fait d'accord avec la dernière identité, parce que ça ne considère que des excursions entières, non?)	Tu voudrais plutôt écrire $\bigcup_{\mathcal{N}_1+1}^{\mathcal{N}_2+1} \subset \mathcal{R}_{\phi_N}( t ) \subset\bigcup_{\mathcal{N}_1}^{\mathcal{N}_2} $?	}

	To construct the range of the tilted RI $\mathscr{I}_{\Psi_N}(u_N)$ intersected with $B$, we proceed inductively by plugging excursions $(\mathcal{E}^{\mathrm{in}}_i)$ of the tilted/confined RW and regrouping them into trajectories.
	To do so, let $\mathcal{I}_0 = \varnothing$ and recursively define
	\begin{equation}
		\iota_i \defeq \inf \mathset{j \geq 1 \, : \, j \not\in \mathcal{I}_{i-1} , Y_j = Z_i} \, , \quad \mathcal{E}^{\mathrm{RI}}_i = \mathcal{E}^{\mathrm{in}}_{\iota_i} \, , \quad \mathcal{I}_i = \mathcal{I}_{i-1} \cup \mathset{\iota_i} \, .
	\end{equation}
	Simply put, we partially order the set $\mathset{Z_i}_{i \geq 1}$ by comparing it to the set $\mathset{Y_i}_{i \geq 1}$. The excursions $(\mathcal{E}^{\mathrm{RI}}_i)_{i \geq 0}$ of the tilted RI are then given by the corresponding excursion of the confined RW.
	
	To construct the full tilted RI from this collection of excursions, we also define a sequence of Bernoulli variables $(U_i)$, that are independent conditionally on $(Z_i)_{i \geq 0}$, and such that
	\begin{equation}
		\QQ_N(U_i = 1 | (Z_i)_{i\ge 0}) = \Pbf^{\Psi_N}_{Z^D_i} (H_B = +\infty) \bar{e}^T_B(Z^R_{i+1}) \, / \, p^Z(Z_i, (Z^R_{i+1}, \partial \Delta)) \, ,
	\end{equation}
	which is the probability for the $\Psi_N$-tilted RW starting at $Z^D_i$ to go to infinity without returning to $B$, and then for another $\Psi_N$-tilted walk ``coming from infinity'' to hit~$B$ at $Z^R_{i+1}$. Setting $V_0 = 0$ and $V_i = \inf \mathset{i > V_{i-1} \, : \, U_i = 1}$, we can group the excursions $(\mathcal{E}^{\mathrm{RI}}_i)_{i \geq 1}$ into $(\mathcal{E}^{\mathrm{RI}}_i)_{V_{j-1} < i \leq V_j}$ for $j \geq 1$. The $j$-th group corresponds to excursions of the $j$-th trajectory $w^j$, and the full tilted RI can be constructed by plugging bridges and trajectories of the walk on the conductances $\Psi_N$, both conditioned to never hit $B$.
	
	Let $(J^N_u)_{u \geq 0}$ be a Poisson process on the real line with intensity parameter $\cpc^{\Psi_N}(B)$, independent of all the previous randomness. Taking the previously fixed $\beta > 0$, we then have $\mathcal{N}_{\Psi_N}(u) = V_{J^N_u}$ and the equality in law
	\begin{equation}
		\mathscr{I}_{\Psi_N}(u) \cap B \overset{(d)}{=} \bigcup_{ j = J_{\beta u}^N}^{J_{(1+\beta)u}^N} \bigcup_{V_{j-1} < i \leq V_j} (\mathcal{E}^{\mathrm{RI}}_i \cap B) = \bigcup_{i = \mathcal{N}_{\Psi_N}(\beta u)}^{\mathcal{N}_{\Psi_N}((1+\beta) u)} (\mathcal{E}^{\mathrm{RI}}_i \cap B) \, .
	\end{equation}
	
	The key property of this construction is the following: for $p, q, m \in \NN$, we have
	\begin{equation}\label{eq:inclusion-points-inclusion-excursions}
		\mathset{Z_i}_{i = 1}^{p} \subseteq \mathset{Y_i}_{i = 1}^{m} \subseteq \mathset{Z_i}_{i = 1}^{q} \Longrightarrow \big\{\mathcal{E}^{\mathrm{RI}}_i\big\}_{i = 1}^{p} \subseteq \big\{\mathcal{E}^{\mathrm{in}}_i\big\}_{i = 1}^{m} \subseteq \big\{\mathcal{E}^{\mathrm{RI}}_i\big\}_{i = 1}^{q} \, ,
	\end{equation}
	which is how we use the previous section to prove Theorem \ref{th:couplage-CRW-entrelac-t_N}.
	
	\begin{proof}[Proof of Theorem \ref{th:couplage-CRW-entrelac-t_N}]
		Recall that $u_N = t_N / \lambda_N \| \phi_N\|_2^2$ and let us write $\mathcal{E}^{\mathrm{in}}_{B,i} \defeq \mathcal{E}^{\mathrm{in}}_i \cap B$ as well as $\mathcal{E}^{\mathrm{RI}}_{B,i} \defeq \mathcal{E}^{\mathrm{RI}}_i \cap B$. On the probability space $(\Omega_N, \mathscr{F}_N, \QQ_N)$, by Proposition \ref{prop:nb-excursions-tN}, with probability at least $1 - C e^{-c' N^c}$, we have
		\begin{equation*}
			\bigcup_{i = \beta (u_N + \tfrac{\eps_N}{4}) \cpc_\Delta^\phi(B)}^{(1+\beta)(u_N - \tfrac{\eps_N}{4}) \cpc_\Delta^\phi(B)} \mathcal{E}^{\mathrm{in}}_{B,i} \eqdef \mathcal{R}_1 \subseteq \mathcal{R}_{\phi_N}(t_N) \cap B \subseteq  \mathcal{R}_2 \defeq \bigcup_{i =  \beta (u_N - \tfrac{\eps_N}{4}) \cpc_\Delta^\phi(B)}^{(1+\beta)(u_N + \tfrac{\eps_N}{4}) \cpc_\Delta^\phi(B)} \mathcal{E}^{\mathrm{in}}_{B,i}
		\end{equation*}
		
		With Proposition \ref{prop:nb-excursions-entrelac}, noting that $u_N \asymp t_N N^{-d}$ and recalling that $\eps_N^2 t_N \geq N^{2+\frac12\delta}$, with probability at least $1 - C e^{-c N^{\frac12 \delta + 1-\gamma}}$, we have
		\begin{equation*}
			\mathscr{I}_{\Psi_N}((1-\tfrac{\eps_N}{2})u_N) \cap B_N \subseteq \mathcal{R}^{\mathrm{RI}}_1 \defeq  \bigcup_{i = \beta (u_N + \tfrac{\eps_N}{3}) \cpc_\Delta^\phi(B)}^{(1+\beta)(u_N - \tfrac{\eps_N}{3}) \cpc_\Delta^\phi(B)} \mathcal{E}^{\mathrm{RI}}_{B,i} 
		\end{equation*}
		and 
		\begin{equation*}
			\bigcup_{i = \beta (u_N - \tfrac{\eps_N}{3}) \cpc_\Delta^\phi(B)}^{(1+\beta)(u_N + \tfrac{\eps_N}{3}) \cpc_\Delta^\phi(B)} \mathcal{E}^{\mathrm{RI}}_{B,i}  \eqdef \mathcal{R}^{\mathrm{RI}}_2 \subseteq \mathscr{I}_{\Psi_N}((1+\tfrac{\eps_N}{2})u_N) \cap B_N \, .
		\end{equation*}
		Finally, using \eqref{eq:proba-couplage-sites} and \eqref{eq:inclusion-points-inclusion-excursions} and taking $\beta$ small enough, with $\QQ_N$-probability at least $1 - c_1 e^{-c_2 N^\eta}$, we have the inclusions $\mathcal{R}^{\mathrm{RI}}_1\subseteq \mathcal{R}_1\subseteq \mathcal{R}_2 \subseteq \mathcal{R}^{\mathrm{RI}}_2$.
		Therefore, combining all of the above concludes the proof of Theorem~\ref{th:couplage-CRW-entrelac-t_N}.
	\end{proof}

	\begin{appendix}
		
		\section{Some Simple Random Walk (gambler's ruin) estimates}\label{appendix:gambler-ruin}
		
In this appendix, we prove Lemma~\ref{lem:gambler-ruin} and Claim~\ref{claim:esp-tps-retour-bord-anneau}.
For convenience, we recall the statements.

	\begin{customlemma}{\ref{lem:gambler-ruin}}
		Let $z \in (B^\eps \setminus B) \cup \partial B$ and $w \in B^\eps \setminus B$ be such that there are $\iota, \jmath \in [0,1)$ and $\eta > 0$ such that
		\[ |z - x_0^N| - \alpha N \in [\eta, \tfrac{1}{\eta}] N^\iota \quad , \quad |w - x_0^N| - (\alpha + \eps) N \in -[\eta, \tfrac{1}{\eta}] N^\jmath \, . \]
		Then, there exist explicit constants that depend only on $\alpha, \eps, d$ such that for $N$ large enough,
		\begin{equation}\label{eq:proba-atteindre-avant-SRW-appendix}
			\begin{split}
				c_1 \tfrac{\eta}{2} N^{\iota-1} + \grdO(N^{-1}) \leq \, &\Pbf_z(H_{B^\eps} < \bar{H}_B) \leq c_2 \tfrac{2}{\eta} N^{\iota-1} + \grdO(N^{-1}) \, , \\
				c_1' \tfrac{\eta}{2} N^{\jmath-1}  + \grdO(N^{-1}) \leq \, &\Pbf_w(H_B < \bar{H}_{\partial B^\eps}) \leq c_2' \tfrac{2}{\eta} N^{\jmath-1}  + \grdO(N^{-1}) \, .
			\end{split}
		\end{equation}
	\end{customlemma}
		
		\begin{proof}
			Proposition 1.5.10 in \cite{lawler2013intersections} gives the following estimate: let $x \in B^\eps \setminus B$, then
			\begin{equation}
				\Pbf_x(H_B < \bar{H}_{\partial B^\eps}) = \frac{|x|^{2-d} - [(\alpha + \eps)N]^{2-d} + \grdO(N^{1-d})}{[\alpha N]^{2-d} - [(\alpha + \eps)N]^{2-d}} \, .
			\end{equation}
			Inject $x = z$ where $|z - x_0^N| - \alpha N \in [\eta, \tfrac{1}{\eta}] N^\iota$, then simplifying by $N^{2-d}$ yields 
	\begin{multline*}
	\frac{(\alpha + \tfrac{1}{\eta} N^{\iota - 1})^{2-d} - (\alpha + \eps)^{2-d} + \grdO(N^{-1})}{\alpha^{2-d} - (\alpha + \eps)^{2-d}}  \leq  \Pbf_x(H_B < \bar{H}_{\partial B^\eps}) \\
	\leq \frac{(\alpha + \eta N^{\iota - 1})^{2-d} - (\alpha + \eps)^{2-d} + \grdO(N^{-1})}{\alpha^{2-d} - (\alpha + \eps)^{2-d}} \, .
	\end{multline*}
			Since $\iota - 1 < 0$, we easily rewrite
			\[ 
			\begin{split} 
			(\alpha + \eta N^{\iota - 1})^{2-d} = \left[ \frac{\alpha^{-1}}{1 + \tfrac{\eta}{\alpha} N^{\iota - 1}} \right]^{d-2}& = \alpha^{2-d} \left[ 1 - \tfrac{\eta}{\alpha} N^{\iota - 1}(1 + \bar{o}(1)) \right]^{d-2} \\
			& = \alpha^{2-d} - (d-2)\alpha^{2-d} \tfrac{\eta}{\alpha} N^{\iota - 1}(1 + \bar{o}(1)) \, .
			\end{split} \]
			Hence, we get the upper bound
			\[ \Pbf_x(H_B < \bar{H}_{\partial B^\eps}) \leq 1 - \frac{(d-2)\alpha^{2-d} \tfrac{\eta}{\alpha}}{\alpha^{2-d} - (\alpha + \eps)^{2-d}} N^{\iota - 1}(1 + \bar{o}(1)) + \grdO(N^{-1}) \, , \]
from which we deduce the first upper bound in \eqref{eq:proba-atteindre-avant-SRW-appendix}. The lower bound is deduced with the same expansion.
		
Similarly, if $|w - x_0^N| - (\alpha + \eps) N \in -[\eta, \tfrac{1}{\eta}] N^\jmath$, then		
\[ 
|w|^{2-d} - [(\alpha + \eps)N]^{2-d} \leq (\alpha + \eps - \eta N^{\jmath - 1})^{2-d} = (\alpha + \eps)^{2-d} \left[ 1 + \frac{d-2}{\alpha + \eps} \eta N^{\jmath - 1}(1+\bar{o}(1)) \right] 
\]
and similarly for the lower bound, hence the lemma.
\end{proof}
		
	\begin{customclaim}{\ref{claim:esp-tps-retour-bord-anneau}}
	For all $\delta \in (0,\eps)$, there is a $c_\delta > 0$ such that
			\[
			\sup_{z \in \partial B} \probaRW{z}{\bar{H}_B \wedge H_{\partial B^\delta} > (\delta N)^2} \leq \frac{c_\delta}{N} \, , \qquad \sup_{w \in \partial B^\delta} \probaRW{w}{H_B \wedge \bar{H}_{\partial B^\delta} > (\delta N)^2 } \leq \frac{c_\delta}{N} \,.
			\]
	\end{customclaim}
	
	\begin{proof}
		For this proof, we can consider the ball $B$ to be centered at $x_0^N = 0$ since we are only interested in the simple random walk.
		Write $\tau_\delta \defeq \bar{H}_B \wedge H_{\partial B^\delta}$ and take $z \in \partial B$. Consider $x$ such that $x \not\in B$ and $x \sim z$. Using the martingale $|X_{t \wedge {\tau_\delta}}|^2 - t \wedge \tau_\delta$, we get
		\begin{equation}
			|x|^2 = \espRW{x}{|X_{t \wedge \tau_\delta}|^2 - t \wedge \tau_\delta} \xrightarrow[]{t \uparrow \tau_\delta} \espRW{x}{|X_{\tau_\delta}|^2 - {\tau_\delta}} \, ,
		\end{equation}
		where we used dominated and monotonous convergence. Splitting the last expectation according to whether $X_{\tau_\delta} \in B$ or not,	we have
		\[ |x|^2 = \probaRW{x}{X_{\tau_\delta} \in B} \espRW{x}{|X_{\tau_\delta}|^2 - \tau_\delta \, \big| \, X_{\tau_\delta} \in B} + \probaRW{x}{X_{\tau_\delta} \not\in B} \espRW{x}{|X_{\tau_\delta}|^2 - \tau_\delta \, \big| \, X_{\tau_\delta} \not\in B} \, . \]
		Rearranging the terms,
		\[ \begin{split}
			\probaRW{x}{X_{\tau_\delta} \in B} \espRW{z}{\tau_\delta \, \big| \, X_{\tau_\delta} \in B} = &\, \probaRW{x}{X_{\tau_\delta} \in B} \espRW{z}{|X_{\tau_\delta}|^2 - |x|^2 \, \big| \, X_{\tau_\delta} \in B} \\
			&+ \probaRW{x}{X_{\tau_\delta} \not\in B} \espRW{z}{|X_{\tau_\delta}|^2 - \tau_\delta  - |x|^2 \, \big| \, X_{\tau_\delta} \not\in B} \, .
		\end{split} \]
		Since we took $x \not\in B$, on the event $\mathset{X_{\tau_\delta} \in B}$ we have $|X_{\tau_\delta}|^2 \leq |x|^2$, and thus we can bound $\espRW{z}{|X_{\tau_\delta}|^2 - |x|^2 \, | \, X_{\tau_\delta} \in B} \leq 0$. Also using that $\tau_\delta \geq 0$, we get
		\[ \probaRW{x}{X_{\tau_\delta} \in B} \espRW{z}{\tau_\delta \, \big| \, X_{\tau_\delta} \in B} \leq \probaRW{x}{X_{\tau_\delta} \not\in B} \espRW{z}{|X_{\tau_\delta}|^2 - |x|^2 \, \big| \, X_{\tau_\delta} \not\in B} \, . \]
		Note that on $\mathset{X_{\tau_\delta} \not\in B}$, we have $|X_{\tau_\delta}|^2 - |x|^2 \leq 2 \alpha \delta N^2$. Since $\probaRW{x}{X_{\tau_\delta} \in B} \geq \probaRW{x}{X_1 = z} = 1/2d$, we get
		\begin{equation}\label{eq:esp-cond-sortie-anneau}
			\espRW{x}{\tau_\delta \, | \, X_{\tau_\delta} \in B} \leq 4 \alpha d \delta N^2 \probaRW{z}{X_{\tau_\delta} \not\in B} \leq c_\delta' N \, ,
		\end{equation}
		where we have used \cite[Lem. 6.4.3]{lawlerRandomWalkModern2010} for the last inequality.
		By the Markov 
		%QB property 
		inequality and decomposing on the possible sites $x$, we get the claim for $z \in \partial B$.
		With a similar proof, we also get the result in \eqref{claim:esp-tps-retour-bord-anneau} for $w \in \partial B^\delta$.
	\end{proof}	
	
	\begin{remark}\label{rem:adapt-pt-depart-Delta}
		To get a version of this claim where the starting point $z$ is not in $\partial B$ but instead is such that $|z-x_0^N| - \alpha N \in [\eta, \tfrac{1}{\eta}] N^\gamma$ for some $\eta > 0$, we see that the proof of Claim \ref{claim:esp-tps-retour-bord-anneau} is still true up until \eqref{eq:esp-cond-sortie-anneau}, in which the bound $\probaRW{z}{X_{\tau_\delta} \not\in B} \leq c_\delta' /N$ must be replaced by $\probaRW{z}{X_{\tau_\delta} \not\in B} \leq c_\delta' N^{\gamma-1}$. Therefore, we get a bound $\espRW{z}{\tau_\delta \, | \, X_{\tau_\delta} \in B} \leq c_\delta'' N^{\gamma}$, which by the Markov inequality yields the bound $N^{\gamma-1}$ that was announced in the proof of Lemma \ref{lem:toucher-B-avant-Beps-start-Delta}.
	\end{remark}
		
		\section{On eigenvalues and eigenvectors}\label{sec:appendix:fonctions-propres}
		
			The main goal of this section is to summarize results on $\phi_N$ which are proven in \cite{vecteurpropre} (or easily deduced from this paper). We will also prove some estimates that we use in the proof of Proposition \ref{prop:nb-excursions-tN}.

	\paragraph*{Some preliminaries.}			
		
		Consider a compact set $D \subset \RR^d$ and a mesh parameter $h > 0$. On $D^{(h)} \defeq D \cap h \ZZ^d$, we define the discrete Laplacian so that for any $x \in D^{(h)} \setminus \partial D_h$, any $v^{(h)} : D^{(h)} \longrightarrow \RR_+$,
		\[ \Delta^{(h)} v^{(h)}(x) \defeq h^{-2} \sum_{e \in \ZZ^d, |e|=1} \Big( v^{(h)}(x+he) - v^{(h)}(x) \Big) =  \frac{2d}{h^2} \Delta_d \tilde{v}^{(h)}(\tfrac{1}{h} x) \, , \]
		where $\tilde{v}^{(h)}(x) = v^{(h)}(hx)$. On $D^{(h)}$, we may consider the discrete Dirichlet problem
		\[ \begin{cases} &\Delta^{(h)} v - \mu^{(h)} v = 0 \text{ on } D^{(h)} \setminus \partial D^{(h)} \, ; \\ &\forall x \in \partial D^{(h)}, v(x) = 0 \, . \end{cases} \]
		Write $(\mu^{(h)}_k,\phi^{(h)}_k)_{k\geq 1}$ for the (ordered) eigenvalues and $\ell^1$-normalized eigenvectors that solve this discrete Dirichlet problem. 
		We also consider the Dirichlet problem on $D$, with the usual Laplace--Beltrami operator $\Delta$, with $L^1$-normalized (ordered) solutions $(\mu_k, \varphi_k)_{k\geq 1}$.
		The Rayleigh-Ritz method gives that $\mu^{(h)}_k = \mu_k (1 + \bar{o}(1))$ as $h \to 0$, for a wide range of domains $D$ (see \cite{pjm/1103040107} for more details).
		
		Consider now the transition matrix $P_N$ of the simple random walk killed on the boundary of $D_N = (N D) \cap \ZZ^d$. Consider now $x = x_* N \in D_N \setminus \partial D_N$ and recall the previous notation, then
		\begin{equation}\label{eq:reecriture-PN}
		\begin{split}
			P_N \tilde{v}^{(1/N)}(x) = \frac{1}{2d} \sum_{y \sim x} \tilde{v}^{(1/N)}(x)(y) & = \tilde{v}^{(1/N)}(x)(x) + \frac{1}{2d} \sum_{|e| = 1} \Big( \tilde{v}^{(1/N)}(x + e) - \tilde{v}^{(1/N)}(x) \Big)  \\
			& = \Big( \mathrm{Id}_N + \frac{1}{2dN^2} \Delta^{(1/N)} \Big) v^{(1/N)}(x_*) \, .
			\end{split}
		\end{equation}
		Since every function $f : D_N \longrightarrow \RR$ can be expressed as some $\tilde{v}^{(1/N)}$, this proves that $P_N$ and $\Delta^{(1/N)}$ have the same eigenvectors.
		Write $(\lambda_N^k,\phi_N^k)$ for the (ordered) eigenvalues and eigenvectors of this transition matrix; we also write $\lambda_N \defeq \lambda_N^1$ and $\phi_N \defeq \phi_N^1$. Using \eqref{eq:reecriture-PN} first, and then $\mu^{(1/N)}_k = \mu_k (1 + \bar{o}(1))$ as mentioned above, we get the following asymptotics for $\lambda_N^{k}$:
		\begin{equation}\label{eq:vp-blowup}
			\lambda_N^k = 1 - \frac{\mu_k^{(1/N)}}{2d N^2} = 1 - \frac{\mu_k}{2d N^2} + \bar{o}(N^{-2}) \, .
		\end{equation} 
		
	Additionally to the convergence of the eigenvalues, there is also a convergence for the first eigenvector $\phi_N$. Recall that $\varphi_1$ is the first eigenfunction of the Laplace operator $\Delta$ on the original domain $D$, which is considered to be $L^1$-normalized. We define on $D_N$ the function $\varphi_N = \varphi_1(\cdot / N)$.
	Recall that $\phi_N$ is the first eigenvector of the matrix $P_N$, that we normalize by $\| \phi_N \|_1 = N^d$. Then, \cite[Corollary 1.10]{vecteurpropre} states that $\phi_N$ is uniformly close to $\varphi_N$. More precisely,
	there is a constant $c > 0$ (that only depends on~$D$) such that for all $N$ large enough,
	\begin{equation}\label{eq:cv-phiN-supremum}
		\sup_{x \in D_N} \big| \phi_N(x) - \varphi_N(x) \big| \leq c  N^{-1/2(d+1)} \, .
	\end{equation}
	In the bulk, this can be refined to have a convergence of the ratios to $1$. For any $\eta > 0$, define $D_N^{-\eta}\defeq \{x\in D_N, d(x,\partial D_N)  > \eta N\}$.
	Then we have
	\begin{equation}\label{eq:cv-phi_N-ratio}
		\sup_{x\in D_N^{-\eta}} \bigg| \frac{\phi_N(x)}{\varphi_N(x)} -1 \bigg| \xrightarrow{N\to\infty} 0\,.
	\end{equation}

	\paragraph*{Some other useful estimates.} 
		We complement the results by some observations that are used in the proof of Proposition \ref{prop:nb-excursions-tN}. First of all, we state a result that tells that, provided that the time is large enough, the first eigencouple gives the probability for the SRW to stay in~$D_N$ when starting from a given point in~$D_N$. The following statement can be easily proven with the methods of the proof of \cite[Lemma 3.10]{dingDistributionRandomWalk2021a} and the fact that $\| \phi_N \| \asymp \| \phi_N\|_2^2$ (which is a consequence of \eqref{eq:phiN-borne}).
		
		\begin{proposition}\label{prop:equivalent-proba-sortie-eigen}
			Let $\tau$ be the exit time of $D_N$, then for any sequence $t_N \gg N^2 \log N$ and any $x \in D_N$, we have
			\begin{equation}
				\Pbf_x(\tau > t_N) = \phi_N(x) (\lambda_N)^{t_N} \frac{\| \phi_N \|}{\| \phi_N\|_2^2}(1 + \grdO(N^{-2})) \, .
			\end{equation}
			In particular, there is a positive constant $c > 0$ such that for any $x \in D_N$ with $N$ large enough, we have
			\begin{equation}\label{rem:encadrement-proba-sortie-eigen}
				c \phi_N(x) (\lambda_N)^{t_N} \leq \Pbf_x(\tau > t_N) \leq \tfrac{1}{c} \phi_N(x) (\lambda_N)^{t_N} \, .
			\end{equation}
		\end{proposition}
		
		We also give the following lemma which, combined with \eqref{eq:vp-blowup}, ensures that removing a macroscopic ball strictly lowers the first eigenvalue of the domain. By Proposition \ref{prop:equivalent-proba-sortie-eigen}, this implies that staying in $D_N$ a time $t_N \gg N^2 \log N$ while avoiding this macroscopic ball is exponentially improbable.
		
		\begin{lemma}\label{lem:vp-retire-boule}
			Let $x \in D$ and $r > 0$ be such that $B_r^x = B(x,r) \subset D$. Then, we have $\lambda(D\setminus B_r^x) > \lambda$.
		\end{lemma}
		
		\begin{proof}
			By the monotonicity of the eigenvalue $\lambda(D)$ with respect to the domain, it suffices to prove this when $r > 0$ is small. By \cite[Theorem 1.4.1]{Henrot2006} we have
			\[ \lambda(D \setminus B_r^x) = \lambda(D) + s_d r^{d-2} \varphi_1^2(x) + \bar{o}(r^{d-2}) \, , \]
			with $s_d$ the volume of the unit sphere of $\RR^d$ (with respect to the Lebesgue measure on $\RR^{d-1}$) and $\varphi_1$ the first Dirichlet eigenfunction on $D$, which is positive on the interior of $D$. The lemma immediately follows.
		\end{proof}

		\paragraph{Acknowledgements} The author warmly thanks his PhD advisors Quentin Berger \& Julien Poisat for their continued support. The author also thanks Alberto Chiarini for fruitful discussions and additional literature review, as well as an anonymous reviewer for pointing out a crucial mistake in a previous version.

	\end{appendix}

	\bibliographystyle{alpha}
	\bibliography{main}
	
\end{document}